\documentclass[10pt,reqno]{amsart}

\usepackage{graphicx}
\graphicspath{ {./images/} }
\usepackage{xargs}
\usepackage{geometry}
\usepackage{fancyhdr}
\usepackage{lastpage}
\usepackage{extramarks} 
\usepackage[most]{tcolorbox}
\usepackage{xcolor} 
\usepackage[hidelinks]{hyperref} 
\usepackage[permil]{overpic}
\usepackage{pict2e} 
\usepackage{listings}
\usepackage{amssymb}
\usepackage{amsthm}
\theoremstyle{plain}
\usepackage[
backend=biber,
style=alphabetic,
sorting=ynt
]{biblatex}
\usepackage{xargs}
\usepackage{float}
\usepackage{lipsum}
\usepackage{geometry}
\usepackage{fancyhdr}
\usepackage{lastpage}
\usepackage{extramarks} 
\usepackage[most]{tcolorbox} 
\usepackage{xcolor}
\usepackage[hidelinks]{hyperref} 
\usepackage[permil]{overpic}
\usepackage{pict2e} 
\usepackage{listings}
\usepackage{tikz}

\newtheorem*{theorem*}{Theorem}

\newtheorem{definition}{Definition}[section]
\theoremstyle{assumption}

\theoremstyle{notation}

\theoremstyle{claim}

\theoremstyle{conj}

\newtheorem{remark}{\textbf{Remark}}[section]
\theoremstyle{theorem}
\newtheorem{theorem}{Theorem}[section]
\newtheorem{conjecture}{Conjecture}[section]
\theoremstyle{proposition}
\newtheorem{proposition}{Proposition}[section]
\theoremstyle{corollary}
\newtheorem{corollary}{Corollary}[section]
\newtheorem{lemma}{Lemma}[section]
\newcommand{\Mod}[1]{\ (\mathrm{mod}\ #1)}
\newtheorem{example}[theorem]{Example}

\title{Kneading the Lorenz Attractor}

\author{Łukasz Cholewa$^*$, Eran Igra$^{**}$}
\address{*Krakow University of Economics}
\email{cholewal@uek.krakow.pl}
\address{**Shanghai Institute for Mathematics and Interdisciplinary Sciences}
\email{eranigra@simis.cn}
\begin{document}

\begin{abstract}
A Lorenz map $f:[0,1]\to[0,1]$ is a piecewise continuous map, modeled after an idealized version of the Lorenz attractor. In this paper we settle the following question - how much of the dynamics of the Lorenz attractor can be modeled by such one-dimensional model? In this paper we will prove there exist open regions in the parameter space of the Lorenz system where one can canonically reduce the dynamics of the Lorenz attractor into those of a symmetric Lorenz map $F_\beta$ with a constant slope $\beta\in(1,2]$. As we will show, not only the map $F_\beta$ encodes many of the essential features of the Lorenz attractor, it also governs many of its bifurcations. As such, our results correlate closely with the results of numerical studies, and possibly explain the bifurcation phenomena observed in the Lorenz attractor.
\end{abstract}

\maketitle
\section{Introduction}

Given three positive parameters  $\sigma,\rho,\mu\in\mathbf{R}$, the Lorenz system is defined as the flow generated by the following three-dimensional system of differential equations:
\begin{equation*} 
\begin{cases}
\dot{x} = \sigma(y-x) \\
 \dot{y} = x(\rho-z)-y\\
 \dot{z}=xy-\mu z
\end{cases}
\end{equation*}
This system, one of the hallmarks of chaotic dynamics, was first discovered in 1963 by E.N. Lorenz (see \cite{Lo}). In particular, Lorenz observed that at the parameter values $(\sigma,\rho,\mu)=(10,28,\frac{8}{3})$ the flow generates a chaotic butterfly attractor, as in Fig.\ref{untitled} (for a survey on the origins and development of the Lorenz model through the years, see \cite{bowen}).

\begin{figure}[h]
\centering
\begin{overpic}[width=0.45\textwidth]{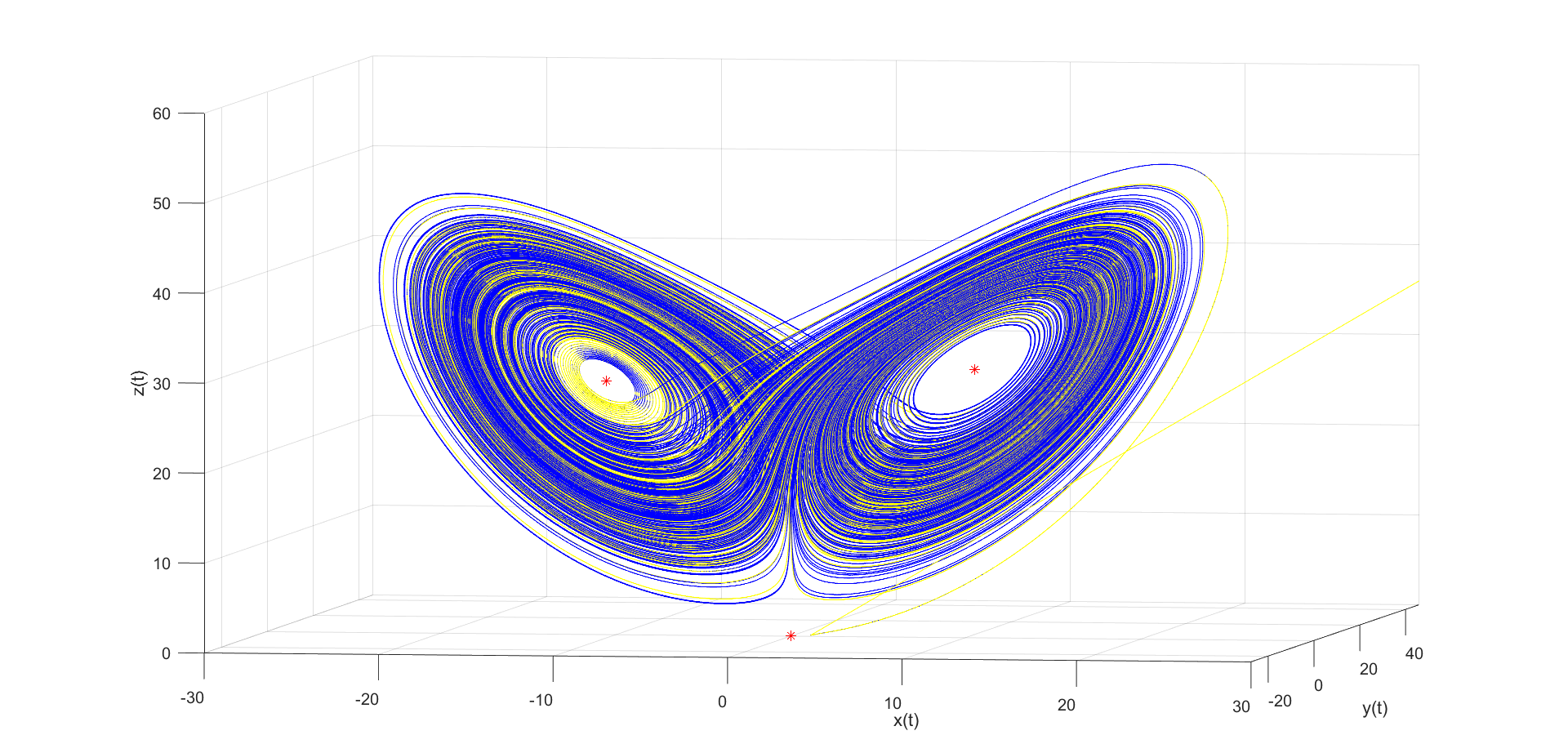}
\end{overpic}
\caption{\textit{The Lorenz attractor at $(\sigma,\rho,\beta)=(10,28,\frac{8}{3})$.}}
\label{untitled}
\end{figure}

Since its introduction, the mechanisms behind the complex dynamics of the Lorenz attractor were the focus of many studies - both analytic and numerical. Among the analytic methods one famous example is that of the geometric Lorenz model, originally introduced in \cite{SAB} and \cite{guck}. In brief, the geometric Lorenz model is a toy model for chaos defined geometrically, inspired by the numerically observed features of the Lorenz attractor. Its topological dynamics are well-understood - for example, they can be completely described in terms of the knots realized as periodic orbits on the attractor (see, for example, \cite{BW} and \cite{Hol}). Similarly, the statistical properties of the said attractor are also well understood - see, for example, \cite{HM}, among many others. One particular feature of this geometric model is that its dynamics can be encoded in terms of lower-dimensional, discrete-time models. In detail, in \cite{Wil} it was proven the first-return map for the geometric Lorenz model can be reduced to a piecewise discontinuous interval map, given by a concrete formula.

In addition to the analytic approaches, the dynamics of the Lorenz attractor were also studied extensively using numerical tools. For example, the domains of existence of the attractor and its bifurcations were thoroughly investigated in \cite{SNS}, \cite{KOC}, \cite{KO1}, \cite{YY}, \cite{KY}, \cite{BS} and \cite{BSS}, among many others (for a comprehensive survey, see \cite{Sparbook}). Moreover, the existence of complex dynamics in the Lorenz attractor was first proven in \cite{MM}, using rigorous numerical methods. Following that, in \cite{TUC} it was proven (again with the aid of rigorous numerical methods) that the dynamics of the Lorenz attractor at $(\sigma,\rho,\mu)=(10,28,\frac{8}{3})$ are conjugate to those of the geometric Lorenz model - thus solving Smale's 14th Problem. Recently, inspired by the numerical studies, similar results to those of \cite{TUC} and \cite{MM} were obtained in \cite{Pi}. However, unlike the results of \cite{MM} and \cite{TUC}, the results of \cite{Pi} were proven using purely analytic methods.

As mentioned above, one approach to study the geometric Lorenz model is by reducing its first-return map to a one-dimensional dynamical system, called a "Lorenz map". Briefly speaking and avoiding the technicalities (for now), Lorenz maps are piecewise monotone interval maps with a single discontinuity (see Fig.\ref{fig:Lorenz_map}). As proven in \cite{Wil}, their dynamics are semi-conjugate to those of the geometric Lorenz model. This connection inspired the development of a rich theory describing the dynamics of Lorenz maps (see, for example, \cite{Glen}, \cite{HubSpar}, \cite{GlenHall}, \cite{GS}, \cite{Alseda}, \cite{Pierre} and \cite{Winckler} - among others). Moreover, as far as the actual Lorenz attractor is concerned, it was observed numerically in both \cite{BS} and \cite{MK} that the dynamics and bifurcations of the Lorenz attractor correlate with those of the Lorenz maps. That being said, while the connection of Lorenz maps to the geometric Lorenz model is clear, their connection with the dynamics of the original Lorenz attractor is far from understood.

It is precisely this gap we address in this paper. Inspired by the ideas of \cite{MK} and \cite{BS} and building on the results of \cite{Pi}, in this paper we rigorously prove the dynamics and bifurcations of the original Lorenz attractor are essentially those of Lorenz maps - thus rigorously establishing their role as idealized, toy models for the evolution of chaos on the Lorenz attractor. Our first major result is the following (see Th.\ref{reduct} and Cor.\ref{reduct2} in Sect.\ref{reductionsect}):
\begin{theorem}
    \label{first} There is an open set of parameters $P$ s.t. for every $(\sigma,\rho,\mu)\in P$ there exists a cross-section $S$, and a continuous first-return map $\psi:S\to S$ for the Lorenz attractor corresponding to $(\sigma,\rho,\mu)$ whose dynamics are essentially one-dimensional. In detail, the following holds:
    
    \begin{itemize}
        \item There exists a Lorenz map $F_\beta$ depending on a parameter $\beta\in(1,2]$ satisfying the following:
        \begin{enumerate}
            \item The map $F_\beta$ has constant slope $\beta$, where $\beta$ depends only on the parameters $(\sigma,\rho,\mu)$.
            \item $F_\beta$ has a discontinuity point at $\frac{1}{2}$, independently of $\beta$ (and of $(\sigma,\rho,\mu)$.
        \end{enumerate}
        
         Let $I_\beta$ denote the maximal invariant set of $F_\beta$ in $[0,1]\setminus\{\frac{1}{2}\}$. Then, there exists an invariant set $I\subseteq S$ and a continuous, surjective $\pi:I\to I_\beta$ s.t. $\pi\circ \psi=F_\beta\circ\pi$
        \item If $x\in I_\beta$ is periodic of minimal period $n$, then $\pi^{-1}(x)$ includes at least one periodic orbit for $\psi$ of minimal period $n$.
    \end{itemize}
    
\end{theorem}
The idea of the proof is relatively simple. We begin by considering the parameter space $\{\rho\geq\max\{1,\frac{(\sigma+1)^2}{4\sigma}\}\}$ originally considered in \cite{Pi}. As proven at Proposition $1$ in \cite{Pi}, for every parameter in the said range the first-return map is defined and continuous. By carefully homotoping the said first return map we collapse its dynamics to $F_\beta$, whose properties listed above will be evident from the proof. Put simply, our argument proves there is a certain "dynamical core" of the attractor where the dynamics are essentially one-dimensional. 

Capitalizing on this idea, following the proof Th.\ref{first}, we proceed by studying the dynamics of the family $\{F_\beta\}_{\beta\in(1,2]}$, and pull them back to the Lorenz attractor. In particular, we will study the topological dynamics, kneading, and renormalization properties of the maps $F_\beta$ - see subsections \ref{subsec:symbolic_dynamics}, \ref{subsec:topmeasure}, and \ref{renorsec}. As we shall see, the family $\{F_\beta\}_{\beta\in(1,2]}$ forms a special class of Lorenz maps that we call the \textbf{symmetric $\beta$--transformations} (see, e.g., \cite{ParryL}). One immediate result that arises from our study of such maps is the following:
\begin{theorem}
    \label{second} Let $P$ be the parameter space from Th.\ref{first}. Then, there exists an open set of parameters in $O\subseteq P$ s.t. the following holds:
    
    \begin{itemize}
        \item For every $v\in O$ the corresponding Lorenz attractor includes infinitely many periodic orbits. 
        \item The said periodic orbits can only be destroyed via a homoclinic bifurcation as we vary $v$ in $O$. 
        \item Finally, given a parameter $p\in P$ where the Lorenz system generates a structurally unstable heteroclinic trefoil knot as in Fig.\ref{trefoil}, then $p$ is the accumulation point of homoclinic bifurcation sets in $P$.
    \end{itemize}

\end{theorem}
For a proof, see both Th.\ref{expl} and Cor.\ref{mixingattractor}. At this point we remark Th.\ref{second} correlates nicely with the numerical studies of the Lorenz system. To explain why, recall the notion of \textbf{$T$--Points} for the Lorenz system. Without going into the exact details and avoiding the precise definition, a $T$--point is a parameter $p\in P$ where the Lorenz system generates a structurally unstable heteroclinic trefoil knot (the existence of parameters in $P$ where the Lorenz system generates heteroclinic trefoil knots was proven in Th.1 in \cite{Pi}). As observed numerically, such parameters are often the accumulation point of homoclinic bifurcation sets (see, for example, \cite{BSS}). In light of the above, the third assertion of Th.\ref{second} possibly serves as an analytic explanation behind the complex bifurcation phenomena observed around some $T$--points.

Following that, inspired by the ideas of \cite{GS} we show how our one-dimensional reduction can be used to analyze the topology of the Lorenz attractor. To state that theorem, we first note the definition of renormalization in the special case of maps $F_\beta\colon[0,1]\to[0,1]$, $\beta\in(1,2]$ boils down to the following one: a map $F_{\beta}$ is \textbf{renormalizable} if there is a proper sub-interval $[u, v]$ of $[0,1]$, centered at $\frac{1}{2}$, and an integer $k > 1$ such that the map $F_\beta^k$ restricted to the interval $[u,v]$ is conjugate to some other $F_{\beta'}$, $\beta'\in(1,2]$. The interval $[u, v]$ is then called the \textbf{renormalization interval} and is uniquely determined by the number $k$ (for the general definition and an intuitive example, see Def.\ref{defn:renormalization} and Fig.\ref{fig:renormalization}).

Similarly, recall a \textbf{Template} for a chaotic attractor is a branched surface which encodes all the knot types realized as periodic orbits on the attractor (for the precise definition, see Def.\ref{deftemp} and the discussion immediately following it). As under certain conditions Templates can be embedded inside three-dimensional chaotic attractors (i.e., the Birman-Williams Theorem - see \cite{BW}), one could think of Templates as the "skeleton" of such complex invariant sets - both topologically and dynamically. Combining Template Theory with Renormalization Theory, we prove the following result in Sect.\ref{renorsec}:
\begin{theorem}
    \label{third} Assume there exists a parameter $v\in P$ s.t. its corresponding Lorenz map $F_\beta$ w.r.t. Th.\ref{first} has a renormalization $G=F_\beta^k|_{[u,v]}$ that is conjugate to the doubling map $2x\Mod{1}$, $x\in[0,1]$. Then, $\beta=\sqrt[2^i]{2}$ for some $i\in\mathbb{N}$, and moreover, there exists a Template $\tau$ s.t. every knot type encoded by $\tau$ is realized as a periodic orbit on the Lorenz attractor corresponding to $v$. 
\end{theorem}

The proof of Th.\ref{third} is given in Th.\ref{templateth} and Lemma \ref{lem:doubling_conjugate}. It is achieved by direct analysis of the flow at such parameters, combined with the application of two-dimensional tools from \cite{Han}. In fact, we prove more than that - as our analysis is mostly topological, as a bonus it implies the Template $\tau$ actually belongs to a relatively famous class of relatively well-understood Templates: Lorenz Templates. Consequently, it follows all the knots $\tau$ encodes are prime knots (see Cor.\ref{prime}). For more details on Lorenz Templates and the topology of the knots they encode, see the surveys in both \cite{Pi2} and \cite{Deh}.

This paper is organized as follows. We begin with Section \ref{reductionsect}, where we survey several basic facts on the Lorenz system, followed by a rigorous reduction of the Lorenz attractor to the $\beta$--transformations as stated in Th.\ref{first}. Following that, in Sections \ref{subsec:symbolic_dynamics}, \ref{subsec:topmeasure} and \ref{renorsec} we survey and study the dynamics of the $\beta$--transformations, and apply them to study the Lorenz attractor. It is in these sections that we prove Th.\ref{second} and Th.\ref{third}. We conclude this paper by discussing the wider context of our results, and how they can possibly be extended to more general $C^k$ perturbations of the Lorenz attractor - as well as to a possibly wider class of three-dimensional flows.

Before we begin, we remark that even though it may not be clear from the text below, many of our ideas were inspired by the Chaotic Hypothesis, originally introduced in \cite{gal}. Briefly speaking, the Chaotic Hypothesis conjectures that for practical purposes, every chaotic attractor is hyperbolic. As will be clear from the proof of Th.\ref{first}, one could interpret our results as saying that for parameter values $(\sigma,\rho,\mu)\in P$, the dynamics on the corresponding Lorenz attractor are complex at least like a hyperbolic dynamical system that can be easily constructed from the map $F_\beta$ - for the precise details, see the map $h_v$ proof of Th.\ref{reduct}. As such, this paper - and in particular Th.\ref{first} and Th.\ref{third} - can be interpreted as a step towards the proof of the Chaotic Hypothesis for three-dimensional flows.

\section{A one-dimensional reduction for the Lorenz Attractor:}\label{reductionsect}
From now on, given three positive parameters  $\sigma,\rho,\mu\in\mathbf{R}$, by the Lorenz system we will always mean the flow generated by the following system of differential equations:
\begin{equation} \label{Vect}
\begin{cases}
\dot{x} = \sigma(y-x) \\
 \dot{y} = x(\rho-z)-y\\
 \dot{z}=xy-\mu z
\end{cases}
\end{equation}

\begin{figure}[h]
\centering
\begin{overpic}[width=0.45\textwidth]{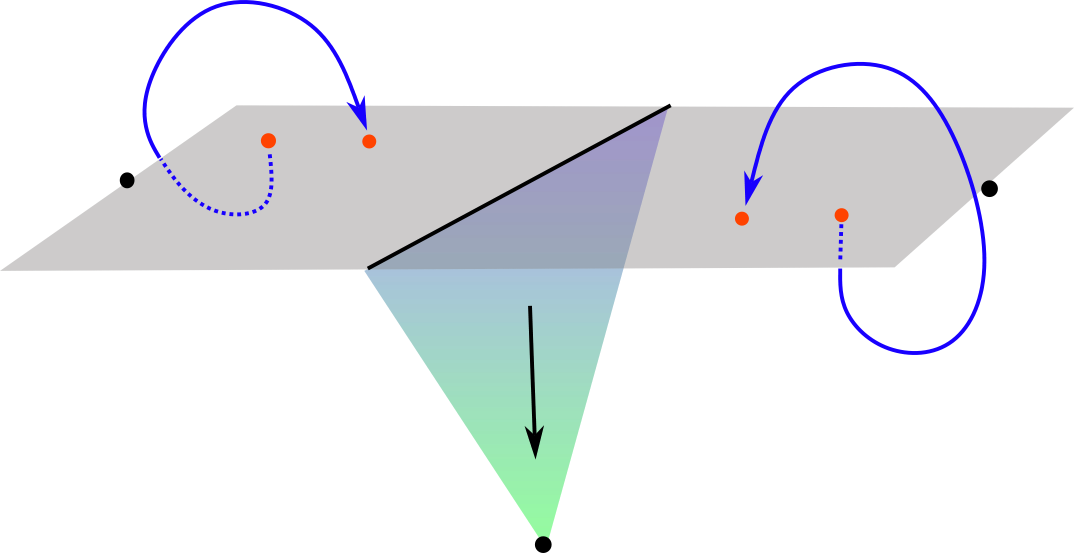}
\put(70,200){$R_0$}
\put(900,450){$R_1$}
\put(390,330){ $W$}
\put(450,0){ $0$}
\put(30,340){$p_0$}
\put(950,330){$p_1$}
\end{overpic}
\caption{\textit{The cross-section $R$ at parameters $p\in P$, along with the action of the flow. The sub-rectangles $R_{0}$ and $R_{1}$ correspond to the components of $R\setminus W$.}}
\label{cross}
\end{figure}
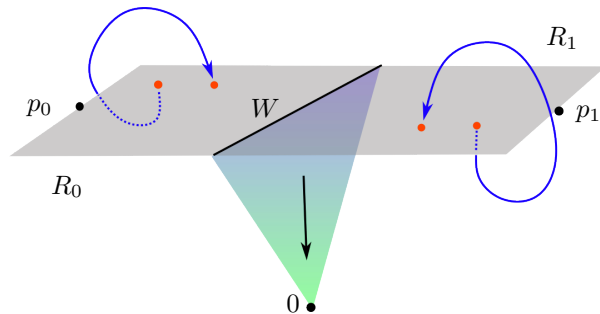

We will often denote the vector field corresponding to the parameters $\sigma,\rho, \mu$ by $L_{\sigma,\rho,\mu}$ - and when $(\sigma,\rho\mu)$ are implicit, we denote the said parameter by $v$, and the corresponding vector field by $L_v$. By direct computation, when $\rho>1,\mu>0$ the vector field $L_{\sigma,\rho,\mu}$ satisfies the following properties (see \cite{Lo}):
\begin{itemize}
    \item The Lorenz system is symmetric - i.e., $L_{\sigma,\rho,\mu}(x,y,z)=-L_{\sigma,\rho,\mu}(-x,-y,z)$.
    \item The origin, $0$, is a real saddle with a two-dimensional stable manifold, and a one-dimensional unstable manifold. We will often denote this two dimensional invariant manifold by $W^s(0)$.
    \item  There also exist $p_1,p_0$, two fixed points given by the formulas $(\pm\sqrt{\mu(\rho-1)},\pm\sqrt{\mu(\rho-1)},\rho-1)$. 
    \item There exists an ellipsoid $V$ s.t. the trajectory of every initial condition $s\in\mathbb{R}^3$ eventually enters $V$, and never escapes it (for the details, see either \cite{Lo} or Appendix $C$ in \cite{Sparbook}).
\end{itemize}

We now recall the following collection of results on the global dynamics of the Lorenz system, proven in Prop.2.1, Lemma 2.2 and Th.1.1 in \cite{Pi}:
\begin{theorem}
    \label{tali} There exists an open, three-dimensional set of positive parameters $(\sigma,\rho,\beta)$ defined by the set $P=\{\rho>max\{1,\frac{(\sigma+1)^2}{4\sigma}\}\}$ s.t. for all $v=(\sigma,\rho,\beta)\in P$ the following holds:
    \begin{enumerate}
        \item The vector field $L_v$ has a cross-section $R$ such that:
        \begin{itemize}
            \item $R$ is a topological rectangle, and the vector field $L_v$ is transverse to $R$ at its interior. Moreover, the fixed points $p_1,p_0$ are on its boundary, while $0\not\in R$ (see the illustration in Fig.\ref{cross}).
            \item The invariant manifold $W^s(0)$ intersects $R$ in an arc $W$ homeomorphic to a curve. As such, $W$ partitions $R$ into two sub-rectangles, $R_1$ and $R_0$ which include $p_1$ and $p_0$ on their boundary (respectively) - see the illustration in Fig.\ref{cross}. 
            \item For all $v\in P$, the first-return map of $L_v$, $\psi_v:R_0\cup R_1\to R$, is well defined and continuous. Moreover, $W$ forms a discontinuity curve for the first-return map  (see the illustration in Fig.\ref{splitting11}).
            \item Any periodic orbit for $L_v$ intersects $R_0\cup R_1$ transversely at least once, and given $s\in\partial R$ s.t. $\psi_v(s)=s$, then $s$ is a fixed point. Moreover, the dynamics of $\psi_v$ on its invariant set in $R_0\cup R_1$ can be factored to some subshift on two symbols.
        \end{itemize}
        \item There exists a collection of parameters $T\subseteq P$, s.t. for all $p\in T$, $L_p$ generates a pair of heteroclinic trajectories connecting $0$ and $p_0,p_1$, as illustrated in Fig.\ref{trefoil}. At such parameters the dynamics of $\psi_p$ on its invariant set include infinitely many periodic orbits, and the first-return map is semi-conjugate on its invariant set in $R_0\cup R_1$ to the double-sided shift $\sigma:\{0,1\}^\mathbb{Z}\to\{0,1\}^\mathbb{Z}$. We refer to such parameters as \textbf{trefoil parameters}.
    \end{enumerate}
\end{theorem}
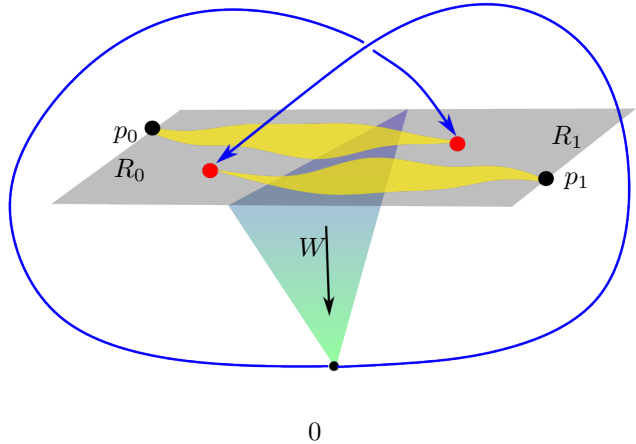
\begin{figure}[h]
\centering
\begin{overpic}[width=0.5\textwidth]{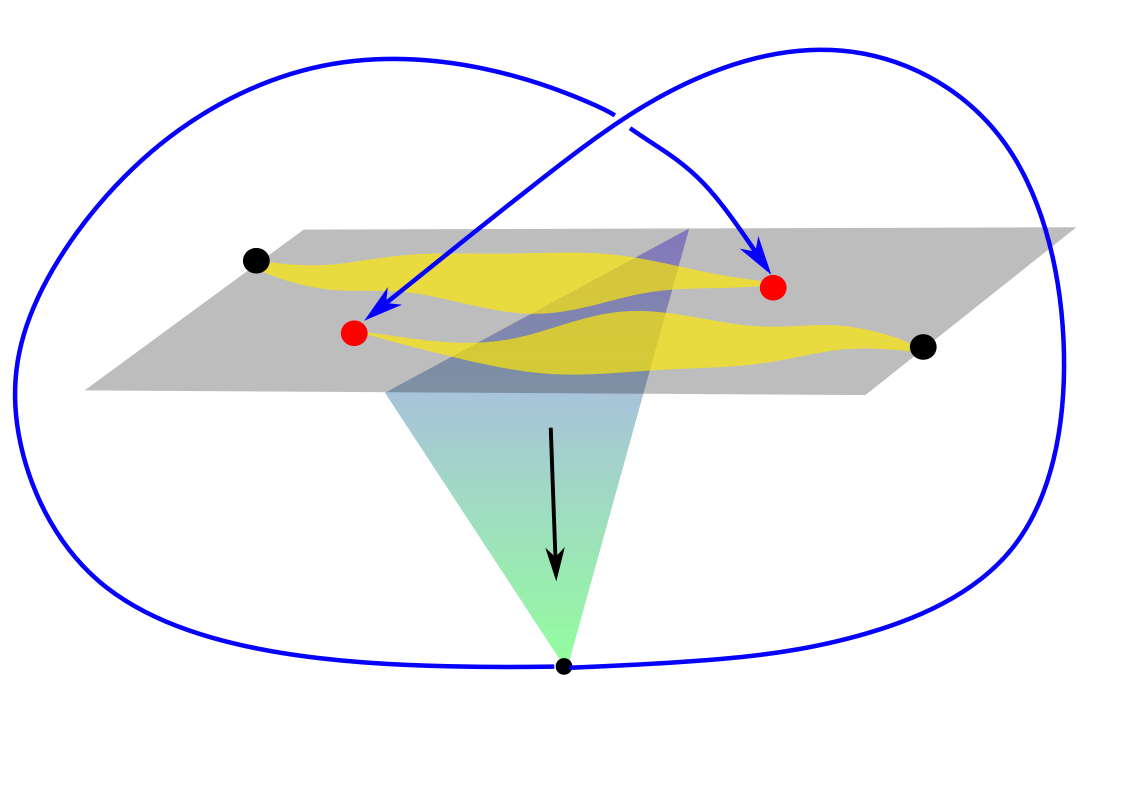}
\put(170,400){$R_0$}
\put(830,450){$R_1$}
\put(435,280){ $W$}
\put(450,0){ $0$}
\put(170,460){$p_0$}
\put(850,390){$p_1$}
\end{overpic}
\caption{\textit{The first return map.}}
\label{splitting11}
\end{figure}

Before moving on, we remark that as proven in \cite{Pi}, the set $T\subseteq P$ of trefoil parameters includes an open set (in detail see Lemma 2.2). That being said, based on the numerical evidence one should also expect the set $T$ to include components which are singletons. In detail, as was observed numerically, for $v_0=(\beta,\sigma,\rho)=(\frac{8}{3},10.2,30.38)$ the Lorenz system appears to create a structurally unstable heteroclinic trefoil knot (see Fig.5 in \cite{BSS}). The parameter $v_0$ is sometimes referred to as the \textbf{first $T$ point} - for more details, see \cite{BSS}, \cite{KOC} and the references therein.

\begin{figure}[h]
\centering
\begin{overpic}[width=0.4\textwidth]{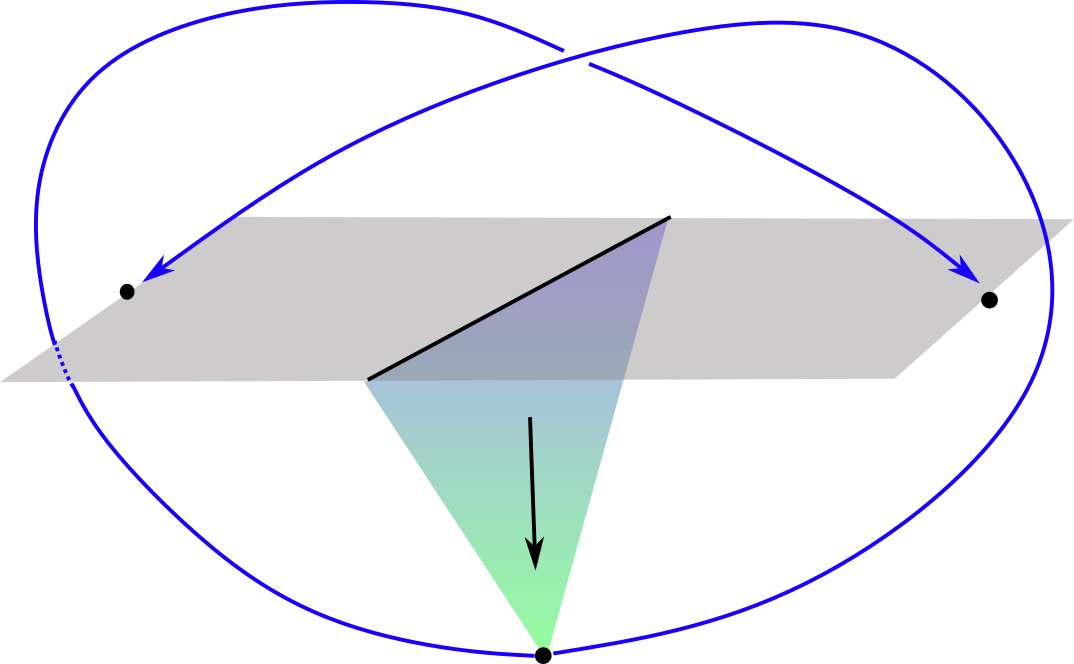}

\end{overpic}
\caption{\textit{The cross-section $R$ and the heteroclinic trajectories connecting the origin to the fixed points $p^\pm$ at trefoil parameters $p\in P$.}}
\label{trefoil}
\end{figure}

Having reviewed the necessary prerequisites on the Lorenz system, we are now ready to begin. To this end, given any parameter $v\in P$ recall we denote by $R_{0}$ and $R_{1}$ the components of $R\setminus W$, and further recall we denote the first-return map associated with the vector field $L_v$ by $\psi_v:R_{0}\cup R_{1}\to R$ (see the illustration in Fig.\ref{cross} and Fig.\ref{RECT}). We first prove the following technical fact, which proves that as we perturb parameters $p\in T$ to some other $v\in P$, the dynamical complexity of the attractor persists in some form:
\begin{proposition}
    \label{persistence}
    Let $p\in T$ be a trefoil parameter for the Lorenz system, and let $k,n_1,...,n_k>0$ be some natural numbers. Then, if $v\in P$ is sufficiently close to $p$ the first return map $\psi_v$ has $k$ periodic orbits of minimal periods $n_1,...,n_k$ in $R$.
\end{proposition}
\begin{proof}
The proof of Prop.\ref{persistence} will be completely two-dimensional. To begin, we first recall the Fixed Point Index, following Ch.VII.5 in \cite{Dold}. Consider a disc $V$ and a continuous map $f:V\to\mathbb{R}^2$ with no fixed points in $\partial V$ - then, the Fixed Point Index of $f$ in $V$ is the degree of $f(x)-x$ in $V$. As proven in Prop.VII.5.5 in \cite{Dold}, when $f$ has no fixed points in $V$ the Fixed Point Index is $0$. The reason we are interested in the Fixed Point Index is due to its homotopy invariance property. Specifically, given a homotopy of continuous maps $f_t:V\to\mathbb{R}^2$, $t\in[0,1]$ set $Fix=\{(x,t)|f_t(x)=x\}$ - then, if $Fix$ is compact in $V\times[0,1]$ the Fixed Point Index of $f_1$ in $V$ is the same as that of $f_0$.

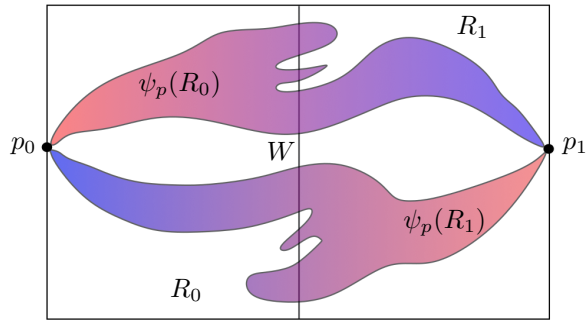
\begin{figure}[h]
\centering
\begin{overpic}[width=0.4\textwidth]{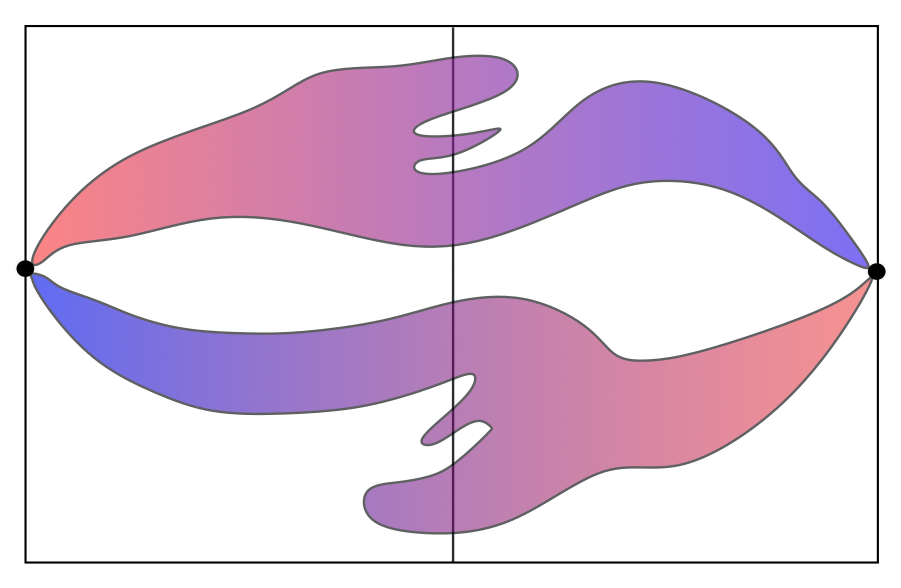}
\put(260,70){$R_0$}
\put(800,570){$R_1$}
\put(425,330){ $W$}
\put(-40,350){$p_0$}
\put(1000,350){$p_1$}
\put(200,460){$\psi_p(R_0)$}
\put(700,200){$\psi_p(R_1)$}
\end{overpic}
\caption{\textit{The cross-section $R$ at trefoil parameters $p\in P$. Again, the sub-rectangles $R_{0}$ and $R_{1}$ correspond to the components of $R\setminus W$.}}
\label{RECT}
\end{figure}

Now, let $p\in T$ be a trefoil parameter - we now recall the proof of part $(2)$ in Th.\ref{tali} (see Th.1.1 d Th.1.2 in \cite{Pi}). Recall that result was proven by "blowing up" the heteroclinic trefoil knot, which allowed for a smooth deformation of the Lorenz system at the parameter $p$ to a vector field $H$, hyperbolic on its invariant set. In detail, this deformation was carried out by expanding the two fixed points $p_0,p_1$ by Hopf bifurcations into arcs on $R_0,R_1$ (see the illustration in Fig.\ref{isot}). As proven in \cite{Pi}, this deformation of the flow induced an isotopy of the first return map $\psi_p:R_0\cup R_1\to R$ to the first-return map $f:R_0\cup R_1\to\mathbb{R}^2$ of $H$, which is conjugate to the Fake Horseshoe map on its invariant set (see the illustration in Fig.\ref{isot}). Following that, using the theory of two-dimensional dynamics it was proven all the periodic orbits for $f$ in $R$ persisted, without changing their minimal period or colliding with one another, as the flow was deformed back to $L_p$ (for the complete details, see the proofs of Th.1.1 and Th.1.2 in \cite{Pi}). Put simply, as $f$ is isotopically deformed back to $\psi_p$, the periodic orbits of $f_p$ are continuously deformed into periodic orbits for $\psi_p$ without changing their minimal period, or colliding with one another - i.e., the dynamics of $\psi_p$ are complex at least like those of $f$.

\begin{figure}[h]
\centering
\begin{overpic}[width=0.5\textwidth]{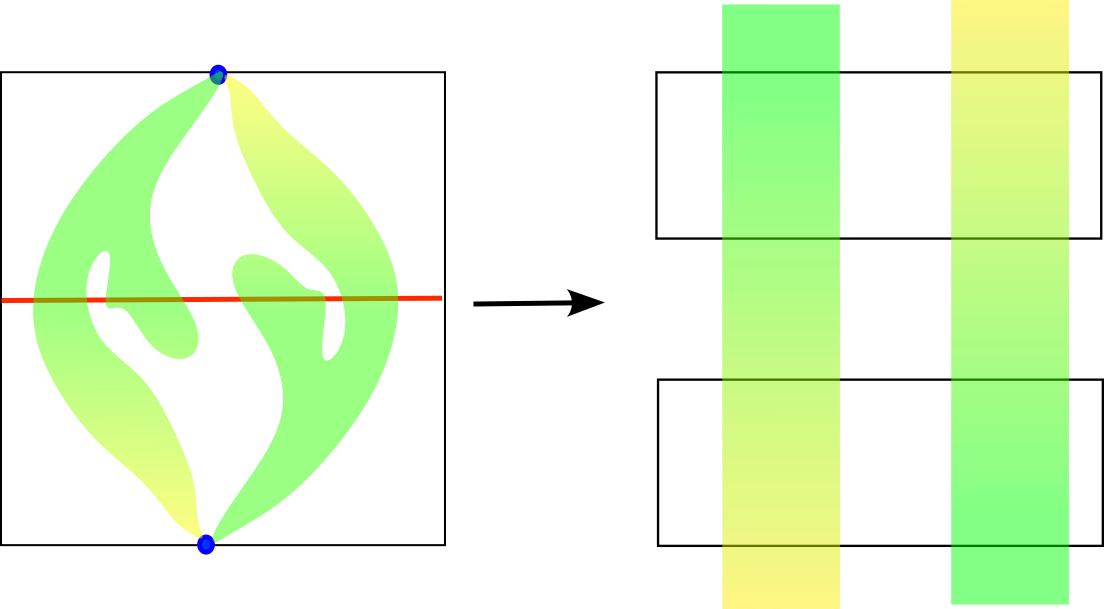}
\put(300,100){$R_0$}
\put(300,420){$R_1$}
\put(-70,270){ $W$}
\put(180,20){$p_0$}
\put(180,520){$p_1$}
\put(560,40){$A$}
\put(1000,40){$B$}
\put(560,490){$C$}
\put(1000,490){$D$}
\end{overpic}
\caption{\textit{The deformation changing $\psi_p$ to the hyperbolic rectangle map $f$ on the right. The fixed points $p_0$ and $p_1$ are opened to the arcs $AB$ and $CD$ (respectively). The red arc $W$ is opened to the region separating $R_0$ and $R_1$.}}
\label{isot}
\end{figure}
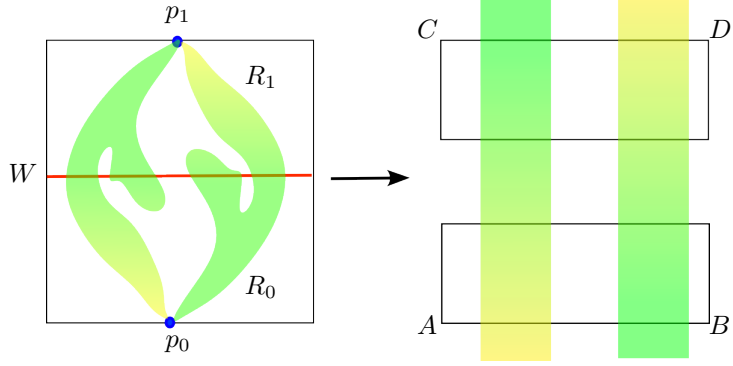

We will prove Prop.\ref{persistence} by analyzing the deformation described above in more detail. To this end, let $I$ denote the invariant set of $\psi_p$ in $R_0\cup R_1$ and let $x$ be a periodic point of minimal period $n$ in $I$ s.t. when we isotope $\psi_p$ to $f$, $x$ is deformed to some $x'$, a periodic point for $f$ of minimal period $n$ (by the above, $x$ and $x'$ exists). We will prove Prop.\ref{persistence} by showing the Fixed Point Index of $\psi^n_p$ on some neighborhood of $x$ is non-zero. To do so, note it is easy to see we can encase $x$ in a topological disc $D$ s.t. the following is satisfied:
\begin{itemize}
    \item $\partial D$ is composed from a finite collection of arcs in $\cup_{k> n}\psi_p^{-k}(W)$.
    \item For all $1\leq j<n$, $\psi^j_p(D)\cap \overline D=\emptyset$.
\end{itemize}

We now denote by $f_t:R\setminus W^n\to \mathbb{R}^2$ the isotopy deforming $\psi^n_p:R\setminus W^n\to R$ to $f^n:R\setminus W^n\to \mathbb{R}^2$ s.t. $f_0=\psi^n_p$, $f_1=f^n$ and $W^n$ is the collection of all pre-images of $W$ up to order $n$ (note $W^n$ changes along the isotopy). Recalling the proofs of Th.1.1 and Th.1.2 in \cite{Pi}, we know $f$ is obtained from $\psi_p$ by moving the pre-images of $W$ in $R$ - or in other words, the isotopy "straightens" the pre-images of $W$ as it deforms $\psi_p$ to the fake horseshoe map $f$. This proves we can choose the set $D$ above s.t. its properties are preserved by the isotopy, i.e.:
\begin{itemize}
    \item For all $t\in[0,1]$, $D$ is a topological disc.
    \item For all $t\in[0,1]$, $\partial D$ is composed from a finite collection of arcs in $\cup_{k>n}f_t^{-k}(W)$.
    \item For all $t\in[0,1]$ and for all $1\leq j<n$, we have $f^j_t(D)\cap \overline D=\emptyset$.
\end{itemize}

Since for all $t\in[0,1]$ we have $\partial D\cap (\cup_{j=1}^nf^{-n}_t(W))=\emptyset$, it is easy to see there can be no periodic orbits of minimal period $n$ in $\partial D$, hence the set $\{(s,t)|f^n_t(s)=s\}$ is compact in $D\times[0,1]$. By the discussion above, it follows the Fixed Point Index is defined and constant along the isotopy, or explicitly, the fixed point index of $\psi^n_p$ in $D$ is the same as that of $f^n$. Therefore, to compute the Fixed Point Index of $\psi^n_p$ at $D$, we now compute the Fixed Point Index  $f^n$. Since $f:R_0\cup R_1\to R$ is smooth, we know the degree of $f^n(s)-s$, $s\in D$ is determined by $\sum_{s\in D,f^n(s)=s} sign(det(J_{f^n}(s)-Id))$, where $J_{f^n}$ denotes the Jacobian matrix and $Id$ denotes the identity matrix. Because $f:R\setminus W\to\mathbb{R}^2$ acts as a fake horseshoe on $R_1\cup R_0$ all the eigenvalues of the Jacobian $J_f(s)$ are positive, hence the same is true for $J_{f^n}(s)$, for all $n$. Moreover, by the hyperbolicity of $f$ we know that for all such $s$, $J_{f^n}(s)$, $s\in D$ has one eigenvalue in $(0,1)$ and another in $(1,\infty)$ - which implies $sign(det(J_{f^n}(s)-Id))=-1$. Or, in other words, we have proven the degree of $f^n(s)-s$ in $D$ is strictly negative - and consequentially, the Fixed Point Index of $\psi^n_p$ in $D$ is also negative.

Having proven the Fixed Point Index of $\psi^n_p$ is negative, we are now in a position to conclude the proof of Prop.\ref{persistence}. We first note that if $v\in P$ is a parameter sufficiently close to $p$, the following holds:
\begin{itemize}
    \item There exists a homotopy $g_t:D\to{R}$ s.t. $g_0=\psi_p$, and $g_1=\psi_v$ - hence the maps $\psi^n_p:D\to R$ and $\psi_v^n:D\to R$ are smoothly homotopic.
    \item For all $1\leq j<n$, $g^j_t(D)\cap \overline{D}=\emptyset$.
\end{itemize}

Now, let $Fix=\{(s,t)|g^n_t(s)=s\}\subseteq D\times[0,1]$ - we note that if $Fix$ has a limit point in $\partial D\times\{0\}$ then $g^n_0=\psi^n_p$ must have a fixed point on $\partial D$. Since by our choice of $D$ we already know this is not the case, it follows that provided $v$ is sufficiently close to $p$ the set $Fix$ lies away from $\partial D\times[0,1]$. As such, the Fixed Point Index of $g^n_1=\psi^n_v$ in $D$ is the same as that of $\psi^n_p$, that is, strictly negative - which proves $\psi^n_v:D\to R$ has a fixed point $x''$ - i.e., $x''$ is a periodic point for $\psi_v$ in $R_0\cup R_1$. And since $\psi^j_p(D)\cap\overline{D}=\emptyset$ for all $1\leq j<n$, it follows the minimal period of $x''$ is $n$. The proof of Prop.\ref{persistence} is now complete.
\end{proof}
\begin{remark}
\label{gen1}    It is easy to see Prop.\ref{persistence} easily generalizes to sufficiently small $C^k$ perturbations of the Lorenz system at trefoil parameter, where $k\geq1$. 
\end{remark}
We now recall that for all $v\in P$ the first return map $\psi_v:R_0\cup R_1\to R$ defines symbolic dynamics on its invariant set. In detail, recall that by Th.\ref{tali} there exists some $\Sigma_v\subseteq\{0,1\}^\mathbb{N}$, invariant under the one-sided shift $\sigma:\{0,1\}^\mathbb{N}\to\{0,1\}^\mathbb{N}$ s.t. if $I_v$ is the invariant set of $\psi_v$ in $R_0\cup R_1$ there exists a continuous, surjective $\pi_v:I_v\to\Sigma_v$ satisfying $\pi_v\circ \psi_v=\sigma\circ\pi_v$ (by Th.\ref{tali}, when $v\in T$ then $\Sigma_v=\{0,1\}^\mathbb{N}$). With these notations in mind, we now prove the following immediate corollary of Prop.\ref{persistence}:
\begin{corollary}
    \label{pers1} Let $p\in T$ be a trefoil parameter, and let $s\in\{0,1\}^\mathbb{N}$ be periodic. Then, for any $v$ sufficiently close to $p$, $s\in\Sigma_v$. In addition, if the minimal period of $s$ w.r.t. the one-sided shift is $n$, $\pi^{-1}_v(s)$ includes at least one periodic point of minimal period $n$.
\end{corollary}
\begin{proof}
    Let $p\in T$ be a trefoil parameter, and let $x\in R$ be a periodic point of minimal period $n$ given by Th.\ref{tali} - in other words, as recounted earlier (and with previous notation), as we isotope $\psi_p:R_1\cup R_0\to R$ to $f:R_1\cup R_0\to\mathbb{R}^2$, $x$ is deformed to $x'$, a periodic orbit for $f$ of minimal period $n$. In detail, by Th.\ref{tali}, for every periodic $s\in\{0,1\}^\mathbb{N}$ there exists such a periodic point $x$, with the same minimal period, s.t. the itinerary of $x$ w.r.t. the first return map $\psi_p:R_0\cup R_1\to R$ is precisely $s$. By Prop.\ref{persistence} we already know that as we perturb the vector field $L_p$ to $L_v$, the point $x$ is perturbed to a periodic point $x''$ for $\psi_v:R_0\cup R_1\to R$ with the same minimal period. Therefore, to complete the proof it remains to show the itineraries of $x$ and $x''$ in $R_0\cup R_1$ are the same - by definition this will immediately imply $s\in\Sigma_v$, and conclude the proof.
    
    To do so, note the set $D$ introduced in the proof of Prop.\ref{persistence} varies continuously as $L_p$ is perturbed to $L_v$ - and in particular, so do the flow lines connecting $D$ to $\psi^n_p(D)$, which are smoothly deformed to the flow lines connecting $D$ and $\psi^n_v(D)$. Therefore, as we constructed $D$ in the proof of Prop.\ref{persistence} s.t. the flow lines connecting $D$ and $\psi^n_p(D)$ do not intersect $W$, the same is true for the flow lines connecting $D$ and $\psi^n_v(D)$ (at least when $v$ is sufficiently close to $p$). This implies the itineraries of $x$ and $x'$ in $R_0\cup R_1$ w.r.t. $\psi_p$ and $\psi_v$ (respectively) are the same, hence $x'\in\pi^{-1}_v(s)$. All in all, the proof of Cor.\ref{pers1} is now complete.
\end{proof}
Prop.\ref{persistence} and Cor.\ref{pers1} together prove the dynamics of Lorenz systems generated by parameters sufficiently close to trefoil parameters must also have "complex" dynamics. This leads us to ask the following - can we somehow describe this complexity? We answer this question by proving Th.\ref{reduct} below. In order to state Th.\ref{reduct} we first need to recall several definitions. To this end, recall an interval map $F\colon [0,1]\to [0,1]$ is called a \textbf{Lorenz-like map} if it satisfies the following conditions:
\begin{enumerate}
\item there is a \textbf{critical point} $c\in (0,1)$ such that $F$ is continuous and strictly increasing on $[0,c)$ and $(c,1]$;
\item $F(c_+)=\lim_{x\to c^{+}}F(x)<\lim_{x\to c^{-}}F(x)=F(c_-)$.
\end{enumerate}
In the case where $F(c_+)=0$ and $F(c_-)=1$, we call it a \textbf{Lorenz map}. If in addition  $F$ is differentiable for all points not belonging to a finite set $E\subseteq [0,1]$ and $\inf_{x\not\in E} F'(x)>1$,

Then, $F$ is an \textbf{expanding Lorenz map}. The above definition does not specify what the image of the critical point $c$ is. We will assume that $F(c)=0$; however, this choice is arbitrary and was made solely to avoid ambiguity. An expanding Lorenz map with a constant slope, that is, a map of the form
$$
	F_{\beta,\alpha}(x)=\beta x+\alpha \Mod{1}=
	\begin{cases}
		\beta x+\alpha, &\text{for}\ x\in\left[0,c\right)  \\
		\beta x+\alpha-1, &\text{for}\ x\in \left[c,1\right] 
	\end{cases},
	$$
where $1<\beta\leq2$, $0\leq\alpha\leq2-\beta$ and $c=\frac{1-\alpha}{\beta}$, will be called a \textbf{$\beta$-transformation} (see Fig.\ref{fig:Lorenz_map}).

\begin{figure}[!h]
	\centering
	\begin{tikzpicture}[scale=5]
	\draw[thick] (0,0) -- (1,0);
	\draw[thick] (0,0) -- (0,1);
	\draw (0,1) node[left]{1};
	\draw (1,0) node[below]{1};
	\draw (-0.02,0) node[below]{0}; 
	\draw[thick,domain=0:1,variable=\y] plot ({1},{\y});
	\draw[thick,domain=0:1,variable=\x] plot ({\x},{1});

	\draw[domain=0:1,smooth,variable=\x,red,very thick] plot ({\x},{\x});

	\draw[domain=0:0.5,smooth,variable=\x,blue,very thick] plot ({\x},{(\x+0.5)^2});

	\draw[domain=0.5:1,smooth,variable=\x,blue,very thick] plot ({\x},{-(\x-1.5)^2+1});
	
	\filldraw [red] (0.5,0) circle (0.2pt) node[anchor=north] {$c$};
	\end{tikzpicture}
        \qquad
        \begin{tikzpicture}[scale=5]
	\draw[thick] (0,0) -- (1,0);
	\draw[thick] (0,0) -- (0,1);
	\draw (0,1) node[left]{1};
	\draw (1,0) node[below]{1};
	\draw (-0.02,0) node[below]{0}; 
	\draw[thick,domain=0:1,variable=\y] plot ({1},{\y});
	\draw[thick,domain=0:1,variable=\x] plot ({\x},{1});
	
	\draw[domain=0:1,smooth,variable=\x,red,very thick] plot ({\x},{\x});

	\draw[domain=0:0.53846,smooth,variable=\x,blue,very thick] plot ({\x},{1.3*\x+0.3});

	\draw[domain=0.53846:1,smooth,variable=\x,blue,very thick] plot ({\x},{1.3*\x+0.3-1});

	\filldraw [red] (0.53846,0) circle (0.2pt) node[anchor=north] {$c$};
	\end{tikzpicture}
	\caption{\textit{Expanding Lorenz maps with non-constant and constant slope.}}
	\label{fig:Lorenz_map}
\end{figure}
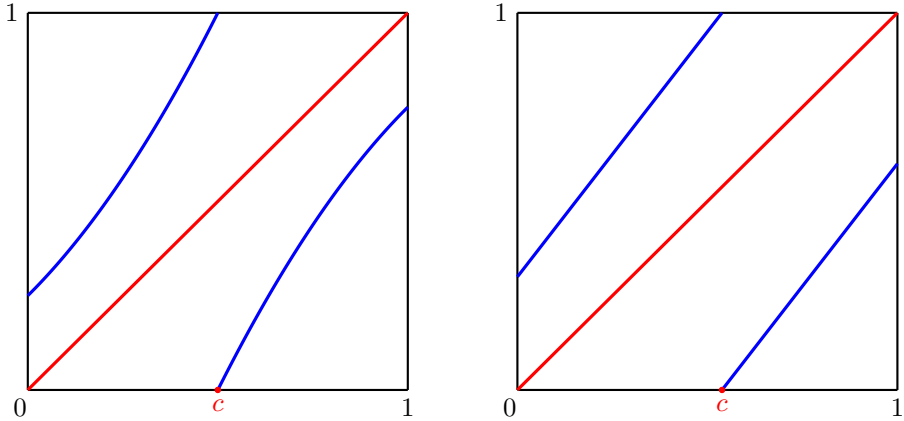

Next, let us recall some facts from the kneading theory. Consider a Lorenz-like map $F\colon[0,1]\to[0,1]$ with the critical point $c$. Set $C_F:=\bigcup_{i=0}^{\infty}F^{-i}(\{c\})$ and $I_F:=[0,1]\setminus C_F$. For $x\in I_F$ we define the \textbf{kneading sequence} (or \textbf{itinerary}) $\eta(x)\in\{0,1\}^{\mathbb{N}_0}$ by
$$
\eta(x)_n=
\begin{cases}
	1;&\text{if}\ F^n(x)>c\\
	0;&\text{if}\ F^n(x)<c
\end{cases},
\quad\text{for}\quad n\in\mathbb{N}_0.
$$
For $x\in C_F$ we define the \textbf{upper and lower kneading sequences} (\textbf{itineraries}) of $x$ as:
		\begin{equation}\label{eq:def_knead}
		\eta_+(x):=\lim\limits_{y\to x_+}\eta(y)\quad\text{and}\quad \eta_-(x):=\lim\limits_{y\to x_-}\eta(y),
		\end{equation}
where the limits are calculated through points $y$ which are not preimages of $c$. The pair: 
$$
k_F=(\eta_+,\eta_-):=(\eta_+(c),\eta_-(c))
$$
is called the \textbf{kneading invariant} of $F$. Note that $I_F=[0,1]\setminus C_F$ is the maximal invariant set of $F$ contained in $[0,1]\setminus\{c\}$. Furthermore, the map $\eta\colon I_F\to\{0,1\}^{\mathbb{N}_0}$ is continuous and satisfies $(\eta\circ F)(x)=(\sigma\circ\eta)(x)$ for any $x\in I_F$, where $\sigma\colon\{0,1\}^{\mathbb{N}_0}\to\{0,1\}^{\mathbb{N}_0}$ is the standard (one-sided) shift map.

We now introduce a similar notion for the Lorenz system. Recall that given a parameter $v\in P$, we can partition the cross section $R$ into $R_0$, $R_1$, the two components of $R\setminus W$ (see the illustration in Fig.\ref{cross}) - as proven in Th.\ref{tali}, we can define symbolic dynamics for the invariant set of the first-return map $\psi_v:R_0\cup R_1\to R$ in ${R}\setminus W$. Moreover, recall the fixed point at the origin is a saddle with a one-dimensional unstable manifold, and let $\Gamma_0,\Gamma_1$ denote the components of this invariant manifold. From now on, the \textbf{kneading invariant of the Lorenz attractor at $v\in P$}, denoted by \textbf{$k_v$}, will denote the pair $(\omega_0,\omega_1)$ where $\omega_0,\omega_1$ are the two symbolic sequences in $\{0,1\}^\mathbb{N}$ describing the itineraries of $\Gamma_0$ and $\Gamma_1$. It is easy to see $\omega_0$ and $\omega_1$ are periodic precisely when $\Gamma_0$ and $\Gamma_1$ are homoclinic trajectories to $0$. With these ideas in mind, we prove:
\begin{theorem}
    \label{reduct} Let $T\subseteq P$ denote the collection of all trefoil parameters for the Lorenz system. Then, for all $v\in P$ sufficiently close to $T$, the dynamics on the attractor can be reduced to the one-dimensional map:  
    \begin{equation}\label{eq:one_dim}
f_{r}(x)=\begin{cases}
			rx, & \text{$x\in[0,\frac{1}{2})$}\\
            1-\frac{r}{2}+r(x-\frac{1}{2}), & \text{$x\in(\frac{1}{2},1]$}
		 \end{cases},
\end{equation}
where $r\in(1,2]$ depends on $v$. In detail, let $I_r$ denote the maximal invariant set for $f_r$ in $[0,1]\setminus\{\frac{1}{2}\}$. Then, there exists an invariant set $I_v$ for  $\psi_v:R\setminus W\to R$ s.t. the following holds:
\begin{itemize}
    \item  There exists a continuous, surjective $\pi:I_v\to I_r$ s.t. $\pi\circ\psi_v=f_r\circ \pi$. 
    \item If $x\in I_r$ is periodic of minimal period $n$, $\pi^{-1}(x)$ includes at least one periodic orbit for $\psi_v$ of minimal period $n$.
    \item If the kneading invariant of the Lorenz-like map $f_r$ consists of \textbf{non-periodic} sequences $\eta_+$ and $\eta_-$, it coincides with the kneading invariant for the Lorenz system at $v$.
\end{itemize}
\end{theorem}
Before proving Th.\ref{reduct}, we remark that despite its technical formulation, Th.\ref{reduct} has a clear geometric meaning. Namely, it proves the dynamics on Lorenz attractors corresponding to parameters which are sufficiently close to trefoil parameters are essentially one-dimensional. In particular, it implies the dynamics of the first-return map $\psi_v$ on its invariant set in $R\setminus W$ are complex at least like those of $f_r$ on its invariant set in $[0,1]\setminus\{\frac{1}{2}\}$.
\begin{proof}

The idea of the proof is as follows - similarly to how the proof of Th.\ref{tali} is based on deforming the flow away from the heteroclinic trefoil knot, we will deform the flow away from the separatrices $\Gamma_0$ and $\Gamma_1$ to derive a "straightened" first return map. In fact, the main difficulty in the proof would be to show we can do so without destroying a certain dynamical core. Therefore, inspired by the proof of Th.\ref{tali}, we will do so by carefully deforming the dynamics on the Lorenz attractor at parameter value $v$, s.t. certain "core" dynamics of the first-return maps are preserved, while the "noise" is removed.

To begin, we first make the following observation - recall that per Th.\ref{tali} given any $v\in P$, every point in $\psi_v(\partial R\setminus\{p_0,p_1\})$ is interior to $R$ - in other words, for every $s\in\partial R\setminus\{p_1,p_0\}$, there is a positive distance between $\psi_p(s)$ and $\partial R$. This implies the maximal invariant set of $\psi_v$ in $R_0\cup R_1$ cannot accumulate on $\partial R\setminus\{p_0,p_1\}$ - and in particular, the periodic orbits for the flow do not accumulate there. Therefore, for any $v\in P$, there exists some open neighborhood $N$ of $\partial R\setminus\{p_1,p_0\}$ in $R$, dependent on $v$, where $\psi_v$ has no periodic orbits. It is easy to see we can choose $N$ s.t. $\psi_v^n(N)\cap\psi_v^k(N)=\emptyset$ for all $n\ne k$ (whenever that iteration is defined). The set $N$ will be useful in the proof due to these properties - namely, we can deform the first-return map isotopically on it without adding (or destroying) any periodic orbits for the flow. More informally, by moving flow lines beginning in $N$ we can "straighten" the dynamics on the attractor with minimal loss of data, as by doing so we just "squeeze and straighten" the dynamics around the attractor. 

\begin{figure}[h]
\centering
\begin{overpic}[width=0.4\textwidth]{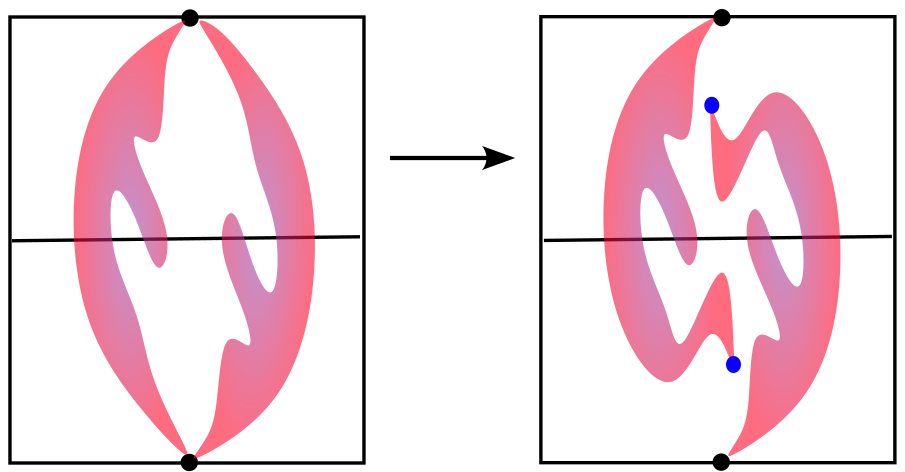}

\end{overpic}
\caption{\textit{The perturbation of the first-return map of a trefoil parameter (on the left) to a non-trefoil parameter (on the right).}}
\label{firstrett}
\end{figure}

We now do just that - namely, deform isotopically the first return map $\psi_v$ by smoothly deforming the flow lines connecting $N$ to its iterates. To this end, recall the arc $W$ partitions the rectangle $R$ into the rectangles $R_0$ and $R_1$ (see Fig.\ref{cross}). As the dynamics of $L_v$ can be smoothly deformed to those of a trefoil parameter as we vary $v$ towards the set $T$, it follows that whenever $v\in P$ is sufficiently close to $T$ the set $\psi_v(R_0\cup R_1)\cap R_i$, $i=0,1$ includes at least two components (see Fig.\ref{firstrett}). Now, consider the pre-image $\psi^{-1}_v(W)$ in the sub-rectangles $R_0,R_1$. From now on, we adopt the convention $p_0\in\partial R_0$ and $p_1\in\partial R_1$ (see the illustration in Fig.\ref{cross}).

\begin{figure}[h]
\centering
\begin{overpic}[width=0.4\textwidth]{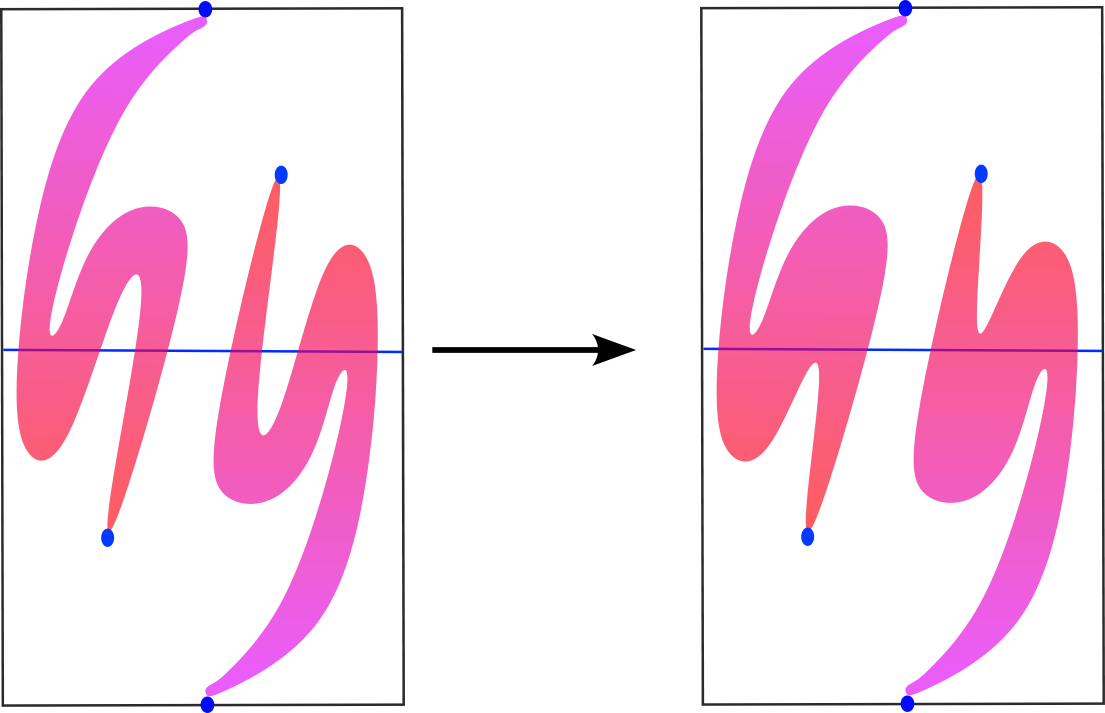}

\end{overpic}
\caption{\textit{Straightening the first return map by expanding $N$ and its iterates.}}
\label{deformation11}
\end{figure}

To continue, consider the sets $\psi_v(R_0)=R_0^1$ and $\psi_v(R_1)=R_1^1$. By expanding the set $N$ and moving initial conditions in it, we can ensure $R^1_0$ and $R^1_1$ each intersect $W$ in precisely one curve that connects the two arcs of $\psi_v(\partial R_j\setminus(W\cup\{pj\}))$, $j=0,1$, as illustrated in Fig.\ref{deformation11}. Moreover, it is easy to see this deformation does not change the kneading of the flow. We now do something similar, simultaneously, for all the iterates of $\psi_v$. In detail, for every $k\geq0$ set $\mathbb{W}_k=\cup_{k\geq t\geq0}\psi^{-t}_v(W)$, and let $R^k_j$ be some component of $\psi^k_v(R_j\setminus\mathbb{W}_k)$. By moving flow lines in $N$ we ensure that $R^k_j\cap W$ has \textbf{at most} one curve connecting the $k$-th iterates of the two arcs $\partial R_j\setminus(W\cup\{pj\})$ on $\partial R^k_j$, $j=0,1$ (see the illustration in Fig.\ref{deformation11}). Again, this proccess does not change the kneading properties of the flow - moreover, we perform this deformation s.t. the symmetry of the Lorenz system w.r.t. the $z$-axis is preserved.

To continue, let $L'_v$ denote the vector field generated by simultaneously deforming all the iterates of $\psi^k_v$ on $N$ as described above, and denote by $\varphi_v:R_0\cup R_1\to{R}$ the corresponding first return map . We now recall that per Th.\ref{tali} there exists some $\Sigma_v\subseteq\{0,1\}^\mathbb{N}$, invariant under the one-sided shift $\sigma:\{0,1\}^\mathbb{N}\to \{0,1\}^\mathbb{N}$, s.t. the dynamics of $\sigma$ on $\Sigma_v$ are a factor map of those of $\psi_v$ on its invariant set in $R_0\cup R_1$. Similarly, let  $\Xi_v\subseteq \{0,1\}^\mathbb{N}$ denote the shift-invariant subset of $\{0,1\}^\mathbb{N}$ which describes the itineraries of the invariant set of $\varphi_v$ in $R_1\cup R_0$ - as the isotopic deformation of $\psi_v$ to $\varphi_v$ on $R_0\cup R_1$ at most removes components from the invariant set of $\psi_v$, it follows $\Xi_v\subseteq \Sigma_v$. That being said, as we perform the deformation only on $N$, all the periodic orbits of $\psi_v$ survive - or in other words, $\Xi_v$ includes all the periodic sequences in $\Sigma_v$ 

We now further isolate a dynamical core inside the invariant set of $\varphi_v$ in $R_0\cup R_1$. To this end, set $\mathbb{W}=\cup_{t\geq0}\varphi^{-t}_v(W)$, and let $A_1,A_2,A_3$ and $A_4$ denote the arcs of $\partial R\setminus(W\cup\{p_0,p_1\})$, as illustrated in Fig.\ref{deformation13}. Due to the construction of $L'_v$ from $L_v$, we know that with respect for the vector field $L'_v$, every component of $\mathbb{W}$ falls into precisely one of the following categories:

\begin{itemize}
    \item Category $A$ - a straight line connecting $A_1$ and $A_2$.
    \item Category $B$ - a straight line connecting $A_3$ and $A_4$.
    \item Category $C$ - a curve in $R_0$ or $R_1$ with precisely two points on $\partial R$, $q_1$ and $q_2$, which lie on the same $A_j$, $j=1,2,3,4$ (see the illustration in Fig.\ref{deformation13}).
\end{itemize}

\begin{figure}[h]
\centering
\begin{overpic}[width=0.5\textwidth]{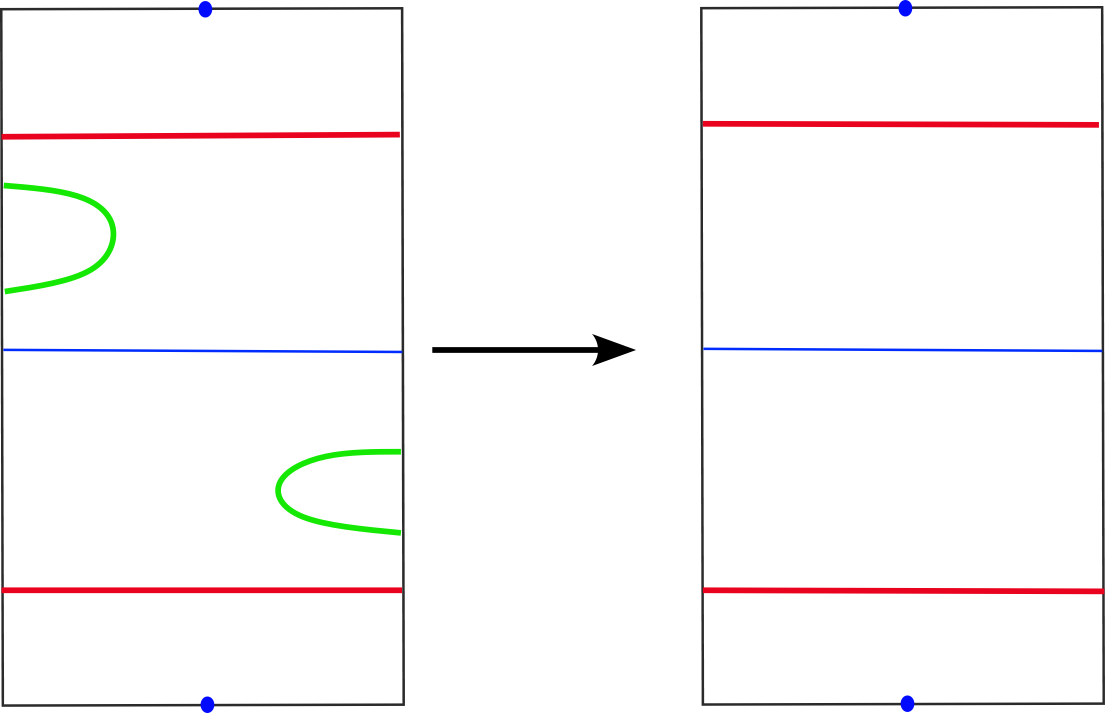}
\put(370,200){$A_2$}
\put(370,420){$A_4$}
\put(-50,200){$A_1$}
\put(-50,420){$A_3$}
\put(-70,310){ $W$}
\put(570,310){ $W$}
\put(170,-30){$p_0$}
\put(800,660){$p_1$}
\put(800,-30){$p_0$}
\put(170,660){$p_1$}
\put(580,200){$A_1$}
\put(1000,200){$A_2$}
\put(580,420){$A_3$}
\put(1000,420){$A_4$}
\end{overpic}
\caption{\textit{The red arcs denote Catgory $A$ and $B$ arcs, while the green denote Category $C$. We collapse Category $C$ arcs to the boundary.}}
\label{deformation13}
\end{figure}
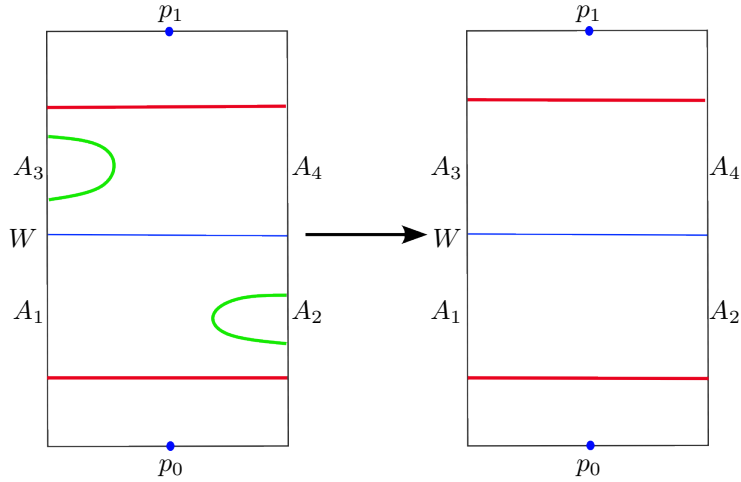

In other words, an arc in $\mathbb{W}$ is Category $C$ if there exists some subarc $J$ in $\partial R\setminus\{p_0,p_1\}$, some arc $J'\subseteq W$ and some $k>0$ s.t. $\varphi^k_v(J)\cup J'$ is a Jordan domain in $R_0$ or $R_1$ whose interior lies wholly inside the $k$-th iterate of some component of $R\setminus\mathbb{W}_k$. We now proceed to isotopically remove all Category $C$ components in $\mathbb{W}$. We do so by isotoping $\varphi_v:R_0\cup R_1\to R$ and all its iterates simultaneously to $\phi_v:R_0\cup R_1 W\to R$ by symmetrically collapsing all Category $C$ components in $\mathbb{W}$ to the boundary while preserving the symmetry of $\varphi_v$, as illustrated in Fig.\ref{deformation13}. With previous notations, if $\Phi_v\subseteq\{0,1\}^\mathbb{N}$ denotes the collection of itineraries of for the invariant set of $\phi_v$ in $R\setminus W$, then again $\Phi_v\subseteq\Sigma_v$ - in other words, $\phi_v$ is a factor map of the first return map of the original flow on its invariant set (even if $\phi_v$ itself is not necessarily the first return map of any flow).

Let $L''_v$ denote the resulting new flow - before we continue we study its kneading pattern, when compared to the original Lorenz attractor, defined by the vector field $L_v$. Note that unlike the isotopy of $\psi_v$ to $\varphi_v$, this isotopy possibly does change the kneading - in the sense that as we simultaneously remove Category $C$ arcs, we possibly collide $\Gamma_0$ and $\Gamma_1$ with $W$. When this happens, $\Gamma_0$ and $\Gamma_1$ collapse to homoclinic trajectories to the origin $0$, which implies $L''_v$ defines a periodic kneading sequence, i.e., if $(\omega_0,\omega_1)$ and $(\omega''_0,\omega''_1)$ are the pair of sequence describing the motion of $\Gamma_0$ and $\Gamma_1$ w.r.t. $L_v$ and $L''_v$, then $\omega_i$ differs from $\omega''_i$, $i=0,1$ at most by $\omega''_i$ being periodic. In detail, our argument shows that if $\omega_i=\{\omega^j_i\}_{j\geq0}$, $\omega''_i=\{\omega''^j_i\}_{j\geq0}$, then whenever $\omega_i\ne \omega''_i$ there exists some $n>0$ s.t. the following occurs:
\begin{itemize}
    \item For all $j\leq n$, $\omega^j_i=\omega''^j_i$.
    \item $\omega''_i$ is periodic of minimal period $n$.
\end{itemize}
In other words, the kneading sequences of $L''_v$ are at most a periodic "cutoff" of the kneading sequences of $L_v$, i.e., of the original Lorenz attractor. Of course, our arguments also say something analogues - when neither $\omega''_0,\omega''_1$ are periodic, then $(\omega_0,\omega_1)=(\omega''_0,\omega''_1)$.  
\begin{figure}[h]
\centering
\begin{overpic}[width=0.5\textwidth]{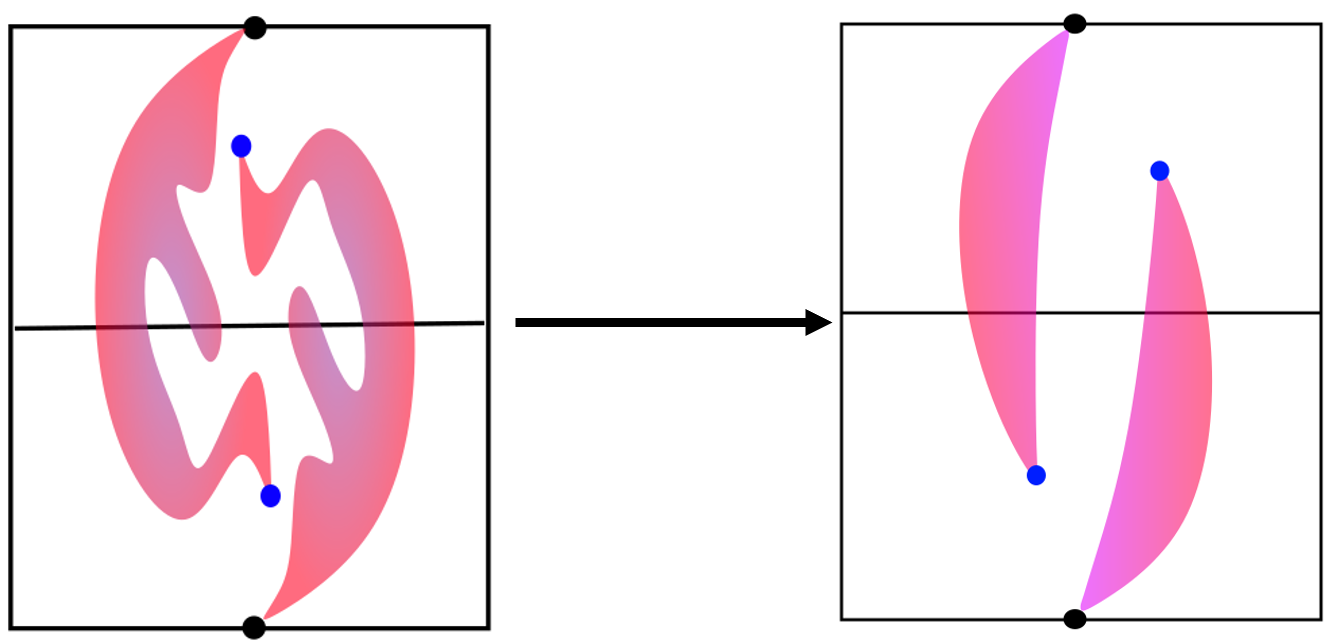}

\end{overpic}
\caption{\textit{Illustrating the deformation of $\psi_v$ to $\phi_v$ - $\phi_v$ is a "straightening" of $\psi_v$.  }}
\label{deformation1}
\end{figure}

Having smoothly deformed $L_v$ to $L''_v$ (thus isotoping $\psi_v$ to $\phi_v$) we are now ready to reduce $\phi_v$ to a one-dimensional interval map. To do so, note that by construction every component of $\mathbb{W}$ w.r.t. $\phi_v$ belongs to either Category $A$ or $B$ (see the illustration in Fig.\ref{deformation1}). In other words, the map $\phi_v$ is a "straightened" isotopic version of the first-return map $\psi_v$, in the sense that it removes all the "extra" dynamical information. In particular, we constructed $\phi_v$ s.t. its invariant set can be described easily, as every component of $\mathbb{W}$ is either Category $A$ or $C$.

We now continue by deforming isotopically $\phi_v:R_o\cup R_1\to R$ to a rectangle map. Begin by opening isotopically $p_0,p_1$ to arcs $AB$ and $CD$ on $\partial R$, and similarly, open $W$ and all its pre-images simultaneously to rectangles connecting the $AC$ and $BD$ sides. In particular, $W$ is opened to a rectangle $R_2$, separating $R_0$ and $R_1$ (see the illustration in Fig.\ref{deformation3}). This transforms $R$ to a rectangle $ABCD$, and allows us we can replace $\phi_v$ with a map $h_v:ABCD\to\mathbb{R}^2$ which is conjugate to $\phi_v$ on its invariant set in $ABCD\setminus(R_2\cup AB\cup CD)$ (see the illustration in Fig.\ref{deformation3}). Since $\phi_v$ is symmetric so is $h_v$ - and since every component of $\mathbb{W}$ is either category $A$ or $B$, it follows every component of $\cup_{n\geq0}h^{-n}_v(R_2)$ in $R_0\cup R_1$ is a rectangle connecting the $AC$ and $BD$ sides inside $ABCD$.

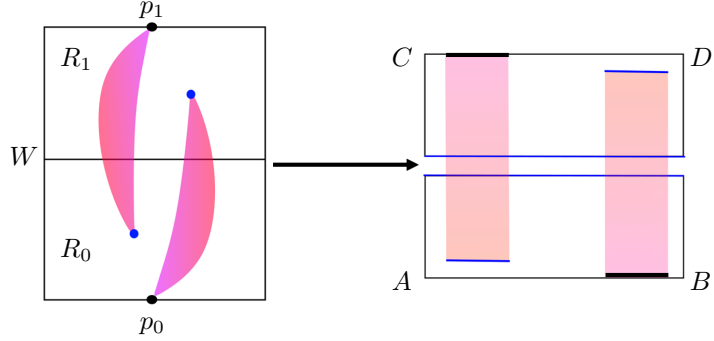
\begin{figure}[h]
\centering
\begin{overpic}[width=0.5\textwidth]{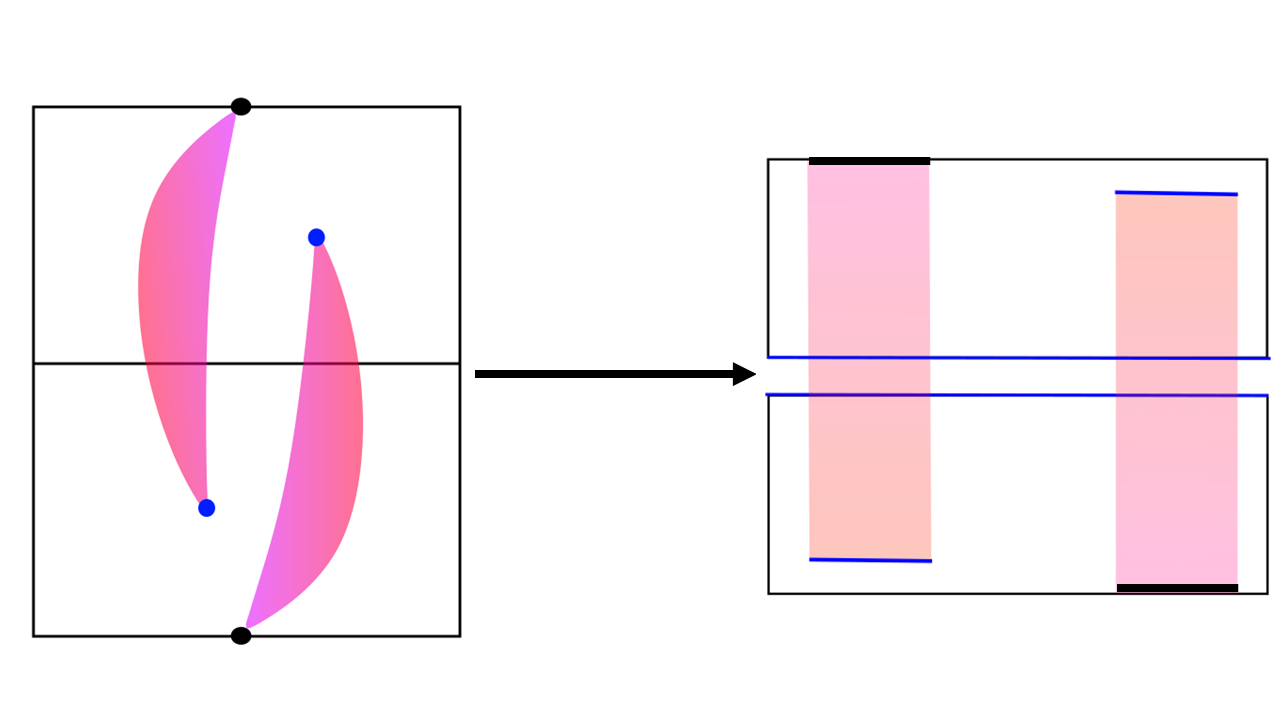}
\put(50,130){$R_0$}
\put(50,420){$R_1$}
\put(-40,270){ $W$}
\put(170,20){$p_0$}
\put(170,500){$p_1$}
\put(550,80){$A$}
\put(1000,80){$B$}
\put(550,420){$C$}
\put(1000,420){$D$}
\end{overpic}
\caption{\textit{Deforming $\phi_v$ to $h_v$. $R_2$ is the rectangle separating the two rectangles on the right.}}
\label{deformation3}
\end{figure}

It follows every component $c$ in the forward-invariant set of $h_v$ in $ABCD\setminus R_2$, can be written as $c=\cap_{i\in\mathbb{N}}D_i$, where $\{D_i\}_i$ is a nested sequence of rectangles, s.t. each $D_i$ connects the $AC$ and $BD$ sides. This proves $c$ is either a rectangle or an interval. Therefore, if necessary we further isotope $h_v:ABCD\to\mathbb{R}^2$ s.t. the following is satisfied:
\begin{itemize}
    \item $h_v$ maps horizontal lines to horizontal lines. In particular, it contracts the horizontal lines in $R_0$ and $R_1$, and expands on the vertical lines.
    \item $h_v$ is orientation preserving on both $R_0$ and $R_1$ and satisfies $h_v(AB)\subseteq AB$, $h_v(CD)\subseteq CD$.
\end{itemize}

In other words, this isotopy "straightens" the dynamics inside the component $c$, s.t. the vertical and horizontal (transverse) foliations are preserved under $h_v$. We now collapse $h_v:R_0\cup R_1\to\mathbb{R}^2$ to the $BD$ side along the horizontal stable foliation the map $h_v$ induces in $R_0\cup R_1$ (that is, the horizontal lines). By identifying the $BD$ side with the interval $[0,1]$ and collapsing the rectangle $R_2$ to the point $\frac{1}{2}$ and all its pre-images to singletons, it follows $h_v$ is collapsed to the following piecewise continuous map (see the illustration in Fig.\ref{model}):

 \begin{equation}
f_{r}(x)=\begin{cases}
			rx, & \text{$x\in[0,\frac{1}{2})$}\\
            1-\frac{r}{2}+r(x-\frac{1}{2}), & \text{$x\in(\frac{1}{2},1]$}
		 \end{cases}
\end{equation}
\begin{figure}[h]
\centering
\begin{overpic}[width=0.2\textwidth]{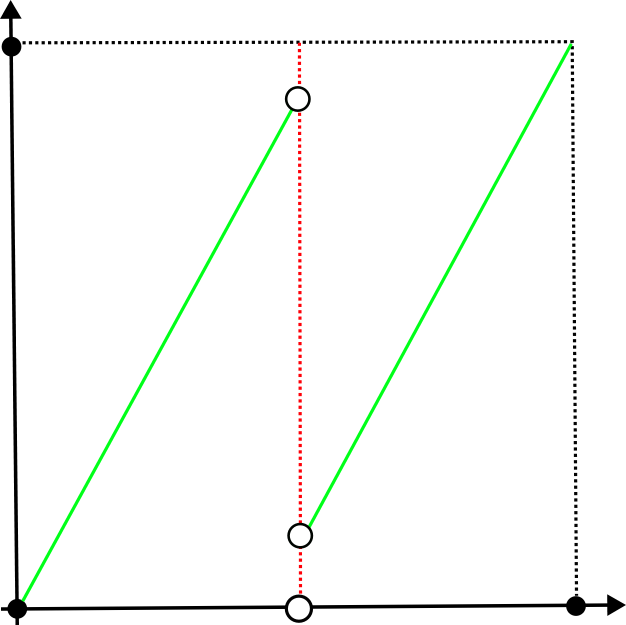}

\end{overpic}
\caption{\textit{The one-dimensional map generated by collapsing $h_v$.}}
\label{model}
\end{figure}

Where the parameter $r$ depends on $v$ - by construction, we must have $r\in(1,2]$. Now, let $I_r$ denote the maximal invariant set of $f_r$ in $[0,1]\setminus\{\frac{1}{2}\}$. Recalling the inclusion $\Phi_v\subseteq\Sigma_v$ and noting that it trivially follows $\Phi_v$ is the collection of itineraries of $h_v$ on its invariant set in $ABCD\setminus R_2$ (and hence also for $f_r$ on its invariant set in $[0,1]\setminus\{\frac{1}{2}\}$), we conclude there exists an invariant set $I_v$ for $\psi_v$ in $R\setminus W$ and a continuous, surjective $\pi_v:I_v\to I_r$ s.t. $\pi_v\circ\psi_v=f_r\circ\pi_v$. Moreover, by our construction it is also easy to see that $f_r$ and $\phi_v$ have the same kneading sequences. Therefore, with previous notations, denoting the kneading invariants of $f_r$ by $(\eta_+,\eta_-)$, precisely one of the following occurs:
\begin{itemize}
    \item $(\omega_0,\omega_1)=(\eta_+,\eta_-)$.
    \item When $(\omega_0,\omega_1)\ne(\eta_+,\eta_-)$, both $\eta_\pm$ are periodic - and writing $\eta_d=\{\eta^d_i\}_{i\geq0}$, $d\in\{+,-\}$ there exists some $n>0$ s.t. with previous notations:
    \begin{enumerate}
    \item For all $j\leq n$, $\eta^d_i=\omega^j_i$, where $j=0$ when $d=-$, and $j=1$ when $d=+$.
    \item Both $\eta_+$ and $\eta_-$ are periodic of minimal period $n$.
    \end{enumerate}
\end{itemize}

Therefore, all in all, to conclude the proof of Th.\ref{reduct} we need to prove that if $x\in I_r$ is periodic of minimal period $n$, then $\pi^{-1}_v(x)$ also includes a periodic point for $\psi_v$ of minimal period $n$. That, however, is immediate - it is easy to see that as we homotope $f_r$ back to $h_v$, the point $x$ is stretched into an interval, $I$, connecting the $AC$ and $BD$ sides, and satisfying the following:

\begin{itemize}
    \item $h^n_v(I)$ is a strict subset of $I$.
    \item For all $1\leq j<n$ we have $h_v^j(I)\cap I=\emptyset$.
\end{itemize}

Consequentially, by the Brouwer Fixed Point Theorem there exists a point $y\in I$ which is periodic of minimal period $n$. Moreover, $y$ does not lie on $I\cap(AC\cup BD)$, which proves there exists a sub-rectangle $D$ satisfying:
\begin{itemize}
    \item $\partial D$ is composed of arcs on $\cup_{n\geq0}h^{-n}_v(R_2)$ - hence $h^n_v$ has no fixed points in $\partial D$.
    \item For all $1\leq j<n$, $h^j_v(D)\cap\overline{D}=\emptyset$.
\end{itemize}

Using the fact $h_v$ contracts the horizontal foliation and expands the vertical foliation on both $R_0$ and $R_1$, similarly to the proof of Prop.\ref{persistence} we conclude the Fixed Point Index of $h^n_v$ on $D$ is non-zero. We now note that as we isotope $h^n_v$ back to $\psi^n_v$ on $D$, no fixed points are added through the boundary - this is true because all the pre-images of $R_2$ are collapsed by the said isotopy to the pre-images of $W$, which has no fixed points of minimal period $n$ for $\psi_v$ in its vicinity. This implies the Fixed Point Index of $\psi^n_v:D\to {R}$ is also non-zero. Consequently, similar arguments to those used in the proof of Prop.\ref{persistence} now prove $D$ includes a periodic point of minimal period $n$ for $\psi_v$. The proof of Th.\ref{reduct} is now complete.
\end{proof}
\begin{remark}
\label{gen2}   The reason we insist on deforming the vector field $L_v$ and the first return map symmetrically is in order to respect the symmetry properties of the Lorenz system. In the spirit of Remark \ref{gen1} we note that if instead we were to consider some arbitrary $C^1$ perturbation of the Lorenz system at trefoil parameter, it would be pointless to deform the isotopies of the first-return map symmetrically. In that case, a similar argument to the one above would result in the one-dimensional factor map:

 \begin{equation}\label{eq:one_dim_pert}
f_{s,r}(x)=\begin{cases}
			rx, & \text{$x\in[0,\frac{1}{2})$}\\
            1-\frac{s}{2}+s(x-\frac{1}{2}), & \text{$x\in(\frac{1}{2},1]$}
		 \end{cases},
\end{equation}
where $s,r\in(1,2]$ - see the illustration in Fig.\ref{nosym}.
\end{remark}
\begin{figure}[h]
\centering
\begin{overpic}[width=0.2\textwidth]{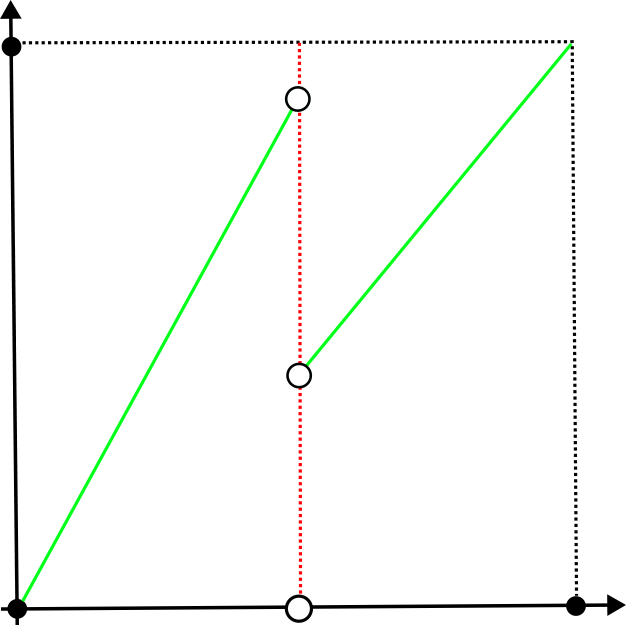}

\end{overpic}
\caption{\textit{Graph of a map $f_{s,r}$ of the form \eqref{eq:one_dim_pert}.}}
\label{nosym}
\end{figure}

Having established proven the dynamics on the Lorenz attractor can be rigorously reduced into a one-dimensional map, we now proceed to study the dynamics of the said model. To this end, we are now going to further reduce the maps $f_r$, $r\in(1,2]$ defined above to a simpler class of Lorenz maps, called the \textbf{$\beta$--transformations}. The reason we do so is because there already exists a rich theory describing the dynamics and bifurcations of the $\beta$--transformations, which we will "pull back" and use it to describe the dynamics and bifurcations of the Lorenz system.

That being said, for reasons which will be apparent later in this paper, we will do so for a more general class of maps. To begin, fix $s,r\in(1,2]$ and consider the map $f_{s,r}$ given by \eqref{eq:one_dim_pert} - as shown above, these maps correspond to the general $C^1$ perturbations of the Lorenz attractor at trefoil parameters. Let $[a,b]$ be an interval defined by the one-sided limits of $f_{s,r}$ at the discontinuity $\frac{1}{2}$, i.e.:
$$
[a,b]:=\left[f_{s,r}\left(\frac{1}{2}_+\right),f_{s,r}\left(\frac{1}{2}_-\right)\right]=\left[1-\frac{s}{2},\frac{r}{2}\right]\ni\frac{1}{2}.
$$
Note that $f_{s,r}([a,b])\subset[a,b]$ and that for every point $x\in(0,1)$ we have $\text{Orb}_{f_{s,r}}(x)\cap[a,b]\neq\emptyset$, where $\text{Orb}_{f_{s,r}}(x)$ denotes the orbit of $x$ under $f_{s,r}$. Therefore, the trajectory of every point from the interval $[0,1]$ (save only for the fixed points $0$ and $1$) will eventually intersect $[a,b]$ and remain there. In particular, the set $(0,1)\setminus[a,b]$ does not contain any periodic points. Since we are mainly interested in the periodic dynamics of the map $f_{s,r}$, we can restrict our study to the interval $[a,b]$. Simple calculations show that the restriction $f_{s,r}|_{[a,b]}\colon[a,b]\to[a,b]$, after rescaling again to $[0,1]$, is an expanding Lorenz map with the critical point $c=\frac{s-1}{r+s-2}$. More precisely, define the rescaling map:
\begin{equation}
\label{eq:rescaling}
h_{u,v}\colon[u,v]\to[0,1],\quad h_{u,v}(x)=\frac{x-u}{v-u}.
\end{equation}
Then:
$$
h_{u,v}^{-1}\colon[0,1]\to[u,v],\quad h_{u,v}^{-1}(x)=(v-u)x+u
$$ 
And for $[u,v]:=[a,b]$, we get:
\begin{equation}
H_{s,r}(x):=\left(h_{a,b}\circ f_{s,r}|_{[a,b]}\circ h_{a,b}^{-1}\right)(x)=
\begin{cases}
			rx+\frac{(2-s)(r-1)}{r+s-2}, & \text{for}\ x\in[0,\frac{s-1}{r+s-2})\\
            sx-\frac{s(s-1)}{r+s-2}, & \text{for}\ x\in(\frac{s-1}{r+s-2},1]
		 \end{cases}.
\end{equation}
\begin{remark}\label{rem:one-to-one}
There exists a one-to-one correspondence between periodic points of the map $H_{s,r}$ and the initial map $f_{s,r}$ given by \eqref{eq:one_dim_pert}, except for fixed points. Furthermore, the kneading invariants of maps $H_{s,r}$ and $f_{s,r}$ coincide.
\end{remark}
Applying this logic for the symmetric case $r=s$, the map $f_r$ given by \eqref{eq:one_dim} and Th.\ref{reduct} is reduced to the (symmetric) $\beta$-transformation: 
$$
F_{\beta}(x):=H_{r,r}(x)=
\begin{cases}
			r x-\frac{r}{2}+1, & \text{for}\ x\in[0,\frac{1}{2})\\
            r x-\frac{r}{2}, & \text{for}\ x\in(\frac{1}{2},1]
		 \end{cases}=rx-\frac{r}{2}+1\Mod{1}=\beta x+\alpha\Mod{1},
$$
Where $\beta=r\in(1,2]$, $\alpha=1-\frac{r}{2}\in[0,\frac{1}{2})$ and the critical point being set at $c=\frac{1}{2}$. As an immediate consequence of Th.\ref{reduct} we obtain the following fact, with which we conclude this section:
\begin{corollary}
    \label{reduct2}
    The Lorenz-like map $f_r$ in the conclusion of Th.\ref{reduct} can be replaced with a $\beta$--transformation $F_\beta$, for the parameter $\beta=r$ as defined above. In particular, $F_\beta$ forms a factor map for the Lorenz attractor, and when its kneading invariant is non-periodic it coincides with that of the corresponding flow.
\end{corollary}

\section{On the dynamics of the $\beta$--transformations}
\label{onesect}

The $\beta$--transformations $F_{\beta,\alpha}(x)=\beta x+\alpha \Mod{1}$, where
$$
(\beta,\alpha)\in\Delta:=\left\lbrace (b,a)\in\mathbb{R}^2\:|\: 1<b\leq2\quad\text{and}\quad 0\leq a\leq 2-b\right\rbrace,
$$
can be considered as the simplest class of Lorenz maps, which at the same time captures a large part of their dynamical aspects. In fact, many Lorenz maps, including maps with slope greater than $\sqrt{2}$ everywhere except the critical point, can be reduced to $\beta$--transformations (see \cite{Parry,Glen,DiCui}). Although our goal is to study the Lorenz attractor, it should be noted that the dynamics of the Lorenz maps is of independent interest: such maps appear, e.g., in economics \cite{Eco,Eco2}, neuroscience \cite{BLS,BNPSR}, and fractal geometry \cite{BarHV}. A comprehensive literature on the applications of maps with discontinuities is presented in the survey article \cite{SIAMrev}. In this section, we are mainly interested in the properties of the symmetric $\beta$-transformations
\begin{equation}
\label{eq:family_maps}
F_{\beta}(x)=\beta x+1-\frac{\beta}{2}\Mod{1},    
\end{equation}
obtained in the previous section, since they can be "pulled back" to the Lorenz attractor via Th.\ref{reduct} (cf. Cor.\ref{reduct2}). However, we will state some of the results in a more general setting than necessary for this purpose, as we hope they will also be useful to scientists working in other fields. Before we start, recall that a Lorenz map $F\colon[0,1]\to[0,1]$ is said to be expanding if its slope satisfies $\inf_{x\not\in E} F'(x)>1$ for a finite set $E\subset[0,1]$. An important consequence of this condition is that the set $C_F:=\bigcup_{i=0}^{\infty}F^{-i}(\{c\})$ is dense in $[0,1]$. The property $\overline{C_F}=[0,1]$ is called the \textbf{topological expanding condition} and will be used repeatedly in our proofs.

The family of maps \eqref{eq:family_maps} depends on a parameter $\beta\in(1,2]$. To describe bifurcations occurring in this family, we will introduce two sequences: $\{\varepsilon_i\}$ and $\{\beta_i\}$. The first one determines a partition of $[\sqrt{2},2]$, and the second one - of the whole interval $(1,2]$ (see Fig.\ref{fig:parameter_space}). For parameters in $[\sqrt{2},2]$, we observe complex topological and measurable dynamics. On the other hand, the maps defined by parameters in $(1,\sqrt{2}]$ globally exhibit a simple periodic structure, but locally, in the vicinity of the critical point, their dynamics is a "copy" of those with slope $\beta\geq\sqrt{2}$. As we will see, not only many of the properties of the maps $F_{\beta}$ survive as we go back to the attractor, but in many ways the bifurcations and the topology of the attractor are determined by them. 

This section is organized as follows: in subsection \ref{subsec:symbolic_dynamics}, we study the maps $F_{\beta}$ from a point of view of the kneading theory, focusing on changes in their symbolic dynamics when the slope exceeds the values $\varepsilon_i$ (see Lemma \ref{lem:eps_knead}, Cor.\ref{cor:kneading_iff} and Prop.\ref{prop:itineraries_rot}). In subsection \ref{subsec:topmeasure}, we collect some results on topological and measurable dynamics of Lorenz maps from the literature and apply them to the maps $F_{\beta}$ (see Cor.\ref{cor:all_periods}, Cor.\ref{cor:density} and Th.\ref{measurprop}). We then use the obtained results, together with knowledge of the symbolic dynamics of the maps $F_{\beta}$, to describe the complexity of the Lorenz attractor around the set of trefoil parameters (see Cor.\ref{mixingattractor}, Cor.\ref{symbolicper} and Cor.\ref{measurableattractor}). We conclude this subsection by comparing the bifurcations in the family of maps $F_{\beta}$ with those of the Lorenz attractors (see Th.\ref{expl}). Following that, in subsection \ref{renorsec} we study renormalization in Lorenz maps (see Lemma \ref{lem:renorm}, Prop.\ref{prop:beta_seq} and Lemma \ref{lem:doubling_conjugate}) and discuss how this phenomenon manifests itself in the Lorenz attractor (see Th.\ref{renormth}). Furthermore, we use the renormalization theory to assign Templates with certain Lorenz attractors (see Th.\ref{templateth}). 

\subsection{Symbolic dynamics and kneading theory}\label{subsec:symbolic_dynamics}
By the classical result of Hubbard and Sparrow \cite{HubSpar}, the kneading invariants (see Def.\eqref{eq:def_knead}) characterize expanding Lorenz maps: any two expanding Lorenz maps with the same kneading invariant are conjugated by a homeomorphism of $[0,1]$. Furthermore, if the pair of sequences $(\eta_+,\eta_-)$ forms the kneading invariant of some expanding Lorenz map $F\colon[0,1]\to[0,1]$, then it satisfies the following condition:
\begin{equation}
    \label{eq:kneading_cond}  \sigma(\eta_+)\preccurlyeq\sigma^n(\eta_+)\prec\sigma(\eta_-)\quad\text{and}\quad\sigma(\eta_+)\prec\sigma^n(\eta_-)\preccurlyeq\sigma(\eta_-)
\end{equation}
for all $n\in\mathbb{N}$, where $\prec$ denotes the lexicographic order on $\{0,1\}^{\mathbb{N}_0}$. The set of the kneading sequences of all points $x\in[0,1]$ for the map $F$ is completely described by its kneading invariant $k_F=(\eta_+,\eta_-)$, that is
\begin{equation}\label{eq:symbol_space}
    \begin{aligned}
\Sigma_F:&=\left\{\eta(x)\:|\:x\in[0,1]\setminus C_F\right\}\cup\left\{\eta_+(x),\eta_-(x)\:|\:x\in C_F\right\}\\
&=\left\{\xi\in\{0,1\}^{\mathbb{N}_0}\:|\: \sigma(\eta_+)\preccurlyeq\sigma^n(\xi)\preccurlyeq\sigma(\eta_-)\ \text{for every}\ n\in\mathbb{N}_0\right\}.
\end{aligned}
\end{equation}
Clearly, for any two kneading invariants $k_F=(\eta_+,\eta_-)$, $k_G=(\xi_+,\xi_-)$ with $\eta_+\preccurlyeq\xi_+$ and $\xi_-\preccurlyeq\eta_-$, we have $\Sigma_G\subset\Sigma_F$. It is also well known that for any expanding Lorenz map $F$, the pair $(\Sigma_F,\sigma_F)$, where $\sigma_F$ is the one-sided shift restricted to $\Sigma_F$, forms a subshift (for an extensive introduction to symbolic dynamics, we refer the reader to \cite{Kurka}).

Our first goal is to discuss the symmetry properties of the maps $F_{\beta}$ and their consequences for the kneading invariants.

\begin{lemma}\label{lem:symmetry}
Let $\beta\in(1,2]$. The following conditions hold:
\begin{enumerate}
    \item\label{lem:symmetry1} $F_{\beta}(x)=1-F_{\beta}(1-x)$ for every $x\in[0,1]$, except the critical point $c=\frac{1}{2}$;
    \item\label{lem:symmetry2} $F_{\beta}^i(0_+)=1-F_{\beta}^i(1_-)$ for every $i\in\mathbb{N}_0$;
    \item\label{lem:symmetry3} the kneading invariant  of the map $F_{\beta}$, i.e.,
    $$
    k_{F_\beta}=(\eta_+,\eta_-)=(\eta^+_0\eta^+_1\ldots,\eta^-_0\eta^-_1\ldots)
    $$ 
    satisfies $\eta_i^+=1-\eta_i^-$ for every $i\in\mathbb{N}_0$.
\end{enumerate}
\end{lemma}
\begin{proof}
    \eqref{lem:symmetry1}: Note that $x<c$ if and only if $1-x>c$. Thus, we have 
    $$
    1-F_{\beta}(1-x)=1-\left(\beta (1-x)-\frac{\beta}{2}\right)=\beta x+1-\frac{\beta}{2}=F_{\beta}(x),\quad\text{for}\quad x\in[0,c),
    $$
    $$
    1-F_{\beta}(1-x)=1-\left(\beta (1-x)+1-\frac{\beta}{2}\right)=\beta x-\frac{\beta}{2}=F_{\beta}(x),\quad\text{for}\quad x\in(c,1].
    $$

    \eqref{lem:symmetry2}: Induction on $i\in\mathbb{N}_0$. The case $i=0$ is trivial. Assume that the equality $F_{\beta}^i(0_+)=1-F_{\beta}^i(1_-)$ is satisfied for some $i\in\mathbb{N}_0$ and consider two cases. If $F_{\beta}^i(0_+)\neq c$, then also $F_{\beta}^i(1_-)\neq c$ and using the condition~\eqref{lem:symmetry1} we get
    $$
    F_{\beta}^{i+1}(0_+)=F_{\beta}(F_{\beta}^{i}(0_+))=F_{\beta}(1-F_{\beta}^i(1_-))=1-F_{\beta}(F_{\beta}^i(1_-))=1-F_{\beta}^{i+1}(1_-).
    $$
    If $F_{\beta}^i(0_+)= c$, then $F_{\beta}^i(1_-)= c$ as well and
    $$
    F_{\beta}^{i+1}(0_+)=0=1-F_{\beta}^{i+1}(1_-).
    $$

    \eqref{lem:symmetry3}: By the condition~\eqref{lem:symmetry2} the inequality $F_{\beta}^i(0_+)<c$ is equivalent to $F_{\beta}^i(1_-)>c$. Moreover, $F_{\beta}^i(0_+)=c$ if and only if $F_{\beta}^i(1_-)=c$, hence $\eta_i^+=1-\eta_i^-$ for every $i\in\mathbb{N}_0$.
\end{proof}
For any sequence $\eta=\eta_0\eta_1\ldots\in\{0,1\}^{\mathbb{N}_0}$ we define its \textbf{mirror image} (or \textbf{conjugate}) as the sequence $\overline{\eta}$ generated by replacing all the $0$--s with $1$'--s and $1$'--s with $0$'--s in $\eta$, that is, $\overline{\eta}=\overline{\eta}_0\overline{\eta}_1\ldots\in\{0,1\}^{\mathbb{N}_0}$, where $\overline{\eta}_i=1-\eta_i$ for every $i\in\mathbb{N}_0$. It follows from Lemma~\ref{lem:symmetry}\eqref{lem:symmetry3} that for any $\beta\in(1,2]$ the kneading invariant of $F_{\beta}$ has the form $k_{F_\beta}=(\eta,\overline{\eta})$ for some $\eta\in\{0,1\}^{\mathbb{N}_0}$, and thus
$$
\Sigma_{F_{\beta}}=\left\{\xi\in\{0,1\}^{\mathbb{N}_0}\:|\: \sigma(\eta)\preccurlyeq\sigma^n(\xi)\preccurlyeq\sigma(\overline{\eta})\ \text{for every}\ n\in\mathbb{N}_0\right\}.
$$
Note that the sequences $\eta$, $\overline{\eta}$ are periodic if and only if $F_{\beta}^n(c_+)=F_{\beta}^n(c_-)=\frac{1}{2}$ for some $n\in\mathbb{N}$. 

Now, we are going to describe how the kneading invariant of $F_{\beta}$ changes when the parameter $\beta$ tends to $2$. For this purpose, we will introduce a sequence of parameters $\{\varepsilon_i\}_{i\in\mathbb{N}_0}\subset(1,2]$ related to the family of polynomials
\begin{equation}\label{eq:polynomials}
    Q_i(\beta)=\beta^{i+1}-2\beta^i+1=\beta^i(\beta-2)+1.
\end{equation}

\begin{lemma}\label{lem:polynomials}
    For each $i\geq2$ the polynomial $Q_i$ has in the interval $(1,2)$ a unique root $\varepsilon_i$.
\end{lemma}
\begin{proof}
We will prove the existence of the root of $Q_i$ by induction on $i\geq2$. For $i=2$, the point $\varepsilon_2:=\frac{1+\sqrt{5}}{2}\in(1,2)$ satisfies $Q_2(\varepsilon_2)=0$. Next, assume that for some $i\geq2$ the polynomial $Q_i$ has a root $\varepsilon_i\in(1,2)$. Note that the equality $Q_i(\varepsilon_i)=0$ implies $\varepsilon_i^i(\varepsilon_i-2)=-1$ and thus
$$
Q_{i+1}(\varepsilon_i)=\varepsilon_i^{i+1}(\varepsilon_i-2)+1=1-\varepsilon_i<0.
$$
Since $Q_{i+1}(2)=1$, by the Intermediate Value Theorem, the polynomial $Q_{i+1}$ has a root $\varepsilon_{i+1}\in(\varepsilon_i,2)\subset(1,2)$.

Let $i\geq2$ and $\beta\in(1,2)$. Observe that the following equivalences hold:
\begin{equation}\label{eq:root_equiv}
Q_i(\beta)=0\iff\beta^{i+1}-2\beta^i+1=0\iff\frac{1-\frac{1}{\beta^i}}{\beta-1}=1\iff\sum_{j=1}^i\frac{1}{\beta^j}=1.
\end{equation}
Moreover,
$$
\beta<\varepsilon_i\implies\sum_{j=1}^i\frac{1}{\beta^j}>\sum_{j=1}^i\frac{1}{\varepsilon_i^j}=1
$$
and
$$
\beta>\varepsilon_i\implies\sum_{j=1}^i\frac{1}{\beta^j}<\sum_{j=1}^i\frac{1}{\varepsilon_i^j}=1.
$$
Hence, we conclude that $\varepsilon_i$ is the unique root of $Q_i$ in $(1,2)$.
\end{proof}

Let $\varepsilon_1:=\sqrt{2}$ and for each $i\geq2$ we define $\varepsilon_i$ to be the unique root of the polynomial $Q_i$ contained in the interval $(1,2)$. By Lemma~\ref{lem:polynomials} the sequence $\{\varepsilon_i\}_{i\in\mathbb{N}}$ is well defined. Furthermore, using the equivalence~\eqref{eq:root_equiv} we conclude that $\{\varepsilon_i\}_{i\in\mathbb{N}}$ is strictly increasing and converges to $2$. The approximate values of its first few elements are
$$
\varepsilon_1=\sqrt{2}\approx 1.41421,\quad\varepsilon_2=\frac{1+\sqrt{5}}{2}\approx1.61803,\quad\varepsilon_3\approx1.83929,\quad\varepsilon_4\approx1.92756,\quad\varepsilon_5\approx1.96595,\quad\varepsilon_6\approx1.98358.
$$
We marked some of the corresponding points in the parameter space
\begin{equation}
    \label{eq:parameter_points_eps}
    (\varepsilon_i,\xi_i)=\left(\varepsilon_i,1-\frac{\varepsilon_i}{2}\right)\in\Delta,\quad i\in\mathbb{N},
\end{equation}
as gray dots in the Figure~\ref{fig:parameter_space}. Note that the sequence $\{\varepsilon_i\}_{i\in\mathbb{N}}$ determines a partition of the parameter range $\beta\in[\sqrt{2},2]$. We will show that the kneading invariants of the corresponding maps $F_{\varepsilon_i}$ have a simple periodic form, which consequently gives us useful upper and lower bounds on the symbolic dynamics of any map $F_{\beta}$ defined by $\beta\in(\varepsilon_i,\varepsilon_{i+1})$. Before that, let us recall that for $n\in\mathbb{N}$ any finite sequence $\omega\in\{0,1\}^n$ is called a \textbf{word of length $n$}. Moreover, for any words $\omega=\omega_0\omega_1\ldots\omega_{n-1}\in\{0,1\}^n$ and $\nu=\nu_0\nu_1\ldots\nu_{m-1}\in\{0,1\}^m$ we define their \textbf{concatenation} as
$$
\omega\nu:=\omega_0\omega_1\ldots\omega_{n-1}\nu_0\nu_1\ldots\nu_{m-1}\in\{0,1\}^{n+m},
$$
and for $p\in\mathbb{N}$ we denote
$$
(\omega)^p:=\underbrace{\omega\omega\ldots\omega}_\text{$p$ times}\in\{0,1\}^{np}\quad\text{and}\quad(\omega)^\infty:=\omega\omega\ldots\in\{0,1\}^{\mathbb{N}_0}.
$$

\begin{lemma}\label{lem:eps_knead}
    Let $\beta\in(1,2]$ and $i\geq2$. Then:
    \begin{enumerate}
    \item\label{lem:eps_knead1} the kneading invariant of the map $F_{\beta}$ has the form
    $
    k_{F_{\beta}}=((10^i)^\infty,(01^i)^\infty)
    $
    if and only if $\beta=\varepsilon_i$;
    \item\label{lem:eps_knead2} $\beta\in(\varepsilon_i,\varepsilon_{i+1})$ if and only if the kneading invariant $k_{F_{\beta}}=(\eta_+,\eta_-)$ of the map $F_{\beta}$ satisfies $$
    (10^{i+1})^\infty\prec\eta_+\prec(10^i)^\infty\quad\text{and}\quad(01^{i})^\infty\prec\eta_-\prec(01^{i+1})^\infty.
    $$
    \end{enumerate}
\end{lemma}
\begin{proof}
It is sufficient to consider the upper kneading sequence $\eta_+=\eta_+(c)$ of the critical point $c=\frac{1}{2}$, because statements about the lower one $\eta_-=\eta_-(c)$ follow from the Lemma~\ref{lem:symmetry}\eqref{lem:symmetry3}.

\eqref{lem:eps_knead1}: Observe that $\eta_+=(10^i)^\infty$ if and only if 
$$
0<F_{\beta}(0)<F_{\beta}^2(0)<\ldots<F_{\beta}^{i-1}(0)<F_{\beta}^i(0)=c,
$$
which in turn is equivalent to
\begin{equation}\label{eq:F_iterations}
0<1-\frac{\beta}{2}<\left(1-\frac{\beta}{2}\right)\left(\beta+1\right)<\ldots<\left(1-\frac{\beta}{2}\right)\left(\beta^{i-2}+\ldots+\beta+1\right)<\left(1-\frac{\beta}{2}\right)\left(\beta^{i-1}+\ldots+\beta+1\right)=\frac{1}{2}.
\end{equation}

Using the equivalence~\eqref{eq:root_equiv} we obtain
$$
\left(1-\frac{\beta}{2}\right)\left(\beta^{i-1}+\ldots+\beta+1\right)=\frac{1}{2}\iff2+\beta+\ldots+\beta^{i-1}-\beta^{i}=1\iff\sum_{j=1}^i\frac{1}{\beta^j}=1\iff Q_i(\beta)=0.
$$
Therefore, by Lemma~\ref{lem:polynomials} we conclude that $\eta_+=(10^i)^\infty$ implies $\beta=\varepsilon_i$.

On the other hand, if $\beta=\varepsilon_i$, then 
$$
\left(1-\frac{\beta}{2}\right)\left(\beta^{i-1}+\ldots+\beta+1\right)=\frac{1}{2},
$$
which implies condition~\eqref{eq:F_iterations}. Thus, $\eta_+=(10^i)^\infty$.

\eqref{lem:eps_knead2}: Let us start with the observation that for any parameters $\beta$, $\beta'\in(1,2]$ and point $x\in[0,c)$ the inequality $\beta<\beta'$ implies $F_{\beta'}(x)<F_{\beta}(x)$. Let $\beta\in(\varepsilon_i,\varepsilon_{i+1})$. It follows from \eqref{lem:eps_knead1} that $F^s_{\varepsilon_i}(0)<c$ for $s=1,\ldots,i-1$ and $F^i_{\varepsilon_i}(0)=c$. Since $\beta>\varepsilon_i$, we obtain by induction that
$$
F^s_{\beta}(0)=F_{\beta}(F^{s-1}_{\beta}(0))<F_{\beta}(F^{s-1}_{\varepsilon_i}(0))<F_{\varepsilon_i}(F^{s-1}_{\varepsilon_i}(0))=F^s_{\varepsilon_i}(0)<c
$$
for $s=1,\ldots,i-1$ and $F^i_{\beta}(0)<F^i_{\varepsilon_i}(0)=c$. Therefore, the sequence $\eta_+$ has form
$$
\eta_+=10^{i+1}\eta^+_{i+2}\eta^+_{i+3}\ldots\prec(10^i)^\infty.
$$
Using again the condition~\eqref{lem:eps_knead1} we conclude that $\eta_+\neq(10^{i+1})^\infty$. Suppose that $\eta_+\prec(10^{i+1})^\infty$. So we must have
$$
\eta_+=(10^{i+1})^m0\ldots\prec(10^{i+1})^\infty
$$
for some $m\in\mathbb{N}$. Similarly as before, based on the inequality $\beta<\varepsilon_{i+1}$ we get $F^s_{\varepsilon_{i+1}}(0)<F^s_{\beta}(0)<c$ for $s=1,\ldots,i$ and $c=F^{i+1}_{\varepsilon_{i+1}}(0)<F^{i+1}_{\beta}(0)$. In particular, it implies $m>1$ and consequently
$$
\sigma(\eta_+)=(0^{i+1}1)^{m-1}0^{i+2}\ldots\succ0^{i+2}\ldots=\sigma^{(i+2)(m-1)}(\sigma(\eta_+)).
$$
However, this contradicts the condition~\eqref{eq:kneading_cond}. We have therefore proven that $(10^{i+1})^\infty\prec\eta_+\prec(10^i)^\infty$.

To prove the second implication, assume that the kneading invariant $k_{F_{\beta}}=(\eta_+,\eta_-)$ satisfies the condition
$$   (10^{i+1})^\infty\prec\eta_+\prec(10^i)^\infty\quad\text{and}\quad(01^{i})^\infty\prec\eta_-\prec(01^{i+1})^\infty.
$$
Since $\eta_+\prec(10^i)^\infty$, the word $0^{i+1}$ must occur in the sequence $\eta_+$. Thus, the condition~\eqref{eq:kneading_cond} implies $\sigma(\eta_+)=0^{i+1}\ldots$, that is, $\sigma(\eta_+)$ starts with the word $0^{i+1}$. Hence, we have $F^s_{\beta}(0)<c$ for $s=0,\ldots,i$. Suppose that $\beta<\varepsilon_i$. But then an argument analogous to the one above gives $F^s_{\varepsilon_i}(0)<F^s_{\beta}(0)<c$ for $s=1,\ldots,i-1$ and $c=F^i_{\varepsilon_i}(0)<F^i_{\beta}(0)$, which leads to a contradiction. So $\beta>\varepsilon_i$.

Similarly, the inequality $\beta>\varepsilon_{i+1}$ implies $F^s_{\beta}(0)<F^s_{\varepsilon_{i+1}}(0)<c$ for $s=1,\ldots,i$ and $F^{i+1}_{\beta}(0)<F^{i+1}_{\varepsilon_{i+1}}(0)=c$, which is impossible due to the condition $(10^{i+1})^\infty\prec\eta_+$. Thus, we proved that $\beta\in(\varepsilon_i,\varepsilon_{i+1})$.
\end{proof}

Let us consider the subshifts 
$$
\Sigma_{F_{\varepsilon_i}}=\left\{\xi\in\{0,1\}^{\mathbb{N}_0}\:|\: (0^i1)^\infty\preccurlyeq\sigma^n(\xi)\preccurlyeq(1^i0)^\infty\ \text{for every}\ n\in\mathbb{N}_0\right\}
$$
corresponding to the family of maps $\{F_{\varepsilon_i}\}_{i\geq2}$. In \cite{BiSS} it was proven that subshifts of this form (i.e., induced by $\beta$-transformations satisfying $F^i_{\beta,\alpha}(0)=c=F^j_{\beta,\alpha}(0)$ for some $i$, $j\in\mathbb{N}$) are of finite type, which leads to many interesting properties. Each such subshift can be defined by a finite set of forbidden words and represented by a finite oriented graph, or a binary adjacency matrix. Moreover, subshifts of finite type are exactly those subshifts that have the shadowing property (see more details in \cite{Kurka}). By the result of \cite{BiSSS}, the set of parameters $(\beta,\alpha)$ for which the corresponding subshift $\Sigma_{F_{\beta,\alpha}}$ is of finite type (or, equivalently, $F^i_{\beta,\alpha}(0)=c=F^j_{\beta,\alpha}(0)$ for some $i$, $j\in\mathbb{N}$) is dense in the parameter space $\Delta$.

Observe that every time the parameter $\beta$ exceeds $\varepsilon_i$, for $i\geq2$, the periodic itinerary $(10^i)^\infty$ is added to the symbolic dynamics of $F_{\beta}$. As an immediate consequence of Lemma~\ref{lem:eps_knead}, we get the following result.

\begin{corollary}
\label{cor:kneading_iff}
    Let $i\geq2$. Then $\beta\in(\varepsilon_i,\varepsilon_{i+1})$ if and only if $\Sigma_{F_{\varepsilon_i}}\subset\Sigma_{F_{\beta}}\subset\Sigma_{F_{\varepsilon_{i+1}}}$. 
\end{corollary}

By Cor.\ref{reduct2}, if the kneading invariant for the Lorenz system at $v$ (sufficiently close to $T$) consists of non-periodic sequences $\omega_0$ and $\omega_1$, then it coincides with the kneading invariant $k_{F_{\beta}}$ of the corresponding factor map $F_{\beta}$. Note that there are methods that allow us to determine the parameter $\beta$ that defines the map $F_{\beta}$ with the given kneading invariant $k_{F_{\beta}}=(\omega_0,\omega_1)$ (see \cite{BarSV,DinSun}). However, in practice, applying them to obtain the exact value of $\beta$ may be difficult or even impossible, especially when the sequences $\omega_0$ and $\omega_1$ are not (eventually) periodic. Using Cor.\ref{cor:kneading_iff} is one of the possible ways to estimate the parameter $\beta$ for a given kneading invariant $k_{F_{\beta}}$, which only requires computing the roots of the appropriate polynomials \eqref{eq:polynomials}. On the other hand, it provides us with bounds on the symbolic dynamics of the map $F_{\beta}$, defined by a given parameter $\beta$. 

As we mentioned above, $\beta>\varepsilon_i$ implies $(10^i)^\infty\in\Sigma_{F_{\beta}}$. Our next aim is to describe the other periodic itineraries that appear when $\beta$ exceeds the values $\varepsilon_i$. To this end, let us recall the basics of rotation theory. The \textbf{rotation number} of $x\in[0,1]$ under a Lorenz map $F\colon[0,1]\to[0,1]$ with the critical point $c$ is defined as
$$
\rho_F(x):=\lim\limits_{n\to\infty}\frac{R_n(x)}{n},
$$
provided that this limit exists, where
$$
R_n(x):=\left| \left\{i\in\{0,1,\dots,n-1\}:F^i(x)\geq c\right\}\right|.
$$
Note that if $x$ is a $q$-periodic point of $F$, then the rotation number of $x$ exists and $\rho_F(x)=\frac{R_q(x)}{q}$. The set
$$
\text{Rot}(F):=\left\{\rho\in\mathbb{R}:\rho_F(x)=\rho\text{ for some }x\in[0,1]\right\}
$$
is called the \textbf{rotation interval} of $F$. It is known (see, e.g. \cite{Alseda,MiGel}) that if $F$ is an overlapping Lorenz map (i.e., $F(0)< F(1)$), then for some numbers $0\leq a\leq b\leq1$ we have $\text{Rot}(F)=[a,b]$. For every $s\in\mathbb{N}$ we denote $A_s:=0^s1$.

\begin{proposition}\label{prop:itineraries_rot}
    If $\beta>\varepsilon_i$ for some $i\geq2$, then the map $F_{\beta}$ has the following properties:
        \begin{enumerate}
     \item the rotation interval of $F_{\beta}$ contains $[\frac{1}{i+1},\frac{i}{i+1}]$;
     \item for $s\in\{2,\ldots,i\}$ and any finite concatenation $C$ of the words $A_s=0^{s}1$ and $A_{s-1}=0^{s-1}1$ there are periodic points $x$ and $x'$ of $F_{\beta}$ with the itineraries $\eta(x)=C^\infty$ and $\eta(x')=\overline{C^\infty}$.
    \end{enumerate}
\end{proposition}
\begin{proof}
Fix a natural number $i\geq2$. Let $\beta>\varepsilon_i$ and $k_{F_{\beta}}=(\eta_+,\eta_-)$ be the kneading invariant of the map $F_{\beta}$. Then, by Lemma~\ref{lem:eps_knead} we have $\eta_+\prec(10^i)^\infty$, so $(10^i)^\infty\in\Sigma_{F_{\beta}}$. Hence, there is a point $y\in[0,1]$ whose itinerary under the map $F_{\beta}$ is $\eta(y)=(10^i)^\infty$. Note that $y$ is a periodic point with a period $i+1$ and $R_{i+1}(y)=1$. Thus, $\rho_{F_{\beta}}(y)=\frac{1}{i+1}$. Similarly, there is a point $y'\in[0,1]$ with $\eta(y')=(01^i)^\infty$ and $\rho_{F_{\beta}}(y')=\frac{i}{i+1}$. Since $\frac{1}{i+1}$, $\frac{i}{i+1}\in\text{Rot}(F_{\beta})$, we conclude that $[\frac{1}{i+1},\frac{i}{i+1}]\subset\text{Rot}(F_{\beta})$.

Next, we will use the tools developed in \cite{MiGel} (see also \cite{BNPSR}). Let $s\in\{2,\ldots,i\}$ and observe that
$$
\left[\frac{s-1}{s},\frac{s}{s+1}\right]\subset\left[\frac{1}{i+1},\frac{i}{i+1}\right]\subset\text{Rot}(F_{\beta}).
$$
Thus the map $F_{\beta}$ has the so-called twist periodic orbits $P$ and $Q$ consisted of points
$$
p_1<p_2<\ldots<p_s\quad\text{and}\quad q_1<q_2<\ldots<q_{s+1}
$$
with rotation numbers $\rho_{F_{\beta}}(p_j)=\frac{s-1}{s}$ and $\rho_{F_{\beta}}(q_j)=\frac{s}{s+1}$. Observe that 
$$
\eta(p_1)=(01^{s-1})^\infty=(\overline{A_{s-1}})^\infty\quad\text{and}\quad\eta(q_1)=(01^{s})^\infty=(\overline{A_{s}})^\infty.
$$
Note also that the numbers $\frac{s-1}{s}<\frac{s}{s+1}$ are Farey neighbors because $s^2-(s-1)(s+1)=1$. According to the results of \cite{MiGel}, for every finite concatenation $\overline{C}$ of words $\overline{A_{s-1}}$ and $\overline{A_{s}}$ there is a periodic point $x'$ with the itinerary $\eta(x')=(\overline{C})^\infty=\overline{C^\infty}$. By the symmetry of $F_{\beta}$, the same is true for the words $A_{s-1}$ and $A_s$.
\end{proof}

\subsection{Topological and measurable dynamics}
\label{subsec:topmeasure}

By a topological dynamical system $(X,S)$ we mean a continuous map $S\colon X\to X$ acting on a compact metric space $X$. We say that $S$ is \textbf{transitive} if for every two nonempty open sets $U,V\subset X$ there is an integer $n>0$ such that $S^n(U)\cap V\neq\emptyset$. In the case where $X$ has no isolated points, it is equivalent to the existence of a point $x\in X$ with dense orbit. The map $S$ is \textbf{(topologically) mixing} if for every two nonempty open sets $U, V\subset X$ there is an $N>0$ such that for every $n>N$ we have $S^n(U)\cap V\neq\emptyset$. Moreover, the map $S$ is called \textbf{strongly transitive} if for every nonempty open set $U\subset X$ there is $n>0$ such that $\bigcup_{i=0}^n S^i(U)=X$. It is easy to see that any strongly transitive or mixing map is transitive, but the converse implications are generally not true.

Since an expanding Lorenz map $F\colon[0,1]\to[0,1]$ has a discontinuity $c$, formally it does not form a topological dynamical system on $[0,1]$. Thus, using the notations and tools from the topological dynamics becomes more delicate. One of the possible ways to deal with this obstacle is to use the standard doubling points construction that extends the map $F$ to a continuous map $\hat{F}$ acting on the Cantor set $\mathbb{X}$ (see e.g. \cite{Raith} for details). Briefly speaking, all elements in the set $(\bigcup_{i=0}^\infty F^{-i}(c))\setminus\{0,1\}$ are doubled (or "blown up") similarly, as it is done in the standard Denjoy extension of rotation on the circle (see \cite[Example~14.9]{Dev}). One can prove that the resulting topological dynamical system $(\mathbb{X},\hat{F})$ is topologically conjugate to the subshift $(\Sigma_F,\sigma_F)$. In particular, each point $x\in\mathbb{X}$ corresponds to exactly one kneading sequence in $\Sigma_F$. We call the expanding Lorenz map $F\colon[0,1]\to[0,1]$ transitive, mixing, or strongly transitive if its extension $\hat{F}\colon\mathbb{X}\to\mathbb{X}$ (or, equivalently, the subshift $\sigma_F\colon\Sigma_F\to\Sigma_F$) is transitive, mixing, or strongly transitive, respectively (cf. \cite{OPR,ChO}). Going back from the map $\hat{F}$ to an initial expanding Lorenz map $F$, these definitions transfer as follows:
\begin{itemize}
    \item $F$ is transitive if and only if for every two nonempty open sets $U,V\subset [0,1]$ there is an integer $n>0$ such that $F^n(U)\cap V\neq\emptyset$;
    \item $F$ is (topologically) mixing if for every two nonempty open sets $U, V\subset [0,1]$ there is an $N>0$ such that for every $n>N$ we have $F^n(U)\cap V\neq\emptyset$;
    \item $F$ is strongly transitive if for every nonempty open set $U\subset [0,1]$ there is $n>0$ such that the inclusion $(0,1)\subset\bigcup_{i=0}^n F^i(U)$ holds.
\end{itemize}
However, it should be emphasized that the properties that are inherent to continuous (strongly) transitive or mixing maps do not automatically carry over to the case of expanding Lorenz maps (some of them may be lost when passing from $\hat{F}$ to $F$).

Often in one-dimensional dynamics, the existence of a cycle (i.e., periodic orbit) of a given type (for instance, with a specific period or itinerary) forces the existence of a cycle of another type, or even determines the key properties of the system. A prime example of such a phenomenon is the famous Sharkovsky's Theorem. A similar characterization of periods for expanding Lorenz maps that can be derived from symmetric unimodal maps was obtained in \cite{ABC}. Since each map $F_{\beta}(x)=\beta x+1-\frac{\beta}{2}\Mod{1}$ can be constructed by "flipping" the right branch of the graph of symmetric unimodal map $U_{\beta}(x)=\beta\min\{x,1-x\}+1-\frac{\beta}{2}$ vertically around $c=\frac{1}{2}$, the result of \cite{ABC} applies to this case. Namely, the maps $F_{\beta}$ satisfy Sharkovsky's Theorem, except for the fixed points.

\begin{proposition}\label{prop:sharkovsky}
Let $\beta\in(1,2]$. If the map $F_{\beta}$ has a periodic point of prime period $n$, it also has a periodic point of prime period $m$ for every $1\prec_s m\prec_s n$ in the Sharkovsky order. 
\end{proposition}

It follows from Lemma~\ref{lem:eps_knead} that for any $\beta\geq\varepsilon_2$ the set $\Sigma_{F_{\beta}}$ contains the sequences $(001)^\infty$ and $(110)^\infty$. Therefore, the corresponding map $F_{\beta}$ has a periodic orbit of prime period $3$ (in the case $\beta=\varepsilon_2$, this orbit passes through the critical point $c=\frac{1}{2}$). By Prop.\ref{prop:sharkovsky} we get the following result.

\begin{corollary}\label{cor:all_periods}
    For any parameter $\beta>\frac{1+\sqrt{5}}{2}$, the map $F_{\beta}$ has periodic orbits of all periods (except for fixed points, which only appear for $\beta=2$).
\end{corollary}

In the study of periodic structure in Lorenz maps, certain special periodic orbits, the so-called \textbf{primary $n(k)$-cycles}, play a crucial role.
\begin{definition}\label{defn:nk_cycle}
Let $F\colon[0,1]\to[0,1]$ be an expanding Lorenz map, and $n$, $k$ be coprime integers with $n>k\geq1$. A periodic orbit $O=\{z_j : j\in\{0,\ldots, n-1\}\}$ of the map $F$ forms an $n(k)$-cycle if its points satisfy:
\begin{enumerate}
\item $z_0 < z_1 <\dots < z_{n-k-1} < c < z_{n-k}<\dots <z_{n-1}$;
	\item $F(z_j)=z_{j+k \Mod{n}}$ for all $j=0,1,\ldots,n-1$;

\end{enumerate}
If additionally
\begin{enumerate}\setcounter{enumi}{2}
	\item $z_{k-1}\leq F(0)$ and $F(1)\leq z_k$,
\end{enumerate}
then the $n(k)$-cycle $O$ is said to be primary.
\end{definition}
The notion of primary $n(k)$-cycles was first introduced by Palmer in \cite{Palmer} to study the weak Bernoulli property in Lorenz maps. In general, determining whether an expanding Lorenz map has such a cycle can be a difficult task. However, in the case of $\beta$-transformations, there are formulas that allow us to find regions in the parameter space $\Delta$ where the map $\beta x+\alpha\Mod{1}$ has a primary $n(k)$-cycle (\cite{OPR}, see also \cite{Palmer,Glen}). Some of these regions are marked in blue in Figure~\ref{fig:parameter_space}. The central, largest of blue regions corresponds to the parameters $(\beta,\alpha)\in\Delta$ for which $\beta x+\alpha\Mod{1}$ has a primary $2(1)$-cycle. Note that the parameters defining the family of maps $\{F_{\beta}\}_{\beta\in(1,2]}$ lie on the red dashed curve given by $\alpha=1-\frac{\beta}{2}$. Hence, the map $F_{\beta}$ has a primary $2(1)$-cycle if and only if $\beta\in(1,\sqrt{2}]$.
\begin{figure}[h]
\centering
\begin{overpic}[width=0.8\textwidth]{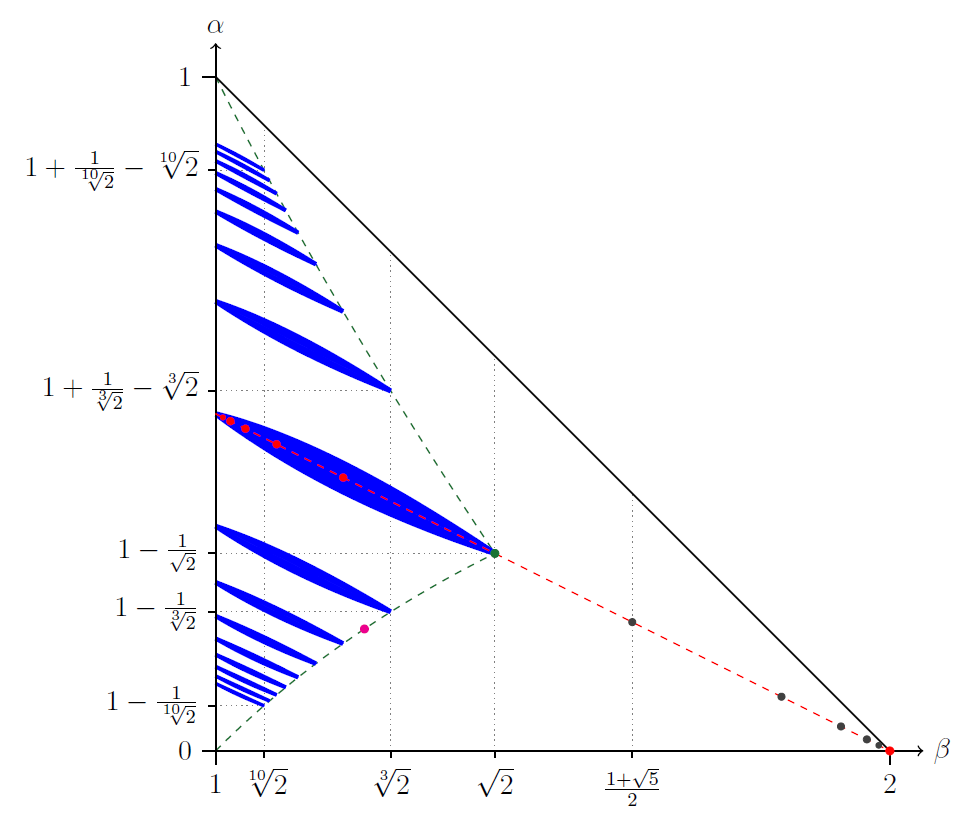}
\end{overpic}
\caption{\textit{The parameter space $\Delta$, where the blue regions correspond to the parameters $(\beta,\alpha)$ for which the map $\beta x+\alpha\Mod{1}$ has a primary $n(1)$- or $n(n-1)$-cycle for some $n\in\{2,3,\ldots,10\}$. The upper and lower green dashed curves are given by $\alpha=1+1/\beta-\beta$ and $\alpha=1-1/\beta$, respectively. The red dashed curve shows the graph of $\alpha=1-\beta/2$. The gray points belong to the sequence $(\varepsilon_i,\xi_i)$ defined by \eqref{eq:parameter_points_eps}, the red points to the sequence $(\beta_i,\alpha_i)$ defined by \eqref{eq:parameter_points}, and the green point belongs to both sequences. The pink dot corresponds to $\beta$--transformation $F$ from Appendix~\ref{ex:map_nonsymmetric1}.}}
\label{fig:parameter_space}
\end{figure}

The existence of a primary $n(k)$-cycle forces the occurrence of a certain periodic structure of the map, which is presented schematically in Fig.\ref{fig:21_cycle} and \ref{fig:nk_cycle} (for a $2(1)$-cycle and general $n(k)$-cycle, respectively). Such cycles were also used by Oprocha, Potorski, and Raith in \cite{OPR} to characterize (topologically) mixing expanding Lorenz maps. The proposition below is a consequence of applying their results to the family~\eqref{eq:family_maps}.
\begin{proposition}\label{prop:mixing_trans}
        The following conditions hold:
    \begin{enumerate}
        \item if $\beta\in(\sqrt{2},2]$, then the map $F_{\beta}$ is (topologically) mixing;
        \item if $\beta=\sqrt{2}$, then the map $F_{\beta}$ is transitive, but not mixing;
        \item if $\beta\in(1,\sqrt{2})$, then the map $F_{\beta}$ is not transitive.
    \end{enumerate}
\end{proposition}
By the result \cite[Theorem~4.12]{ChO}, every transitive and expanding Lorenz map is strongly transitive and thus, by \cite[Theorem~3.1.1]{Lectures}, has a dense set of periodic points.
\begin{corollary}
\label{cor:density}  If $\beta\in[\sqrt{2},2]$, then the map $F_{\beta}$ is strongly transitive and the set of its periodic points is dense in $[0,1]$.
\end{corollary}

In \cite{ChO}, it is shown that the rotation interval of an expanding Lorenz map with a primary $n(k)$-cycle degenerates to a singleton. Thus, for symmetric $\beta$--transformations with slopes $\beta\leq\sqrt{2}$ we cannot apply the tools from rotation theory, as in Prop.\ref{prop:itineraries_rot}, to describe "a lower bound" of their periodic dynamics. In this case, however, the existence of a primary $2(1)$-cycle imposes restrictions on the form of realized periodic itineraries, giving "an upper bound".

\begin{proposition}\label{prop:itineraries}
    If $\beta\in(1,\sqrt{2}]$, then the map $F_{\beta}$ has the following properties:
            \begin{enumerate}
                \item the rotation interval of $F_{\beta}$ degenerates to singleton $\{\frac{1}{2}\}$;
                \item for every periodic point $x$ of $F_{\beta}$ there is a finite concatenation $C$ of the words $A_1=01$ and $\overline{A_1}=10$ such that $\eta(x)\in\{C^\infty,0C^\infty,1C^\infty\}$.
            \end{enumerate}
        \end{proposition}
\begin{proof}
Let $\beta\in(1,\sqrt{2}]$. In this case, the map $F_{\beta}$ has a primary $2(1)$-cycle $O=\{z_0,z_1\}$. Thus, it follows from Theorem~5.1 in \cite{ChO} that $\text{Rot}(F_{\beta})=\{\frac{1}{2}\}$.
\begin{figure}[h]
	\centering
	\begin{tikzpicture}[scale=1.2]
	\draw[black, thick] (-5,0) -- (5,0);
	\draw[orange, very thick] (10*0.31-5,0) -- (0,0);
	\draw[blue, very thick] (0,0) -- (10*0.69-5,0);
	\draw[->] (5*0.31-2.5,0.1) to [out=30,in=150,looseness=1] (5*0.69,0.1);
	\draw[->] (5*0.69-2.5,0.1) to [out=150,in=30,looseness=1] (5*0.31-5,0.1);
	\draw[blue, very thick] (-5,0) -- (10*0.31-5,0);
	\draw[orange, very thick] (10*0.69-5,0) -- (5,0);
	\draw[blue, very thick] (10*0.4-5,-1) -- (10*0.69-5,-1);
	\draw[orange, very thick] (10*0.31-5,-0.5) -- (10*0.6-5,-0.5);
	\filldraw [black] (10*0.4-5,-1) circle (1.5pt) node[anchor=north] {\small$F_{\beta}(0)$};
	\filldraw [black] (10*0.69-5,-1) circle (1.5pt) node[anchor=north] {$z_1$};
	\filldraw [black] (10*0.6-5,-0.5) circle (1.5pt) node[anchor=north] {\small$F_{\beta}(1)$};
	\filldraw [black] (10*0.31-5,-0.5) circle (1.5pt) node[anchor=north] {$z_0$};
	\filldraw [red] (0,0) circle (1.5pt) node[anchor=north] {$c$};
	\filldraw [black] (5,0) circle (1.5pt) node[anchor=north] {$1$};
	\filldraw [black] (-5,0) circle (1.5pt) node[anchor=north] {$0$};
	\filldraw [black] (10*0.31-5,0) circle (1.5pt) node[anchor=north] {$z_0$};
	\filldraw [black] (10*0.69-5,0) circle (1.5pt) node[anchor=north] {$z_1$};
	\draw[->] (5*0.69,-0.1) to [out=210,in=0,looseness=1] (10*0.6-4.8,-0.5);
	\draw[->] (5*0.31-5,-0.1) to [out=330,in=180,looseness=1] (10*0.4-5.2,-1);
	\end{tikzpicture}
	\caption{\textit{Primary $2(1)$-cycle.}}
	\label{fig:21_cycle}
	\end{figure}
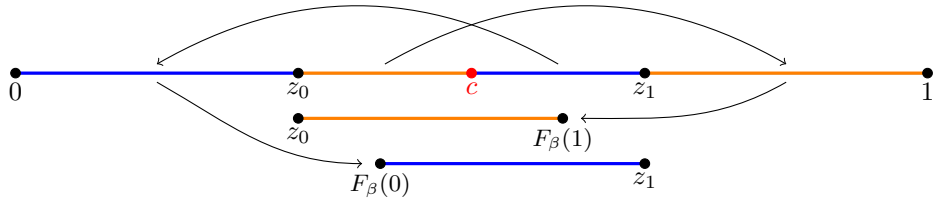
    
Observe that the itinerary of any point from the interval $[z_0,c)$ starts with $A_1$ and the itinerary of any point from $(c,z_1]$ starts with $\overline{A_1}$ (see Figure~\ref{fig:21_cycle}). Moreover, by Definition~\ref{defn:nk_cycle} we have 
$$
F_{\beta}([0,z_0])\subset[z_0,z_1],\quad F_{\beta}([z_1,1])\subset[z_0,z_1]\quad\text{and}\quad F^2_{\beta}([z_0,z_1])=[z_0,z_1].
$$
Let $x$ be a periodic point of $F_{\beta}$. From the above observations, we conclude that the itinerary of $x$ has the form
\begin{itemize}
    \item $\eta(x)=C^\infty$, if $x\in[z_0,z_1]$;
    \item $\eta(x)=0C^\infty$, if $x\in[0,z_0)$;
    \item $\eta(x)=1C^\infty$, if $x\in(z_1,1]$,
\end{itemize}
where $C$ is a finite concatenation of the words $A_1$ and $\overline{A_1}$.
\end{proof}

The results presented above indicate the complicated topological dynamics of the maps $F_{\beta}$ for $\beta$ close to $2$. We now use them to study the dynamics of the Lorenz equations - in detail, we are now going to "pull back" the obtained properties of the $\beta$--transformations to describe the Lorenz attractor. Recall that by Th.\ref{reduct} and Cor.\ref{reduct2}, for $v\in P$ sufficiently close to trefoil parameters, the dynamics on the attractor can be reduced to the symmetric $\beta$--transformation $F_{\beta}$ with $\beta=\beta(v)$ depending on $v$. We begin with the following Lemma:
\begin{lemma}
    \label{betato2} Let $T\subseteq P$ denote the closure of the set of all trefoil parameters in $P$, and let $v\in P$ be a parameter with a corresponding $\beta$--transformation $F_{\beta}$, $\beta=\beta(v)$. Then, as $v\to T$ we have $\beta\to2$.
\end{lemma}
\begin{proof}
Given a parameter $v\in P$, consider the map $\phi_v$ introduced in the proof of Th.\ref{reduct}, and assume $v\to p$, for some $p\in T$. By definition, $\phi_v:R_0\cup R_1\to R$ is a "straightened" version of the first-return map $\psi_v:R_0\cup R_1\to R$, in the sense that its symbolic dynamics and kneading are the "minimal possible" - or, put simply, they are defined precisely by the covering relations of $R_i$ or $R_j$, $i,j\in\{0,1\}$. As we push $v$ towards $p$ it is easy to see $\phi_v$ tends to a map as in the rightmost side of Fig.\ref{stretch} - i.e., a map which makes $R_i$ stretch all over $R_j$, s.t. $\phi_v(R_i)$ connects both $p_1,p_0$ (see the illustration in Fig.\ref{stretch}). This implies that once a periodic orbit is added for $\phi_v$, it cannot be removed.\\

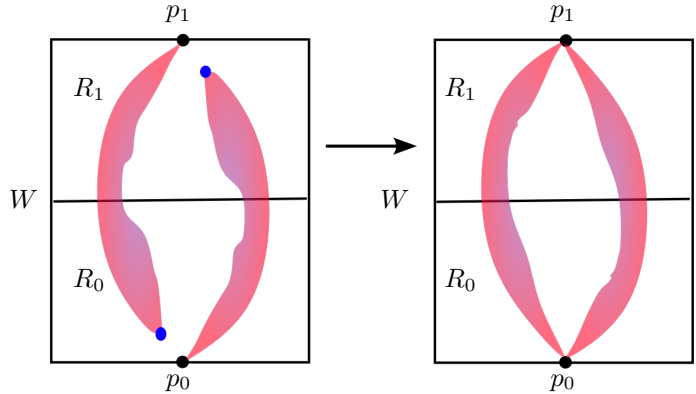
\begin{figure}[h]
\centering
\begin{overpic}[width=0.5\textwidth]{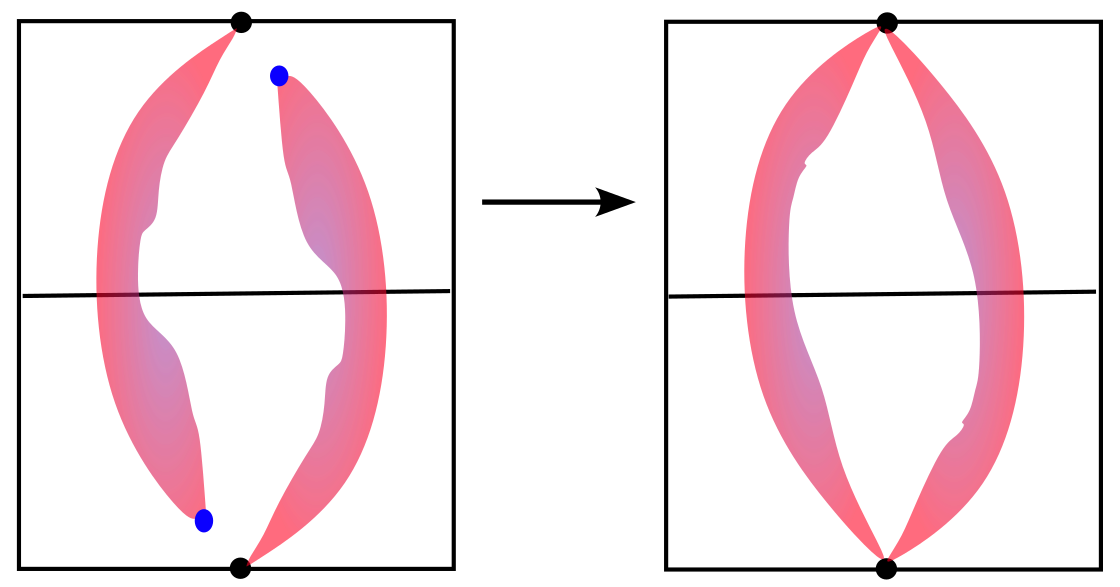}
\put(50,130){$R_0$}
\put(50,420){$R_1$}
\put(-60,250){ $W$}
\put(500,250){ $W$}
\put(190,-20){$p_0$}
\put(190,540){$p_1$}
\put(610,130){$R_0$}
\put(770,-20){$p_0$}
\put(610,420){$R_1$}
\put(770,540){$p_1$}
\end{overpic}
\caption{\textit{Isotoping $\phi_v$ to $\phi_p$, for some $p\in T$. As $v\to p$, the first hitting points of the separatrices (the blue dots) are pushed upwards towards the fixed points.}}
\label{stretch}
\end{figure}
By the proof of Th.\ref{reduct}, Cor.\ref{reduct2} we know $\beta$ depends on $v$. Denote by $I_v\subseteq R_0\cup R_1$ is the maximal invariant set for $\phi_v$ which is then deformed to $I_\beta$, the maximal invariant set of $F_\beta$ in $[0,1]\setminus\{\frac{1}{2}\}$ (this set exists by Cor.\ref{reduct2}). By Th.\ref{reduct} and Cor.\ref{reduct2} we know there exists a continuous, surjective map $\pi:I_v\to I_\beta$ s.t. for any $x\in I_\beta$, $\pi^{-1}(x)$ is connected. In other words, $\phi_v$ and $F_\beta$ have exactly the same symbolic dynamics. Therefore, letting $v\to p$ is like saying that the first hit points of the separatrices $\Gamma_0$ and $\Gamma_1$ hit $R_1$ and $R_0$ (respectively) closer and closer to $p_1$ and $p_0$. This implies that as $v\to p$, $R_i$ covers "more" of $R_j$, thus forcing the existence of more periodic orbits for $\psi_v$, the original first-return map which survive the transition to $\phi_v$. In particular, for any $i\in\mathbb{N}$ and $v$ sufficiently close to $T$, the corresponding map $F_{\beta}$ has a periodic orbit with itinerary $(10^i)^\infty$. Thus, by Lemma~\ref{lem:eps_knead} we have $\beta\geq\varepsilon_i$. Since the sequence $\{\varepsilon_i\}$ tends to $2$, by this discussion we conclude that as $v\to p$ we have $\beta\to2$, and the assertion follows.
\end{proof}

\begin{remark}
\label{convergence} Consider the kneading invariant $k_{F_\beta}=(\eta_+,\eta_-)$ of a map $F_\beta$. The argument used to prove Lemma~\ref{betato2} also implies as $v\to T$, we have $k_{F_\beta}\to(10^\infty,01^\infty)$, in the sense that each symbol in $\eta_+$ and $\eta_-$ is eventually constant and equal to the corresponding symbol in $10^\infty$ and $01^\infty$, respectively. In other words, the closer $v$ is to the collection of trefoil parameters $T$, the larger is the set of kneading sequences $\Sigma_{F_\beta}$ for the map $F_\beta$ corresponding to the Lorenz attractor at parameter values $v$.\end{remark}

As a corollary of the properties of symmetric $\beta$--transformations, we conclude the following:
\begin{corollary}
    \label{mixingattractor} Let $T\subseteq P$ denote the closure of the collection of trefoil parameters. Then, for any $v\in P$ sufficiently close to $T$ the corresponding Lorenz system satisfies the following:
    \begin{itemize}
        \item There exist infinitely many periodic orbits inside the attractor corresponding to $v$. Moreover, the first-return map for the attractor has periodic orbits of all periods.
        \item The dynamics of the first-return map for the attractor can be semiconjugated to a $\beta$--transformation that is both strongly transitive and mixing.
        \item  The dynamics at trefoil parameter are "essentially" dynamically maximal - if $v\not\in T$, there are infinitely many periodic sequences in $\{0,1\}^{\mathbb{N}_0}$ that \textbf{are not} realized as itineraries of periodic points of the corresponding $\beta$--transformation..

    \end{itemize}
\end{corollary}
\begin{proof}
    By Lemma \ref{betato2} we know that for $p$ sufficiently close to $T$ we have $\beta=\beta(v)\in(\frac{1+\sqrt{5}}{2},2]$. The first two assertion are immediate consequences of Th.\ref{reduct} (cf. Cor.\ref{reduct2}), Cor.\ref{cor:all_periods}, Prop.\ref{prop:mixing_trans} and Cor.\ref{cor:density}. Therefore, we need only prove the dynamics of the map $F_\beta$ corresponding to the Lorenz system at $v\notin T$ include less periodic orbits compared to the maps $F_{\beta'}$ corresponding to parameters $p\in T$. That, however, is immediate - by Lemma \ref{betato2} we know that for any parameter $p\in T$, its corresponding $F_{\beta'}$ per Th.\ref{reduct} is simply the doubling map $F_2(x)=2x\Mod{1}$, that is, $\beta'=2$. At this point we observe that the doubling map is dynamically maximal, i.e., it satisfies the following:
    \begin{itemize}
        \item We have $\Sigma_{F_2}=\{0,1\}^{\mathbb{N}_0}$ (see Def.\eqref{eq:symbol_space}). In particular, the collection of periodic orbits for the doubling map can be bijected to the collection of all the periodic sequences in $\{0,1\}^{\mathbb{N}_0}$.  
        \item For all $\beta<2$, the set $\Sigma_{F_{\beta}}$ is a proper subset of $\{0,1\}^{\mathbb{N}_0}$. Moreover, since $\frac{1+\sqrt{5}}{2}<\beta<2$ implies $\beta\in[\varepsilon_i,\varepsilon_{i+1})$ for some $i\geq2$, it follows from Lemma \ref{lem:eps_knead} that for any $j>i$ the map $F_\beta(x)=\beta x+1-\frac{\beta}{2}\Mod{1}  $ does not have periodic orbits with itineraries $(10^{j})^\infty$ or $(01^{j})^\infty$. Thus, the collection of periodic orbits for the $\beta$--transformation $F_\beta$ can only be embedded in the collection of periodic sequences in $\{0,1\}^{\mathbb{N}_0}$; it cannot be bijected with it - in other words, $F_\beta$ has less periodic orbits in $[0,1]$ compared to $F_2$.
    \end{itemize}
    
It is easy to see that given a parameter $v\not\in{T}$, as the separatrices $\Gamma_0$ and $\Gamma_1$ do not form heteroclinic trajectories then the corresponding $\beta$ transformation $F_\beta$ cannot be the doubling map. Or in other words, $\beta<2$, and the proof now follows by the discussion above.
\end{proof}
\begin{remark}
    One intuitive meaning of Cor.\ref{mixingattractor} is the following - if $v\not\in{T}$ is a parameter s.t. $\psi_v:R_0\cup R_1\to R$ can be semiconjugated to the double-sided shift, then the dynamics on the attractor are much more complex than those of the corresponding $F_\beta$, i.e., the "topological lower bound".
\end{remark}

Recall that by Th.\ref{tali} for all $v\in P$ the dynamics of the first return map $\psi_v:R_0\cup R_1\to R$ on its invariant set in $R_0\cup R_1$ is semiconjugated to some subshift $\Sigma_v\subseteq\{0,1\}^{\mathbb{N}_0}$. We already know that for any periodic sequence $s\in\{0,1\}^{\mathbb{N}_0}$ we have $s\in\Sigma_v$, provided that $v$ is sufficiently close to a trefoil parameter (see Cor.\ref{pers1}). Let $Per_v\subseteq \Sigma_v$ be the collection of all periodic sequences corresponding to the periodic orbits of $\psi_v$, which then survive as we deform it to $F_\beta$. Using Th.\ref{reduct} (cf. Remark \ref{rem:one-to-one} and Cor.\ref{reduct2}), Prop.\ref{prop:itineraries_rot} and Prop.\ref{prop:itineraries} we obtain bounds on periodic dynamics of $\psi_v$, depending on the slope $\beta$ of its one-dimensional factor map.
\begin{corollary}
    \label{symbolicper}
    Let $v\in P$ be sufficiently close to $T$ s.t. the dynamics on the attractor can be reduced to a $\beta$--transformation $F_{\beta}$, and let $\beta=\beta(v)\in(1,2]$ be the slope of $F_{\beta}$. Then, we have the following:
    \begin{itemize}
        \item The lower bound: If $\beta>\epsilon_i$ for some $i\geq2$, then for $s\in\{2,\ldots,i\}$ and every finite concatenation $C$ of words $A_s=0^s1$ and $A_{s-1}=0^{s-1}1$ both $C^\infty$ and $\overline{C^\infty}$ are in $Per_v$.
        \item The upper bound: If $\beta\in(1,\sqrt{2}]$ and $s\in Per_v$, then there is a finite concatenation $C$ of the words $A_1=01$ and $\overline{A_1}=10$ s.t. $s\in\{C^\infty,0C^\infty,1C^\infty\}$. 
    \end{itemize}
\end{corollary}

Cor.\ref{mixingattractor} and Cor.\ref{symbolicper} establish that the Lorenz attractor around the set $T$ has complex topological dynamics. We now address the analogous question of its measurable dynamics. To this end, we first recall the measurable properties of the symmetric $\beta$--transformations:
\begin{theorem}
    \label{measurprop} Given a $\beta$--transformation $F_\beta(x)=\beta x+1-\frac{\beta}{2}\Mod{1}  $, we have the following:
    \begin{itemize}
        \item The topological entropy of $F_\beta$ is $\ln(\beta)$.
        \item There exists a unique $F_\beta$--invariant probability measure $\mu$ s.t. the metrical entropy w.r.t. $\mu$ is also $\ln(\beta)$.
        \item $\mu$ is absolutely continuous w.r.t. the Lebesgue measure on $[0,1]$.
        \item The density function of $\mu$ is given by the formula $K\sum_{n\geq0}\beta^{-n}(\mathbb{1}_{[0,F^n_\beta(1)]}-\mathbb{1}_{[0,F^n_\beta(0)]})$ where $K$ is a normalizing factor and $\mathbb{1}_I$ denotes the indicator map over the interval $I$.
        \item The support of $\mu$ is a finite collection of intervals $I_1,...,I_n$ s.t. $\cup_{i=1}^n I_i=[0,1]$ whenever $\beta\in[\sqrt{2},2]$.
    \end{itemize}
\end{theorem}
For the proof, see \cite{HOF} (for more details on the measurable properties of $\beta$--transformations, see the survey of Section $6.2$ in \cite{BLS} and the results of \cite{DC}). As an immediate corollary of Lemma \ref{betato2} and Th.\ref{measurprop} we obtain:
\begin{corollary}
    \label{measurableattractor} Given $v\in P$, let $F_\beta$ denote the Lorenz map corresponding to it. Then, provided $v$ is sufficiently close to the set $T$, $F_\beta$ has an absolutely continuous invariant probability measure $\mu$ w.r.t. the Lebesgue measure on $[0,1]$. The support of $\mu$ is the whole interval $[0,1]$.
\end{corollary}

\begin{figure}[h]
\centering
\begin{overpic}[width=0.4\textwidth]{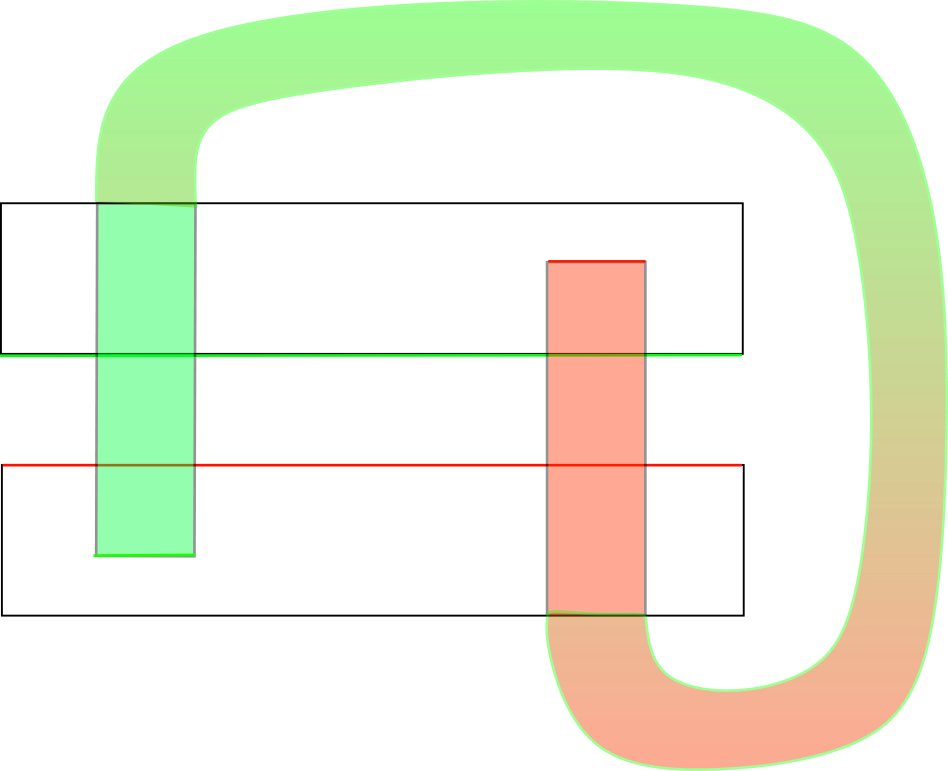}
\put(70,170){$h_v(C'D')$}
\put(-65,420){$C'$}
\put(-65,310){ $A'$}
\put(-50,600){$C$}
\put(790,600){$D$}
\put(790,150){$B$}
\put(-50,150){$A$}
\put(790,420){$D'$}
\put(530,550){$h_v(A'B')$}
\put(790,300){$B'$}
\end{overpic}
\caption{\textit{We isotope the map $h_v:ABCD\to\mathbb{R}^2$ above to a horseshoe map by stretching $h_v(A'B')$ upwards over $CD$ and by stretching $h_v(C'D')$ below $AB$. }}
\label{model1}
\end{figure}
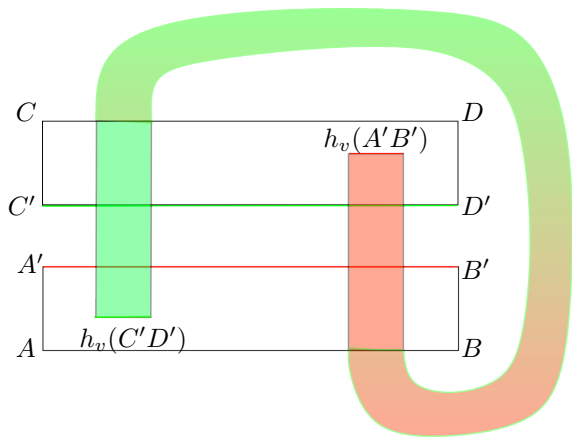

Cor.\ref{mixingattractor} and Cor.\ref{measurableattractor} together  teach us the dynamics on the Lorenz attractor at the vicinity of trefoil parameters are essentially those of a deformed, "blown up and stretched out" one-dimensional map whose dynamics can be described relatively well, both topologically and measurably. In particular, it proves the $\beta$--transformations serve as an idealized model for the Lorenz system, and serve as "topological lower bounds" for the possible complexity of the Lorenz attractor. We now extend this idea, by proving a result which compares the bifurcations of the $\beta$--transformations with those of the Lorenz attractor: 
\begin{theorem}
\label{expl} Every periodic orbit (except for the fixed points) of the symmetric $\beta$--transformation as we vary $\beta$ away from $2$ is destroyed via collision with the critical point $\frac{1}{2}$. Similarly, if $p$ is a trefoil parameter for the Lorenz system and $p$ is perturbed to some $v\not\in T$, the periodic orbits given by Th.\ref{tali} can only be destroyed via collapsing into a pair of homoclinic trajectories to the origin.
\end{theorem}
\begin{proof}
We first consider the maps $h_v$ introduced in the proof of Th.\ref{reduct} (see the illustration in Fig.\ref{deformation3}).  We begin by extending $h_v$ to a diffeomorphism of a rectangle $h_v:ABCD\to\mathbb{R}^2$, as depicted in Fig.\ref{model1}. Now, let us define a $C^1$--isotopy $h_t:ABCD\to\mathbb{R}^2$, $t\in[0,1]$ by pulling the images of the $A'B'$ and $C'D'$ arcs upwards s.t. $h_0=h_v$ and $h_1$ is a Horseshoe map, where $A'B'$ and $C'D'$ are as illustrated in Fig.\ref{model1}.

Now, let $x\in ABCD$ be a periodic point of minimal period $n$, which is destroyed as $h_1$ is deformed back to $h_0=h_v$. Using similar arguments to those used to prove Prop.\ref{persistence}, it is easy to see we can enclose every periodic orbit for $h_1$ in a neighborhood $V$ defined as in Prop.\ref{persistence}, s.t. the Fixed Point Index of $h^n_1$ on $V$ is either $-1$ or $1$ and for all $1\leq j<n$ we have $h^j_1(V)\cap V=\emptyset$. As we perturb $h_1$ to $h_t$, $t\in[0,1]$, it is also easy to see that for $x$ to be destroyed or to change its period, the Fixed Point Index in $V$ must first change. For this to happen, $x$ must collide with $\partial V$ first - and since $\partial V$ is composed of the preimages of the boundary of the middle rectangle $R_2$, this is the same as saying $x$ can only be destroyed by colliding with $\partial R_2$. We now project this isotopy to the maps $f_r$, and from them to the $\beta$--transformations $F_\beta(x)=\beta x+1-\frac{\beta}{2}\Mod{1}  $. It is easy to see that varying the parameter $t$ between $(0,1)$ is akin to varying $\beta$ between $(\gamma,2]$, where $\gamma=\gamma(v)$ corresponds to the parameter $v$. By the discussion above, it proves the periodic orbits of $F_2$ can be destroyed only by colliding with the critical discontinuity point.

We now "lift" the isotopy $h_t:ABCD\to\mathbb{R}^2$ to the isotopy of the first-return maps $\psi_v:{R_0}\cup {R_1}\to R$, $v\in P$. As we do that, the set $\partial V$ is transformed into arcs and curves in the set $\cup_{n\geq1}\psi^{-n}_v(W)$, where $W$ denotes the intersection $W^s(0)\cap R$ (see the illustration in Fig.\ref{cross}). Similarly to the proof of Prop.\ref{persistence} we know that for as long as $x$ persists, the Fixed Point Index on $V$ has to be $-1$ - which, using similar arguments, proves the only way the periodic dynamics can be destroyed is by collision with the set $\cup_{n>0}\psi_v^{-n}(W)$ - or in other words, by collision with the two-dimensional invariant manifold, $W^s(0)$ (see the illustration in Remark \ref{homoc}). By the symmetry of the Lorenz system, this implies that at the collision two homoclinic trajectories are formed and the assertion follows.
\end{proof}
\begin{remark}
    Let $s\in\{0,1\}^{\mathbb{N}_0}$ be periodic of minimal period $n$ and let $\overline{s}$ be its conjugate. It is easy to see by the symmetry of the Lorenz system that the periodic orbits corresponding to $s$ and $\overline{s}$ can only collapse together to a pair of homoclinic orbits. Similarly, the symmetry of the maps $F_\beta$ also force the corresponding periodic orbits to $s$ and $\overline{s}$ in $[0,1]$ to collapse together to the critical point $\frac{1}{2}$. From the point of view of symbolic dynamics, when this happens the kneading invariant $k_{F_\beta}=(\eta_+,\eta_-)$ of the map $F_\beta$ coincide with the sequences $(\sigma^i(s))^\infty$ and $(\sigma^i(\overline{s}))^\infty$, for some $0\leq i<n$.
\end{remark}
\begin{figure}[h]
\centering
\begin{overpic}[width=0.4\textwidth]{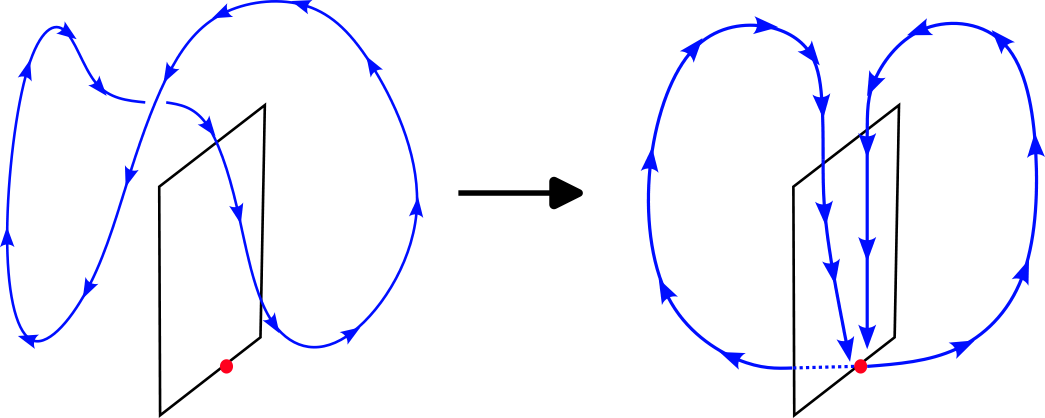}

\end{overpic}
\label{homoc}
\end{figure}
Before moving on, we remark the existence of complex homoclinic scenarios for the Lorenz system is well-known numerically for quite some time. As observed numerically in \cite{BSS}, the first $T$--point at parameter values $v_0=(\beta,\sigma,\rho)=(\frac{8}{3},10.2,30.38)$ is the accumulation point of homoclinic bifurcation sets of the origin (see Fig.8.B in \cite{BSS}). In light of this numerical evidence, Th.\ref{expl} might serve as a possible analytic explanation for such phenomena. Moreover, the similarity between these collision with the critical points and homoclinic bifurcations of the Lorenz attractor was already recognized in \cite{GS}.

Before concluding this section we remark that as proven in \cite{KY3}, in general one should expect complex dynamics to arise by period-doubling cascades. At first, this may appear to contradict Th.\ref{expl} - as the said Theorem implies the chaotic dynamics in the Lorenz attractor develop through successive homoclinic bifurcations accumulating on the parameter set $T$. We will not continue discussing this phenomenon in this section, but rather defer it to Appendix \ref{compressing}. As we will see there, one could intuitively think of the homoclinic bifurcations of the Lorenz attractor when $v\to T$, $v\in P$ as "compressed" period-doubling cascades. In detail, we will show the first-return map $\psi_v:R_0\cup R_1\to R$ can be extended to a continuous rectangle map $h_v:ABCD\to\mathbb{R}^2$, whose transition into horseshoe chaos as $v\to T$ is exactly as dictated by \cite{KY3} (in fact, $h_v$ would be essentially the same map as in the proof of Th.\ref{reduct}). For the details and precise formulation, see Appendix \ref{compressing}.

\subsection{Renormalizating the Lorenz attractor}
\label{renorsec}
Th.\ref{expl} establishes that many of the interesting properties of the Lorenz attractor in the parameter range $P$ are merely consequence of its one-dimensional reduction - or, in other words, the $\beta$--transformations. It is easy to see any $\beta$--transformation can be deformed to a hyperbolic homeomorphism which has the same symbolic dynamics - therefore, in the spirit of the Chaotic Hypothesis (see \cite{gal}) we are led to ask the following - when can we treat the Lorenz attractor as an "essentially hyperbolic flow"? Assuming we can answer this question, we will be able to import tools from the theory of Uniform Hyperbolicity to study the dynamics of the Lorenz system. By the results of \cite{Pi}, we know this is the case whenever the Lorenz system generates a heteroclinic trefoil - in other words, at such parameters the dynamics on the attractor are essentially a deformed version of the geometric Lorenz attractor. We now prove the same is true in another class of parameters $v\in P$ - but in order to do so we will need to first investigate the renormalization properties of the expanding Lorenz maps.

Although we have already defined the Lorenz map as a map on the interval $[0,1]$, this definition easily extends to any interval $[u,v]$: we say that an interval map $F\colon[u,v]\to[u,v]$ is a Lorenz map on $[u, v]$ if the composition $h_{u,v}\circ F\circ h^{-1}_{u,v}\colon[0,1]\to[0,1]$ (see Def.\eqref{eq:rescaling}) is a Lorenz map on $[0,1]$. Following \cite{Ding} and \cite{DiCui}, we present the next definition.

\begin{definition}
    \label{defn:renormalization}
Let $F\colon[0,1]\to[0,1]$ be a Lorenz map with critical point $c$. If there is a proper subinterval $(u, v) \ni c$ of $(0,1)$ and integers $l, r > 1$ such that the map $G\colon [u, v]\to [u, v]$ defined by
	\begin{equation*}
	G(x)=\begin{cases}F^l (x),\,&\text{if 
		$x\in\left[u,c\right)$}\\
	F^r (x),\,&\text{if $x\in [c,v]$}\end{cases},
	\end{equation*}
is itself a Lorenz map on $[u, v]$, then we say that $F$ is \textbf{renormalizable} or that $G$ is a \textbf{renormalization} of $F$. The interval $[u, v]$ is called the \textbf{renormalization interval}. 
\end{definition}
From the definition of a Lorenz map, it follows that the renormalization interval $[u, v]$ must be equal to $[F^r(c_+),F^l(c_-)]$. Therefore, the numbers $l$ and $r$ determine the renormalization interval $[u, v]$ and the notation $G=(F^l,F^r)$ uniquely describe the renormalization $G$. We also note that the renormalization $G=(F^l,F^r)$ is a return map of the map $F$ to a smaller interval around the critical point $c$, with return times $l$ and $r$ for the left- and right-sided neighborhoods of $c$, respectively. 

\begin{example}\label{ex:renorm}
Consider the map $F_{\beta}$ defined by $\beta=\sqrt{2}$, that is $F_{\sqrt{2}}(x)=\sqrt{2} x+1-\frac{\sqrt{2}}{2}\Mod{1}$. The graphs of map $F_{\sqrt{2}}$ and its second iterate $F_{\sqrt{2}}^2$ are presented in Figure~\ref{fig:renormalization}. Observe that $F_{\sqrt{2}}^2$, restricted to the interval $[u,v]$, is also a Lorenz map. Hence, the map $G=(F_{\sqrt{2}}^2,F_{\sqrt{2}}^2)$ is a renormalization of $F_{\sqrt{2}}$, which after rescaling to $[0,1]$ is the doubling map $F_{2}(x)=2 x\Mod{1}$.
\begin{figure}[h]
\centering
\begin{overpic}[width=0.8\textwidth]{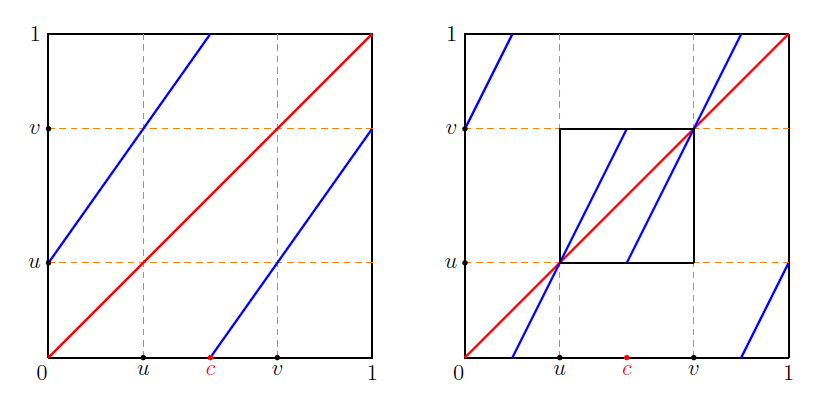}
\end{overpic}
\caption{\textit{The graphs of map $F_{\sqrt{2}}$ (on the left) and its second iterate $F_{\sqrt{2}}^2$} (on the right).} 
\label{fig:renormalization}
\end{figure}
\end{example}

Renormalizations in Lorenz maps were investigated from a symbolic perspective in \cite{GS} and \cite{GlenHall}, where they were used, among others, to describe the decomposition of a nonwandering set and study the topological entropy of these maps. Note that by the definition, a renormalization $G$ after appropriate rescaling of the domain is again a Lorenz map on the interval $[0,1]$, which may or may not be renormalizable. In the first case, we can repeat the renormalization process - in this sense, the map $F$ may be renormalizable multiple times. This process can be seen as a kind of "zooming" both in the dynamics of a given map and in the parameter space. While there exist Lorenz maps for which the renormalization process does not end (i.e., infinitely renormalizable maps; see \cite{Winckler}), in the case of $\beta$-transformations, after finitely many steps we always get a map which is no longer renormalizable. This last claim is an immediate consequence of the expanding condition: if $F$ is a Lorenz map with $\gamma:=\inf_{x\not\in E_0} F'(x)>1$ and $G$ is a renormalization of $F$, then $\inf_{x\not\in E_1} G'(x)\geq\gamma^n$ for some $n>1$, where $E_0$, $E_1$ are finite sets. However, the slope of a Lorenz map cannot exceed $2$ for all points outside a finite set.

We will use renormalizations to describe the periodic structure in Lorenz maps. By Theorem~3.5(1) in \cite{ChO}, every expanding Lorenz map $F$ with a primary $n(k)$-cycle has a renormalization $G=(F^n,F^n)$. The following lemma establishes the relation between the sets of periods of the map $F$ and its renormalization $G$.

\begin{lemma}\label{lem:renorm}
    Let $F\colon[0,1]\to[0,1]$ be an expanding Lorenz map with a primary $n(k)$-cycle. Then the set of periods $P_F$ of the map $F$ is described by the following formula
    $$
    P_F=n\cdot P_G\cup\left\lbrace n \right\rbrace=\left\lbrace n\cdot q\:|\: q\in P_G \right\rbrace\cup\left\lbrace n\right\rbrace,
    $$
    where $P_G$ is the set of periods of the renormalization $G=(F^n,F^n)$.
\end{lemma}
\begin{proof} 
Let $O=\{z_j : j\in\{0,\ldots, n-1\}\}$ be a primary $n(k)$-cycle of the map $F$. Since $F$ is expanding, we have $F(0)<F(1)$. By Definition~\ref{defn:nk_cycle} we obtain 

$$
F^n([z_{n-k-1},z_{n-k}])=F^{n-2}([z_{k-1},z_{k}])=[z_{n-k-1},z_{n-k}]\quad\text{and}\quad(0,1)\subset\bigcup_{i=0}^{n-1}F^i([z_{n-k-1},z_{n-k}])
$$
(see Figure~\ref{fig:nk_cycle}). In particular, the trajectory of every point $x\in[0,1]$ intersects the interval $[z_{n-k-1},z_{n-k}]$, that is
\begin{equation}
    \label{eq:orb_nk}
    \text{Orb}_{F}(x)\cap[z_{n-k-1},z_{n-k}]\neq\emptyset.
\end{equation}
Furthermore, the renormalization $G$ is defined on the interval $[u,v]=[F^{n-1}(0),F^{n-1}(1)]$, which is contained in $[z_{n-k-1},z_{n-k}]$.
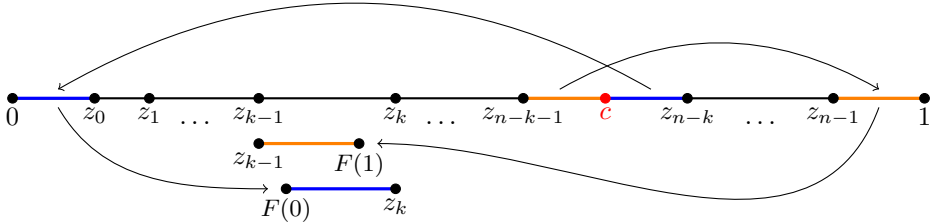
\begin{figure}[h]
		\centering
		\begin{tikzpicture}[scale=1.2]
		\draw[black, thick] (-5,0) -- (5,0);
		\draw[orange, very thick] (0.6,0) -- (1.5,0);
		\draw[blue, very thick] (1.5,0) -- (2.4,0);
		\draw[blue, very thick] (-5,0) -- (-4.1,0);
		\draw[orange, very thick] (4,0) -- (5,0);
		\draw[blue, very thick] (-2,-1) -- (-0.8,-1);
		\draw[orange, very thick] (-2.3,-0.5) -- (-1.2,-0.5);
		\filldraw [red] (1.5,0) circle (1.5pt) node[anchor=north] {$c$};
		\filldraw [black] (5,0) circle (1.5pt) node[anchor=north] {$1$};
		\filldraw [black] (-5,0) circle (1.5pt) node[anchor=north] {$0$};
		\filldraw [black] (-4.1,0) circle (1.5pt) node[anchor=north] {$z_0$};
		\filldraw [black] (-3.5,0) circle (1.5pt) node[anchor=north] {$z_1$};
		\draw[white] (-3.1,-0.15) -- (-3,-0.15) node[black,anchor=north] {$\ldots$};
		\filldraw [black] (-2.3,0) circle (1.5pt) node[anchor=north] {$z_{k-1}$};
		\filldraw [black] (-0.8,0) circle (1.5pt) node[anchor=north] {$z_k$};
		\draw[white] (-0.4,-0.15) -- (-0.3,-0.15) node[black,anchor=north] {$\ldots$};
		\filldraw [black] (0.6,0) circle (1.5pt) node[anchor=north] {$z_{n-k-1}$};
		\filldraw [black] (2.4,0) circle (1.5pt) node[anchor=north] {$z_{n-k}$};
		\draw[white] (3.1,-0.15) -- (3.2,-0.15) node[black,anchor=north] {$\ldots$};
		\filldraw [black] (4,0) circle (1.5pt) node[anchor=north] {$z_{n-1}$};
		\filldraw [black] (-2,-1) circle (1.5pt) node[anchor=north] {\small $F(0)$};
		\filldraw [black] (-0.8,-1) circle (1.5pt) node[anchor=north] {$z_k$};
		\filldraw [black] (-1.2,-0.5) circle (1.5pt) node[anchor=north] {\small $F(1)$};
		\filldraw [black] (-2.3,-0.5) circle (1.5pt) node[anchor=north] {$z_{k-1}$};
		\draw[->] (1,0.1) to [out=30,in=150,looseness=1] (4.5,0.1);
		\draw[->] (2,0.1) to [out=150,in=30,looseness=1] (-4.5,0.1);
		\draw[->] (4.5,-0.1) to [out=250,in=0,looseness=1] (-1,-0.5);
		\draw[->] (-4.5,-0.1) to [out=300,in=180,looseness=1] (-2.2,-1);
		\end{tikzpicture}
		\caption{\textit{Primary $n(k)$-cycle.}}
		\label{fig:nk_cycle}
	\end{figure}

We will show that for every $x\in[0,1]\setminus O$ we have $\text{Orb}_{F}(x)\cap[u,v]\neq\emptyset$. By the condition~\eqref{eq:orb_nk} it is sufficient to show the claim for $x\in(z_{n-k-1},u)\cup(v,z_{n-k})$. Assume that $z_{n-k-1}<x<u$ (the case $v<x<z_{n-k}$ is dealt similarly) and denote
$$
J_G:=\left\lbrace x\in [0,1]\:|\: \text{Orb}_F(x)\cap \left(u, v\right) =\emptyset\right\rbrace .
$$
Note that $z_{n-k-1}\neq u=F^{n-1}(0)$ implies $z_{k-1}\neq F(0)$. If additionally $z_k\neq F(1)$, then by Theorem~3.5(4) in \cite{ChO} we get
$$
z_{n-k-1} =\sup\left\lbrace x\in J_G \:|\: x<c\right\rbrace
$$
and consequently $\text{Orb}_F(x)\cap(u,v)\neq\emptyset$. Thus, consider the case $z_k=F(1)$ and suppose $\text{Orb}_F(x)\cap(u,v)=\emptyset$. Then $z_{n-k}=F^{n-1}(1)=v$, so $c\notin F^i((z_{n-k-1},x))$ for $i=0,1,\ldots,n-1$ and
$$
F^n((z_{n-k-1},x))=(z_{n-k-1},F^n(x))\subset(z_{n-k-1},u).
$$
Since the trajectory of $F^n(x)$ does not intersect $(u,v)$, repeating the reasoning we obtain $c\notin F^i((z_{n-k-1},x))$ for any $i\in\mathbb{N}_0$, which contradicts the topological expanding condition for $F$. Hence $\text{Orb}_{F}(x)\cap(u,v)\neq\emptyset$.

Now, let $m\in P_F$ and $\widetilde{O}$ be a periodic orbit of $F$ with period $m$. Assume that $m\neq n$, so $\widetilde{O}\neq O$. Hence $\widetilde{O}\subset[0,1]\setminus O$ and there is a point $y\in\widetilde{O}\cap[u,v]$. Observe that every point from the interval $[z_{n-k-1},z_{n-k}]$ returns to it after exactly $n$ iterations. Therefore $y\in[u,v]\subset[z_{n-k-1},z_{n-k}]$ and $F^m(y)=y$ implies that $n$ divides $m$. Let $q$ be a natural number such that $m=nq$. Then, we have
$$
G^q(y)=(F^n)^q(y)=F^m(y)=y
$$
and the point $y$ is $q$-periodic for $G$. In particular, $m\in n\cdot P_G$.

On the other hand, if $y$ is a periodic point of $G$ with period $q\in P_G$, then
$$
y=G^q(y)=F^{n q}(y)
$$
and $F^i(y)\neq y$ for $i=0,1,\ldots, n q-1$. As a consequence, the point $y$ is $nq$-periodic for $F$ and $nq\in P_F$, which completes the proof.
\end{proof}

From a symbolic perspective, for an expanding Lorenz map $F$ with a primary $n(k)$-cycle $\{z_j \}_{j=0}^{n-1}$, the first $n$ symbols of the kneading sequence (or sequences) of all points $x\in[z_{n-k-1},c)\cup(c,z_{n-k}]$ are equal to 
$$
L:=\eta(z_{n-k-1})_0\eta(z_{n-k-1})_1\ldots\eta(z_{n-k-1})_{n-1},
$$
if $x\in[z_{n-k-1},c)$, and
$$
R:=\eta(z_{n-k})_0\eta(z_{n-k})_1\ldots\eta(z_{n-k})_{n-1},
$$
if $x\in(c,z_{n-k}]$. For the critical point $c$, the kneading sequence $\eta_+=\eta_+(c)$ starts with $R$ and $\eta_-=\eta_-(c)$ with $L$. Since $F^n([z_{n-k-1},z_{n-k}])=[z_{n-k-1},z_{n-k}]$, we conclude that the kneading invariant $k_F=(\eta_+,\eta_-)$ can be expressed in the following way:
\begin{equation*}
			\begin{cases}
				\eta_+=RL^\infty\ \ \
				\text{ or }\ & 
				\eta_+=R^{n_1}L^{n_2}R^{n_3}L^{n_4}\ldots,\\
				\eta_-=LR^\infty\ \ \
				\text{ or }\ &
				\eta_-=L^{m_1}R^{m_2}L^{m_3}R^{m_4}\ldots,
			\end{cases}
\end{equation*}
for some sequences $\{n_i\}_{i=1}^\infty$, $\{m_i\}_{i=1}^\infty\subset\mathbb{N}$. The kneading invariants of the above form are called \textbf{renormalizable} (see \cite{Glen,GS}).

With the above notation, we obtain the following generalization of Prop.\ref{prop:itineraries} as a consequence of \cite[Theorem~5.1]{ChO} and the proof of Lemma \ref{lem:renorm}:

\begin{proposition}\label{prop:itineraries_general}
    Let $F\colon[0,1]\to[0,1]$ be an expanding Lorenz map with a primary $n(k)$-cycle. Then:
            \begin{enumerate}
                \item the rotation interval of $F$ degenerates to singleton $\{\frac{k}{n}\}$;
                \item for every periodic point $x$ of $F$ there is a number $i\in\{0,1,\dots,n-1\}$ and a finite concatenation $C$ of the words $L$ and $R$ such that $\sigma^i(\eta(x))=C^\infty$.
            \end{enumerate}
        \end{proposition}

Since any $\beta$-transformation with slope greater than $\sqrt{2}$ is not renormalizable and for $\beta\leq\sqrt{2}$ the map $F_{\beta}$ has a primary $2(1)$-cycle, we conclude that $F_{\beta}$ is renormalizable if and only if $\beta\in(1,\sqrt{2}]$. If $F_{\beta}$ is renormalizable, then it has (at least) renormalization $G_{\beta}:=(F^2_{\beta},F^2_{\beta})$. The maps $F_{2}$ and $F_{\sqrt{2}}$ from Example~\ref{ex:renorm} are the first two elements of a more general sequence of maps that will be of particular interest for our analysis. That is, the sequence $\{F_{\beta_i}\}_{i\in\mathbb{N}_0}$ defined by the parameters $\beta_i:=\sqrt[2^i]{2}$ for $i\in\mathbb{N}_0$. The corresponding points in the parameter space $\Delta$,
\begin{equation}
    \label{eq:parameter_points}
    (\beta_i,\alpha_i)=\left(\sqrt[2^i]{2},1-\frac{\beta_i}{2}\right)=\left(\sqrt[2^i]{2},1-\frac{\sqrt[2^i]{2}}{2}\right),\quad i\in\mathbb{N}_0,
\end{equation}
are marked as red dots in the Figure~\ref{fig:parameter_space}. Note that the sets
$$
\Delta_i:=\left\lbrace\left(\beta,1-\frac{\beta}{2}\right)\in\Delta\:|\:\beta\in(\beta_{i+1},\beta_i]\right\rbrace,\quad i\in\mathbb{N}_0,
$$
form a partition of the curve $\alpha=1-\frac{\beta}{2}$, $\beta\in(1,2]$.

\begin{proposition}
\label{prop:beta_seq}
    Let $\beta\in(\beta_{i+1},\beta_i]$ for some $i\in\mathbb{N}$. Then the following conditions hold.
    \begin{enumerate}
    \item\label{prop:beta_seq1} The map $F_{\beta}$ has $i$ renormalizations 
    $$
    G_{\beta,j}=(F^{2^j}_{\beta},F^{2^j}_{\beta}),\quad\text{for}\quad j=1,\ldots,i,
    $$ 
    that after rescaling to $[0,1]$ are equal to the maps $F_{\beta^{2^j}}$ defined by the parameters $\beta^{2^j}\in(\beta_{i+1-j},\beta_{i-j}]$.     
    \item\label{prop:beta_seq2} The set of periods of $F_{\beta}$ is given by
    $$
    P_{F_{\beta}}=2\cdot P_{F_{\beta^{2}}}\cup\{2\}.
    $$
    Moreover, if $i\geq2$, then
    $$
    P_{F_{\beta}}=2^{j}\cdot P_{F_{\beta^{2^j}}}\cup\left\lbrace2^{j-1},\ldots,2\right\rbrace,\quad\text{for}\quad j=2,\ldots,i.
    $$
    \end{enumerate}
\end{proposition}
\begin{proof}
The proof is by induction on $i\in\mathbb{N}$. First, observe that for $i=1$ the conditions~\eqref{prop:beta_seq1} and \eqref{prop:beta_seq2} follow immediately from Lemma~\ref{lem:renorm}.

Assume that there is $i$ such that the conditions~\eqref{prop:beta_seq1} and \eqref{prop:beta_seq2} holds for every parameter from $(\beta_{i+1},\beta_i]$. Let $\beta\in(\beta_{i+2},\beta_{i+1}]$. Since $(\beta_{i+2},\beta_{i+1}]\subset(1,\sqrt{2}]$, the map $F_{\beta}$ has renormalization $G_{\beta,1}:=G_{\beta}=(F^2_{\beta},F^2_{\beta})$, that is
$$
G_{\beta}(x)=
\begin{cases}F^2 (x),\,&\text{if 
		$x\in\left[u,c\right)$}\\
	F^2 (x),\,&\text{if $x\in (c,v]$}\end{cases}
=
\begin{cases}\beta^2x+\frac{\beta-\beta^2}{2},\,&\text{if 
		$x\in\left[u,c\right)$}\\
	\beta^2x+1-\frac{\beta+\beta^2}{2},\,&\text{if $x\in (c,v]$}\end{cases},
$$
where $[u,v]=[F_{\beta}(0),F_{\beta}(1)]$ and $c=\frac{1}{2}$. Let $h_{u,v}\colon[u,v]\to[0,1]$ be the map defined by~\eqref{eq:rescaling}. Simple calculations yield that 
$$
\begin{aligned}
\left(h_{u,v}\circ G_{\beta}\circ h^{-1}_{u,v}\right)(x)&=
\begin{cases}h_{u,v}\left(\beta^2h_{u,v}^{-1}(x)+\frac{\beta-\beta^2}{2}\right),\,&\text{if 
		$x\in\left[0,c\right)$}\\
	h_{u,v}\left(\beta^2h_{u,v}^{-1}(x)+1-\frac{\beta+\beta^2}{2}\right),\,&\text{if $x\in (c,1]$}\end{cases}\\
    &=
\begin{cases}\beta^2x+1-\frac{\beta^2}{2},\,&\text{if 
		$x\in\left[0,c\right)$}\\
	\beta^2x-\frac{\beta^2}{2},\,&\text{if $x\in (c,1]$}\end{cases}\\
    &=\beta^2 x+1-\frac{\beta^2}{2}\Mod{1}=F_{\beta^2}(x).
\end{aligned}
$$
So the renormalization $G_{\beta,1}$ after rescaling to $[0,1]$ is equal to the map $F_{\beta^2}$ and $\beta^2\in(\beta_{i+1},\beta_i]$.

By the induction assumption the map $F_{\beta^2}$ has $i$ renormalizations 
$$
    G_{\beta^2,j}=(F^{2^j}_{\beta^2},F^{2^j}_{\beta^2}),\quad\text{for}\quad j=1,\ldots,i,
$$ 
that (after rescaling) are equal to the maps $F_{(\beta^2)^{2^j}}=F_{\beta^{2^{j+1}}}$, where $\beta^{2^{j+1}}\in(\beta_{i+1-j},\beta_{i-j}]$. Fix $j\in\{1,\ldots,i\}$ and denote by $[u_j,v_j]$ the renormalization interval of $G_{\beta^2,j}$. Then, for every $x\in[u_j,v_j]$ we have
$$
G_{\beta^2,j}(x)=F^{2^j}_{\beta^2}(x)=\left(h_{u,v}\circ G_{\beta}\circ h^{-1}_{u,v}\right)^{2^j}(x)=\left(h_{u,v}\circ G_{\beta}^{2^j}\circ h^{-1}_{u,v}\right)(x)=\left(h_{u,v}\circ F_{\beta}^{2^{j+1}}\circ h^{-1}_{u,v}\right)(x).
$$
Since $h_{u_j,v_j}\circ G_{\beta^2,j}\circ h^{-1}_{u_j,v_j}=F_{\beta^{2^{j+1}}}$, we obtain
$$
F_{\beta}^{2^{j+1}}(x)=\left(h_{u,v}^{-1}\circ G_{\beta^2,j}\circ h_{u,v}\right)(x)=\left[(h_{u_j,v_j}\circ h_{u,v})^{-1}\circ F_{\beta^{2^{j+1}}}\circ(h_{u_j,v_j}\circ h_{u,v})\right](x)
$$
for $x\in[a,b]:=h_{u,v}^{-1}([u_j,v_j])=[h_{u,v}^{-1}(u_j),h_{u,v}^{-1}(v_j)]$. Next, observe that $h_{u_j,v_j}\circ h_{u,v}\colon[a,b]\to[0,1]$ and
$$
\left(h_{u_j,v_j}\circ h_{u,v}\right)(x)=\frac{x-(u_j(v-u)+u)}{(v_j-u_j)(v-u)}=\frac{x-h_{u,v}^{-1}(u_j)}{h_{u,v}^{-1}(v_j)-h_{u,v}^{-1}(u_j)}=\frac{x-a}{b-a}=h_{a,b}(x).
$$
Therefore $h_{a,b}\circ F_{\beta}^{2^{j+1}}\circ h^{-1}_{a,b}=F_{\beta^{2^{j+1}}}$, so the map $F_{\beta}^{2^{j+1}}\colon[a,b]\to[a,b]$ is an expanding Lorenz map on $[a,b]$. Hence $G_{\beta,j+1}:=(F_{\beta}^{2^{j+1}},F_{\beta}^{2^{j+1}})$ is a renormalization of $F_{\beta}$ for each $j=1,\ldots,i$.

Finally, let us note that by Lemma~\ref{lem:renorm} we have
$$
P_{F_{\beta}}=2\cdot P_{F_{\beta^{2}}}\cup\{2\}.
$$
Using again the induction assumption, we obtain
$$
P_{F_{\beta}}=2\cdot\left(2\cdot P_{F_{(\beta^2)^{2}}}\cup\{2\}\right)\cup\{2\}=2^2\cdot P_{F_{\beta^{2^2}}}\cup\{2^2,2\}
$$
and
$$
P_{F_{\beta}}=2\cdot\left(2^{j}\cdot P_{F_{(\beta^2)^{2^j}}}\cup\left\lbrace2^{j-1},\ldots,2\right\rbrace\right)\cup\{2\}=2^{j+1}\cdot P_{F_{\beta^{2^{j+1}}}}\cup\left\lbrace2^{j},\ldots,2\right\rbrace,
$$
for $j=2,\ldots,i$. The proof is completed.
\end{proof}
By Prop.\ref{prop:beta_seq} the renormalization operator $R_{(2,2)}F_{\beta}=F_{\beta^2}$, acting on the family of maps $\{F_{\beta}\}_{\beta\in(1,\sqrt{2}]}$, is well defined. The action of $R_{(2,2)}$ is reflected in the parameter space $\Delta$ by the map
$$
R_{\Delta}\colon\bigcup_{i=1}^\infty\Delta_i\to\bigcup_{i=0}^\infty\Delta_i,\quad \left(\beta,1-\frac{\beta}{2}\right)\mapsto\left(\beta^2,1-\frac{\beta^2}{2}\right),
$$
that sends a set $\Delta_i$ to $\Delta_{i-1}$, for each $i\in\mathbb{N}$ (cf. Figure~\ref{fig:parameter_space}). A further partition of sets $\Delta_i$ is induced by the sequence of parameters $\{\varepsilon_i\}_{i=1}^\infty$ introduced in subsection \ref{subsec:symbolic_dynamics}. The points $(\varepsilon_i,1-\frac{\varepsilon_i}{2})$ divide the curve $\Delta_0$ into countably many parts, and similarly, their preimages under the map $R_{\Delta}$ divide the rest of the curves $\Delta_i$. Recall that for parameters from $\Delta_0$, bifurcations and dynamics of $\beta$--transformations were discussed in the previous subsections (see, in particular, Cor.\ref{cor:kneading_iff}, Prop.\ref{prop:itineraries_rot}, Cor.\ref{cor:all_periods} and Cor.\ref{cor:density}). As a consequence of Prop.\ref{prop:beta_seq}, for $(\beta,1-\frac{\beta}{2})\in\Delta_i$ the behavior of the map $F_{\beta}$ in a small neighborhood of the critical point $c=\frac{1}{2}$ is described by the map $F_{\gamma}$ with $(\gamma,1-\frac{\gamma}{2})\in\Delta_0$, where $\gamma=\beta^{2^i}$. In other words, the bifurcations and behavior of the maps defined by parameters in $\Delta_0$ can also be observed on a microscopic scale in $\Delta_i$, for $i>0$.

In the context of the above discussion, let us note that if $F$ is an expanding Lorenz map with renormalization $G=(F^l,F^r)$, then every periodic point of the map $G$ is also periodic for $F$. In particular, Corollary~\ref{cor:density} together with Proposition~\ref{prop:beta_seq} implies that for any $\beta\in(1,2]$ the map $F_{\beta}$ has infinitely many periodic orbits.

\begin{example}
    Consider the sequence of maps $\{F_{\beta_i}\}_{i\in\mathbb{N}_0}$. By Proposition~\ref{prop:beta_seq} each element $F_{\beta_i}$ in this sequence is the renormalization $(F_{\beta_{i+1}}^2,F_{\beta_{i+1}}^2)$ of the next element $F_{\beta_{i+1}}$ (after appropriate rescaling). It is known that the first element, the doubling map $F_{\beta_0}=F_2$, has a periodic orbit of every period $n\in\mathbb{N}$. Hence, we get
    $$
    P_{F_{\beta_1}}=2\cdot P_{F_{\beta_0}}\cup\left\lbrace2\right\rbrace=2\cdot\mathbb{N}
    $$
    and
     \begin{equation*}
	P_{F_{\beta_i}}=2\cdot P_{F_{\beta_{i-1}}}\cup\left\lbrace2\right\rbrace=2^2\cdot P_{F_{\beta_{i-2}}}\cup\left\lbrace2^2,2\right\rbrace=\ldots
	=2^{i}\cdot\mathbb{N}\cup\left\lbrace2^{i-1},\ldots,2\right\rbrace,
\end{equation*}
for $i>1$. In particular, the greatest period of the map $F_{\beta_i}$ in the Sharkovsky order is equal to $3\cdot2^{i}$ (cf. Proposition~\ref{prop:sharkovsky}). Thus, increasing the parameter from $\beta_i$ to $\beta_{i-1}$ causes a "jump" in this order.

To describe the renormalization phenomenon from a symbolic point of view, first recall that the kneading invariant of the doubling map $F_{\beta_0}$ is given by $k_{F_{\beta_0}}=(10^\infty,01^\infty)$. By replacing the symbols $0$ and $1$ in $k_{F_{\beta_0}}$ with words $L=01$ and $R=10$ we obtain the kneading invariant 
$$
k_{F_{\beta_1}}=(RL^\infty,LR^\infty)=(10(01)^\infty,01(10)^\infty)
$$
of the map $F_{\beta_1}$. Using the same substitution in $k_{F_{\beta_1}}$ we get
$$
k_{F_{\beta_2}}=(RL(LR)^\infty,LR(RL)^\infty)=(1001(0110)^\infty,0110(1001)^\infty).
$$
Proceeding this way, we can construct the kneading invariant of each map $F_{\beta_i}$ for $i\in\mathbb{N}$.
\end{example}

We now show that $\{F_{\beta_i}\}_{i\in\mathbb{N}}$ are the only maps of the form~\eqref{eq:family_maps} that renormalize to the doubling map $F_2$.

\begin{lemma}\label{lem:doubling_conjugate}
Let $\beta\in(1,2]$. Then the map $F_\beta$ has a renormalization $G=(F^k_\beta,F^k_\beta)$ that is conjugate to the doubling map if and only if $\beta=\beta_i=\sqrt[2^i]{2}$ for some $i\in\mathbb{N}$. Moreover, in this case for each $1\leq j\leq k$ the map $F^j_\beta$ is continuous on both $[u,\frac{1}{2})$ and $(\frac{1}{2},v]$, where $[u,v]$ denotes the renormalization interval of $G$.
\end{lemma}
\begin{proof}
    Firstly, suppose that for some $i\in\mathbb{N}$ and $\beta\in(\beta_{i+1},\beta_i)$ the map $F_{\beta}$ has a renormalization $G=(F^k_\beta,F^k_\beta)$ that is conjugate to $F_2$. i.e., the doubling map. Since $F_{\beta}$ has a primary $2(1)$-cycle, it follows from \cite[Theorem~3.5(2)]{ChO} that the number $k$ must be even. Furthermore, we get $\beta=\sqrt[k]{2}\in(\beta_{i+1},\beta_i)$, because the slope of $G$ must be equal to $2$. Since $\sqrt[2^{i+1}]{2}<\beta<\sqrt[2^{i}]{2}$, we conclude that $2^i<k<2^{i+1}$. In particular, we have $i\geq2$, because the number $k$ is even. Next, consider the renormalization interval $[u,v]=[F^{k-1}_{\beta}(0_+),F^{k-1}_{\beta}(1_-)]$ of $G$. Since the map $G$ is conjugate to the doubling map, we have $F^k_{\beta}(u)=G(u)=u$, so $u$ is a periodic point for $F_{\beta}$. Denote by $s$ the minimal period of $u$ (under the map $F_{\beta}$). So $s\leq k<2^{i+1}$. Since $\beta\in(\beta_{i+1},\beta_i)$ for some $i\geq2$, by Proposition~\ref{prop:beta_seq} the set of periods of $F_{\beta}$ is given by
    $$
    P_{F_{\beta}}=
	2^{i}\cdot P_{F_{\beta^{2^i}}}\cup\left\lbrace2^{i-1},\ldots,2\right\rbrace.
    $$
    Hence $s=2^j$ for some $j\in\{1,\ldots,i-1\}$. Denote by $[\widetilde{u},\widetilde{v}]:=[F^{2^i-1}_{\beta}(0_+),F^{2^i-1}_{\beta}(1_-)]$ the renormalization interval of $G_{\beta,i}=(F^{2^i}_{\beta},F^{2^i}_{\beta})$. Since $2^i<k$, we have $F^{k-2^i}_{\beta}(\widetilde{u})=F^{k-1}_{\beta}(0_+)=u$. Observe that $\widetilde{u}<u$ implies $G_{\beta,i}(u)=F^{2^i}(u)=u$, which is impossible because $G_{\beta,i}$ is an expanding Lorenz map on $[\widetilde{u},\widetilde{v}]$. The case $\widetilde{u}=u$ is also excluded, as $G_{\beta,i}$ is not conjugate to the doubling map. So $\widetilde{u}\in(u,v)$. But then $G(\widetilde{u})=F^k_{\beta}(\widetilde{u})=F_{\beta}^{2^i}(u)=u$, which is possible only when $\widetilde{u}=u$ or $\widetilde{u}=\frac{1}{2}$. Since neither of these cases can occur, we get a contradiction.

    To prove the second implication,, assume that $\beta=\beta_i$ for some $i\in\mathbb{N}$. It follows from Proposition~\ref{prop:beta_seq} that $F_{\beta}$ has $i$ renormalizations $G_{\beta,l}=(F^{2^l}_{\beta},F^{2^l}_{\beta})$ defined on the intervals $[F^{2^l-1}_{\beta}(0_+),F^{2^l-1}_{\beta}(1_-)]$ (for $l=1,\ldots,i$). In particular, the last one $G:=G_{\beta,i}=(F^{2^i}_{\beta},F^{2^i}_{\beta})$ is conjugate to the doubling map $F_2$. Finally, let 
    $$
    [u,v]:=[F^{2^i-1}_{\beta}(0_+),F^{2^i-1}_{\beta}(1_-)]
    $$ 
    be the renormalization interval of $G$ and 
    $$
    J_G:=\left\lbrace x\in [0,1]\:|\: \text{Orb}_{F_{\beta}}(x)\cap \left(u, v\right) =\emptyset\right\rbrace .
    $$
    Note that the renormalization $G$ satisfies the assumptions of Theorem~6.1(2) in \cite{ChO} (cf. Theorem~A in \cite{Ding}). Thus, the points
    $$
    e_-:=\sup\{x\in J_G\:|\: x<c\},\quad e_-:=\inf\{x\in J_G\:|\: x>c\},
    $$
    are well defined and for $t=1,\ldots,2^i-1$ the sets $F_{\beta}^t([e_-,c))$ and $F_{\beta}^t((c,e_+])$ do not contain the critical point $c=\frac{1}{2}$. Since $e_-\leq u<c<v\leq e_+$, we conclude that $F_{\beta}^j$ is continuous on both intervals $[u,c)$, $(c,v]$, for all $j=1,\ldots,2^i$.   
\end{proof}
\begin{remark}
    Although intuitively this seems to be true in a general case, we have not found a proof that for any renormalization $G=(F^l,F^r)$ of an expanding Lorenz map $F$ the renormalization interval $[u,v]$ satisfies $[u,c)\cap\bigcup_{j=0}^{l-1}F^j(c)=\emptyset$ and $(c,v]\cap\bigcup_{j=0}^{r-1}F^j(c)=\emptyset$. We will discuss this issue in more detail in Appendix~\ref{ex:map_nonsymmetric1}.
\end{remark}

We now consider the pullback of renormalization phenomena from the Lorenz maps back to the Lorenz attractor. To this end, we first need to define a realistic notion of Renormalization in the Lorenz attractor. Based on Th.\ref{reduct} and Cor.\ref{reduct2}, we know the dynamics of the Lorenz attractor are essentially those of a deformed $\beta$--transformation (at least sufficiently close to the set $T$). Recalling each parameter $v\in P$ defines a first-return map $\psi_v:R_0\cup R_1\to R$, this motivates us to define Renormalization as follows:
\begin{definition}
\label{renormvector}    A Lorenz attractor at a parameter value $v\in P$ is said to be \textbf{renormalizable} if its corresponding $\beta$--transformation $F_\beta$ is well defined and renormalizable. In this case, the parameter $v$ is also said to be renormalizable.
\end{definition}
\begin{remark}
    This is not the only way to define renormalization on the Lorenz attractor. Another definition can be made using Template Theory (see Def.2.4.6 in \cite{KNOTBOOK}). We will consider this approach later on.
\end{remark}
In other words, our approach to Renormalization on the attractor looks only at the "dynamical core" of the attractor, i.e., the parts which can be factored to the Lorenz map. The reason we focus only on these parts is because, as Th.\ref{expl} suggests, it is the change in these parts where the "essential" changes occur as a parameter $v\in P$ tends towards the set $T$ of trefoil parameters.

Recalling the discussion above, we know a $\beta$--transformation $F_\beta$ is renormalizable if and only if $\beta\in(1,\sqrt{2}]$. Since by Lemma \ref{betato2} we know the closer a parameter $v\in P$ is to the set $T$ the closer the parameter $\beta$ is to $2$, it follows we cannot expect Lorenz attractors "too close" to the set $T$ to be renormalizable. At this point we remark that even though such heuristic may motivate one to try and define a Renormalization Operator on the Lorenz attractor, in many ways our definition above is not enough for such a construction. To illustrate, let $v_1,v_2$ be two parameters in $P$ with corresponding Lorenz maps $F_{\beta_1}$ and $F_{\beta_2}$. Assuming $F_{\beta_1}$ can be renormalized to $F_{\beta_2}$, there is no reason to assume we can find an invariant subset on the attractor for $v_1$ which can be factored to the dynamics on the attractor corresponding to $v_2$. The reason this is so is because we do not know the fine structure of the attractor for either $v_1$ or $v_2$ - which could hamper such an explicit construction. That being said, we will now prove the following result which, in part, explains how these ideas play out on the Lorenz attractor:

\begin{theorem}
    \label{renormth}
   Let $v\in P$ be a parameter for the Lorenz system, and  recall we denote its first-return map by $\psi_v:R_0\cup R_1\to R$. Let $I_v$ denote the maximal invariant set of $\psi_v$ in $R_0\cup R_1$, and let $F_\rho$ denote its corresponding $\beta$--transformation with a slope $\rho=\rho(v)$. Assume $v$ is renormalizable - then the following is true:
   \begin{itemize}
       \item There exists an invariant set $I'\subset I_v$, a $\beta$--transformation $F_\gamma$ with a slope $\gamma\in(\rho,2]$, a natural number $n>0$ and a continuous map $\pi:I'\to I_\gamma$ s.t. $\pi\circ \psi^n_v=F_\gamma\circ\pi$ - where $I_\gamma$ is the invariant set of $F_\gamma$ in $[0,1]\setminus\{\frac{1}{2}\}$.
       \item $n$ can be chosen to be a dyadic number, i.e., $n=2^j$ for some $j>0$.
   \item There exists a set of periodic orbits for $\psi_v$ in $I_v$ that includes orbits of length $2^j q$, where $q$ is some arbitrary period of a periodic orbit for $F_\gamma$ - as well as orbits of periods $1,2,2^2,...,2^{j-1}$.
   \end{itemize}
\end{theorem}
\begin{proof}
    Since the map $F_{\rho}$ is renormalizable, by Prop.\ref{prop:beta_seq} it has a renormalization $G=(F_\rho^{2^j},F_\rho^{2^j})$, for some $j>0$, which after rescaling is equal to a $\beta$--transformation  $F_\gamma$ with a slope $\gamma=\rho^{2^j}$. Thus, the first two assertions are immediate consequences of Th.\ref{reduct} and Cor.\ref{reduct2}. Therefore, to conclude the proof it remains to prove the final assertion about the length of periodic orbits. That, however, is immediate - by Prop.\ref{prop:beta_seq} we know the sets of periods of maps $F_\rho$ and $F_\gamma$ are related by the formula $P_{F_{\rho}}=2\cdot P_{F_\gamma}\cup\{2\}$, if $j=1$,
    and
    $P_{F_\rho}=2^j\cdot P_{F_\gamma}\cup\{ 2^{j-1},...,2\}$, if $j>1$. By Th.\ref{reduct} we conclude all these periods carry over from $F_\rho$ back to the original attractor, and the conclusion follows.
\end{proof}
\begin{remark}
    At this point, we note that our methods do not allow for the detection of renormalizable parameters in $P$ - this is something that we  believe can only be studied numerically. 
\end{remark}

\begin{figure}[h]
\centering
\begin{overpic}[width=0.3\textwidth]{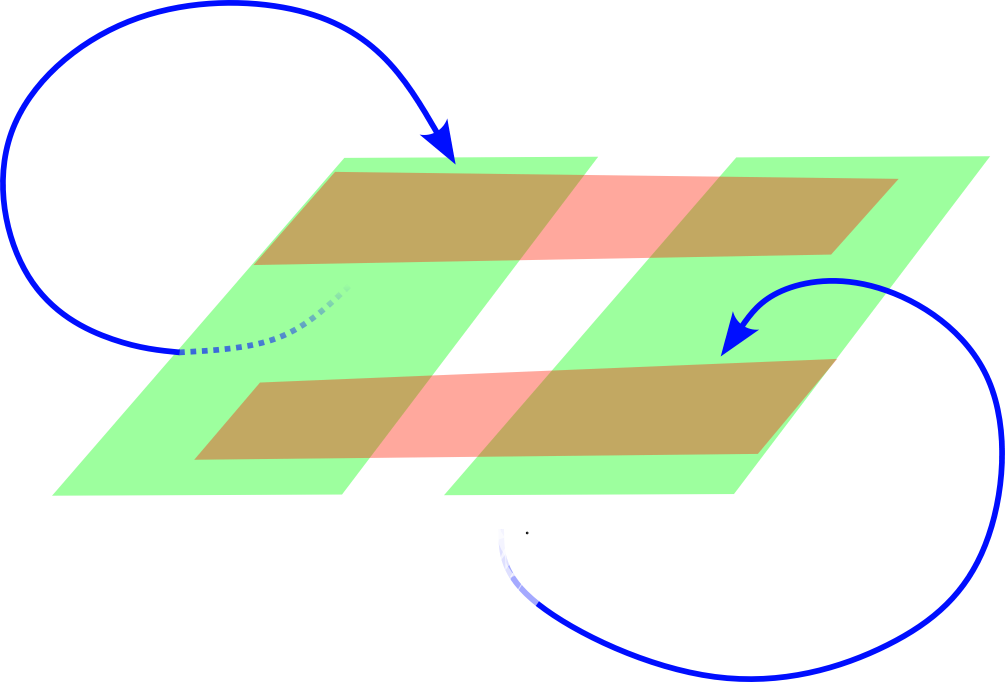}

\end{overpic}
\caption{\textit{An example of a flow corresponding to some geometric model whose kneading is that of some Lorenz map.}}
\label{geometricflow}
\end{figure}

Before moving on, we remark that in a sense, Th.\ref{renormth} is the closest we can (currently) get to defining a meaningful Renormalization Operator on the Lorenz attractor. To illustrate assume $v\in P$ is a parameter whose attractor $A_v$ corresponds to a $\beta$--transformation $F_\rho$, s.t. $F_\rho$ can be renormalized as follows - $F_\rho\to F_{\rho_1}\to F_{\rho_2}\to...\to F_{\rho_k}$. Now, let $L_j$, $j=1,...,k$ denote a geometric model flow for $F_{\rho_j}$ whose kneading is the same as $F_{\rho_j}$ (as illustrated in Fig.\ref{geometricflow}) - i.e., some flow which stretches two rectangles over one another. Then, by Th.\ref{renormth} one would expect, at least intuitively, for the attractor $A_v$ to include copies of the attractors corresponding to $L_j,j=1,...,k$ inside it. We now make this idea precise, using Template Theory - at least for some parameters $v\in P$. We begin by recalling the following definition:

\begin{definition}\label{deftemp}
    A \textbf{Template} is a compact branched surface with a boundary endowed with a smooth expansive semiflow, built locally from two types of charts - \textbf{joining} and \textbf{splitting}. Additionally, the gluing maps between charts all respect the semiflow, and act linearly on the edges (see Fig.\ref{TEMP2} for an illustration).
\end{definition}

\begin{figure}[h]
\centering
\begin{overpic}[width=0.5\textwidth]{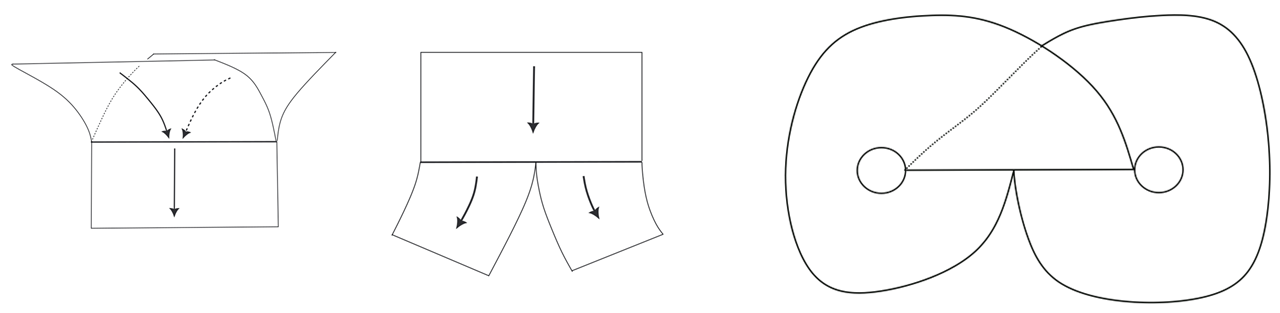}

\end{overpic}
\caption{\textit{From left to right - a joining chart, a splitting chart, and an example of a Template (the $L(0,0)$ Template).}}
\label{TEMP2}
\end{figure}

Templates are an extremely useful tool for the analysis of topological dynamics of three dimensional flows. The connection between three-dimensional flows and Templates is given by the \textbf{Birman-Williams Theorem} (see \cite{BW} or \cite{KNOTBOOK}). That Theorem states the following:
\begin{theorem}
    \label{birwil}  Given a three-manifold $M$ and a smooth flow $\phi:M\times\mathbb{R}\to M$ hyperbolic on its maximal invariant set, $I$, there exists a Template $\tau$ embedded in $I$, with a semiflow $\varphi:\tau\times\mathbb{R}^+\to \tau$. Moreover, every periodic orbit $\mathcal{T}\subseteq I$ in $I$ projects to a unique periodic orbit for $\varphi$, which encodes the knot type of $\mathcal{T}$. Conversely, every periodic orbit for $\varphi$ in $\tau$, save possibly for two, encodes the knot type of at least one periodic orbit $\mathcal{T}\subseteq I$.
\end{theorem}
In other words, a Template for a chaotic attractor can be thought of as a "Knot holder" for a flow which encodes its topological dynamics, or more precisely, its "skeleton" - where the periodic orbits play the role of the "bones" in the said skeleton. We now prove that by combining the theory of renormalization with Th.\ref{birwil} one can, in some cases, assign a Template to Lorenz attractors. Moreover, we will do so without imposing any hyperbolicity assumption on the dynamics. To begin, inspired by Section 4.3 in \cite{GS} we first introduce the following definition: 
\begin{definition}
    A renormalizable parameter $v\in P$ is said to be \textbf{fully renormalizable} if its associated $\beta$--transformation $F_\beta$ has a renormalization $G=(F_\beta^n,F_\beta^n)$ that is conjugate to the doubling map $F_2(x)=2x\Mod{1}$. 
\end{definition}

By Lemma \ref{lem:doubling_conjugate}, a parameter $v\in P$ is fully renormalizable if and only if the corresponding $\beta$--transformation has slope $\beta_i$ for some $i\geq1$.

\begin{figure}[h]
\centering
\begin{overpic}[width=0.55\textwidth]{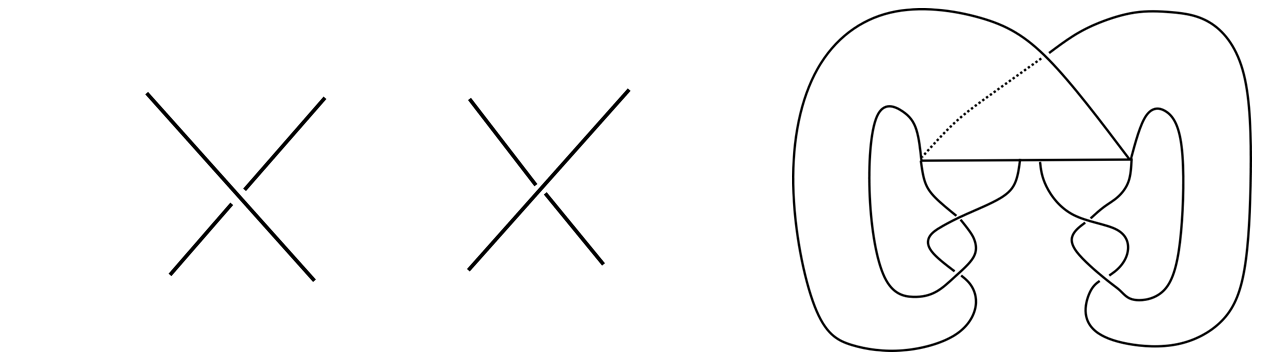}

\end{overpic}
\caption{\textit{From left to right - a positive crossing, a negative crossing, and the $L(-2,-2)$ Template.}}
\label{loret}
\end{figure}

We will also need to recall certain definitions about templates. Given a Template that splits into two strips at the branch line, we say it is an $L(m,n)$ \textbf{Template} (where $m,n\in\mathbb{Z}$, and $L$ stands for "Lorenz") if the left branch has $m$ half--twists and the right one has $n$ half--twists (where the positive and negative sign convention is as illustrated in Fig.\ref{loret}). With that idea in mind, we prove:
\begin{theorem}
    \label{templateth} Given a fully renormalizable $v\in P$ , there exists an even number $k\in\mathbb{Z}$ s.t. every knot encoded by the $L(k,k)$ Template is realized as a periodic orbit on the corresponding Lorenz attractor $A_v$. Moreover, the Template $L(k,k)$ is is \textbf{not} universal, i.e., it does not encode all possible knots and links.
\end{theorem}
\begin{proof}
To begin, consider a fully renormalizable parameter $v\in P$. Let $F_\beta$ denote the corresponding $\beta$ transformation given by Th.\ref{reduct} and Cor.\ref{reduct2}, and let $[u,v]\subseteq[0,1]$ be the renormalization interval with a renormalization factor $n$. Recalling the construction of $F_\beta$ from the first-return map $\psi_v:R_0\cup R_1\to R$ (where $R$ is the cross-section given by Th.\ref{tali}), we know that as we homotope the Lorenz map $F_\beta:[0,1]\setminus\{\frac{1}{2}\}\to[0,1]$ back to the first return map $\psi_v:R_0\cup R_1\to R$, the interval $[u,v]$ is deformed to a sub-rectangle $R'\subseteq R$. Moreover, as the critical point $\frac{1}{2}\in(u,v)$ is deformed to the line $W$ - i.e., the final intersection curve of $W^s(0)$ and $R$ - we know $\frac{1}{2}$ is opened under the homotopy to a straight line bisecting $R'$, as illustrated in Fig.\ref{renorm}.

Let $R'_1,R_2'$ denote the components of $R'\setminus W$ which are deformed to $[u,\frac{1}{2})$ and $(\frac{1}{2},v]$, indexed s.t. $R'_i\subseteq R_i$, $i=0,1$ (see the illustration in Fig.\ref{renorm}). We now claim $\psi^n_v:R'_0\cup R'_1\to R$ is continuous, and moreover, that for all $1\leq j<n$, $\psi^j_v(R'_i)\cap W=\emptyset$, $i=0,1$. To see why, note that if this was not the case, then $\psi^j_v(R'_i)\setminus W$ would be divided into (at least) two distinct components, which are torn away from one another as they hit $W$ - which implies these components diverge away from one another, and form sets in $\mathbb{R}^3$ with a positive distance between them. By Lemma \ref{lem:doubling_conjugate} this cannot be, as it would imply there exists some $1\leq j<k$ s.t. $F^j_\beta$ is discontinuous on either $[u,\frac{1}{2})$ or $(\frac{1}{2},v]$ - and moreover, the only possibility is that $[u,v]\cap( \cup_{j=1}^{n-1}F^{-j}_\beta(\frac{1}{2}))=\emptyset$).

Therefore, since $F^n_\beta:[u,v]\to[u,v]$ acts like the doubling map, it follows the diffeomorphism $\psi^n_v:R'_0\cup R'_1\to R'$ acts on $R'_0\cup R'_1$ like a (topological) Smale Horseshoe map (see the illustration in Fig.\ref{renorm}). In other words, there exists an invariant set $I\subseteq R'_0\cup R'_1$ and a continuous, surjective map $\pi:I\to\{0,1\}^\mathbb{Z}$ s.t.:

\begin{itemize}
    \item $\pi\circ\psi^n_v=\zeta\circ\pi$ - where $\zeta:\{0,1\}^\mathbb{Z}\to\{0,1\}^\mathbb{Z}$ is the double-sided shift. 
    \item If $s\in\{0,1\}^\mathbb{Z}$ is periodic of minimal period $k$, $\pi^{-1}(s)$ is connected and includes at least one periodic orbit of minimal period $k$.
\end{itemize}

By this discussion it follows Th.\ref{templateth} will be proven if we show we can associate a Template with the dynamics of $I$. In more detail, we are going to prove there exists a Lorenz Template $\tau$ as in the statement of the Theorem, s.t. every knot type encoded by $\tau$ (save possibly for two) is realized as a periodic orbit for the flow. We will do so by continuously deforming the first-return map $\psi^n_v$ on $I$ into a hyperbolic state, from which the assertion would follow.
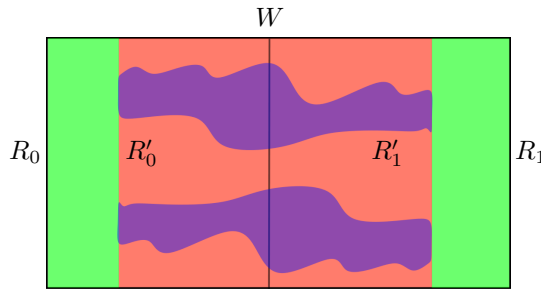
\begin{figure}[h]
\centering
\begin{overpic}[width=0.35\textwidth]{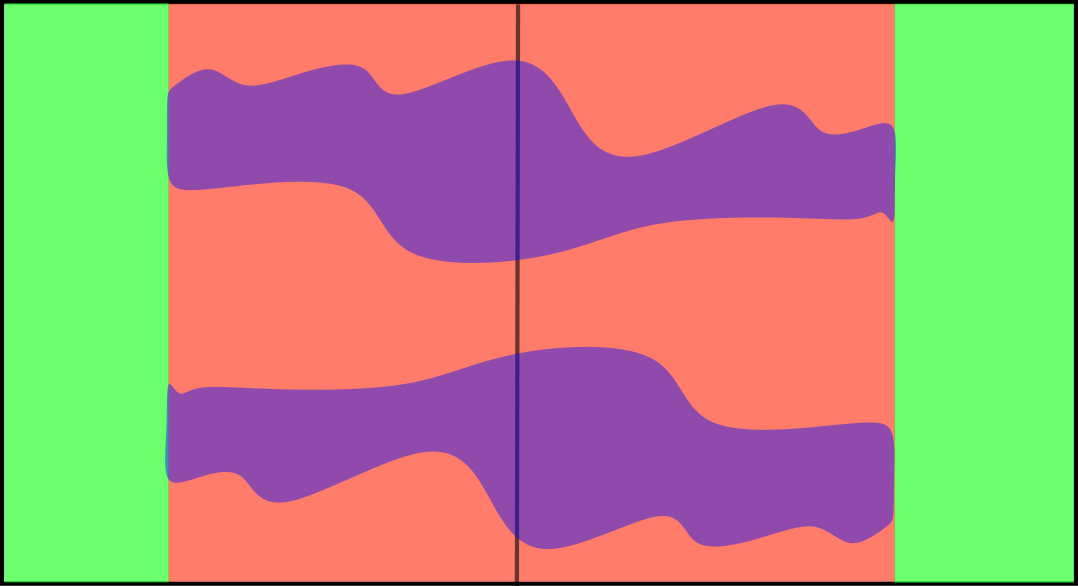}
\put(170,280){$R'_0$}
\put(-100,280){ $R_0$}
\put(1010,280){$R_1$}
\put(700,280){$R'_1$}
\put(450,570){$W$}
\end{overpic}
\caption{\textit{The blue regions denote the $n$--th iterate of the two red rectangles, $R'_0,R'_1$, inside the cross-section.}}
\label{renorm}
\end{figure}

To this end, we begin by deforming the flow as described below. First, we smoothly split the fixed point at the origin from a saddle into an attractor and two repellers by a pitchfork bifurcation, as illustrated in Fig.\ref{splitting}. This has the effect of opening $W$ into an open region, s.t. $R'_0$ and $R'_1$ become two rectangles separated by an open set, $W'$, a topological rectangle. Moreover, the first-return map $\psi^n_v:R'_0\cup R'_1\to R'$ also changes isotopically to the first-return map for this new flow $\phi^n_v:R'_0\cup R'_1\to R'$, as illustrated in Fig.\ref{splitting}. In particular, we do so such that the topological horseshoe structure persists, hence by Th.1 and Th.2 in \cite{Han} we know the components of $I$ vary smoothly and persist throughout this deformation.

\begin{figure}[h]
\centering
\begin{overpic}[width=0.6\textwidth]{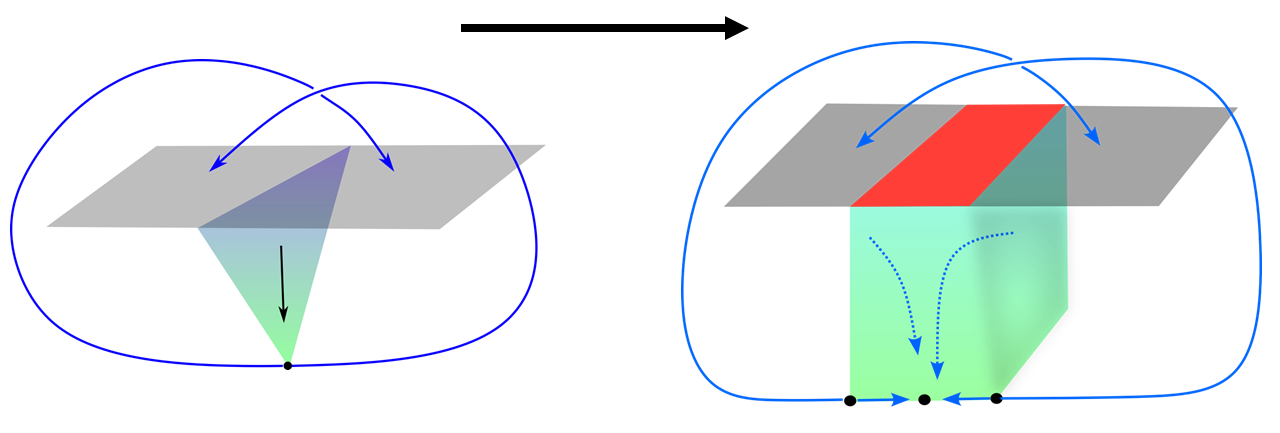}

\end{overpic}
\caption{\textit{Splitting the origin in a pitchfork bifurcation, and opening $W$ into an open basin of attraction.}}
\label{splitting}
\end{figure}

In fact, we can say more - as by Th.1 and Th.2 in \cite{Han} the periodic orbits in $I$ are uncollapsible and unremovable, we immediately conclude the following: let $s\in\{0,1\}^\mathbb{Z}$ be periodic of minimal period $k$. Then, there exists a periodic point $x\in\pi^{-1}(s)$ that lies on a periodic solution curve $\mathcal{T}_x$ that is changed isotopically (in $\mathbb{R}^3$) as we smoothly isotope $\psi^n_v$ to $\phi^n_v$ - and moreover, the number of intersection points of $\mathcal{T}_x$ with $R'$ doesn't change along the isotopy. In other words, the periodic orbit $\mathcal{T}_x$ for the vector field $L_v$ (i.e., the original Lorenz system) persists without changing its knot type as we move from $\psi^n_v$ to $\phi^n_v$.

We now do a similar trick, and "straighten" the dynamics of $\phi^n_v$ to make it into a proper Smale Horseshoe map. In order to do so, we first must push the invariant set of $\phi^n_v$ in $R'_0\cup R'_1$ away from the fixed points. To begin, let $B_0\subseteq \partial R'_0$ and $B_1\subseteq \partial R'_1$ be the boundary arcs of $\partial R'_0$, $\partial R'_1$ intersecting $W$. We smoothly deform the flow by extending them into open rectangles with interiors $b_0,b_1$, which separate $R_i\setminus b_i$ and $W$, $i=0,1$ (see the illustration in Fig.\ref{interior}). In particular, we do so s.t. this deformation of the flow $\phi^n_v:R'_0\cup R'_1\to R'$ is isotopically deformed to some other topological horseshoe $\varphi^n_v:R'_0\cup R'_1\to R'$, as illustrated in Fig.\ref{interior}. Using the same arguments and notations as above, it similarly follows that as we deform $\psi^n_v$ to $\phi^n_v$ to $\psi^n_v$ the periodic orbit $\mathcal{T}_x$ persists, without changing its knot type. And again, we can say more - by Th.1 and Th.2 in \cite{Han} we know the invariant set $I$ varies continuously, without losing any component or collapsing two components into one, as we continuously vary the flows from the Lorenz system into this new vector field, $V$. In detail, there exists an invariant set $I_V\subseteq R'_0\cup R'_1$ for $\varphi^n_v:R'_0\cup R'_1\to R'$ and a continuous, surjective map $\pi_v:I_V\to\{0,1\}^\mathbb{Z}$ s.t.:

\begin{itemize}
    \item $\pi_V\circ\varphi^n_v=\zeta\circ\pi_V$ - where $\zeta:\{0,1\}^\mathbb{Z}\to\{0,1\}^\mathbb{Z}$ is the double-sided shift.
    \item If $s\in\{0,1\}^\mathbb{Z}$ is periodic of minimal period $k$, $\pi^{-1}_V(s)$ is connected and includes at least one periodic orbit of minimal period $k$.
    \item For all $s\in\{0,1\}^\mathbb{Z}$, as we deform $V$ back to the original Lorenz system corresponding to $v$, $\pi^{-1}_V(s)$ is isotopically deformed to $\pi^{-1}(s)$ (where $s$ is as before). 
    \item Moreover, when $s$ is periodic of minimal period $n$, the periodic orbits for $\varphi^n_v$ in $\pi_V^{-1}(s)$ vary to the periodic orbits of $\psi_v^n$ in $\pi^{-1}(s)$ when $\varphi^n_v$ is isotoped back to $\psi^n_v$. In addition, at least one periodic orbit in $\pi^{-1}_V(s)$ has period $n$, and persists when $\varphi^n_v$ is isotoped back to $\psi^n_v$ without changing its period.
\end{itemize}

\begin{figure}[h]
\centering
\begin{overpic}[width=0.7\textwidth]{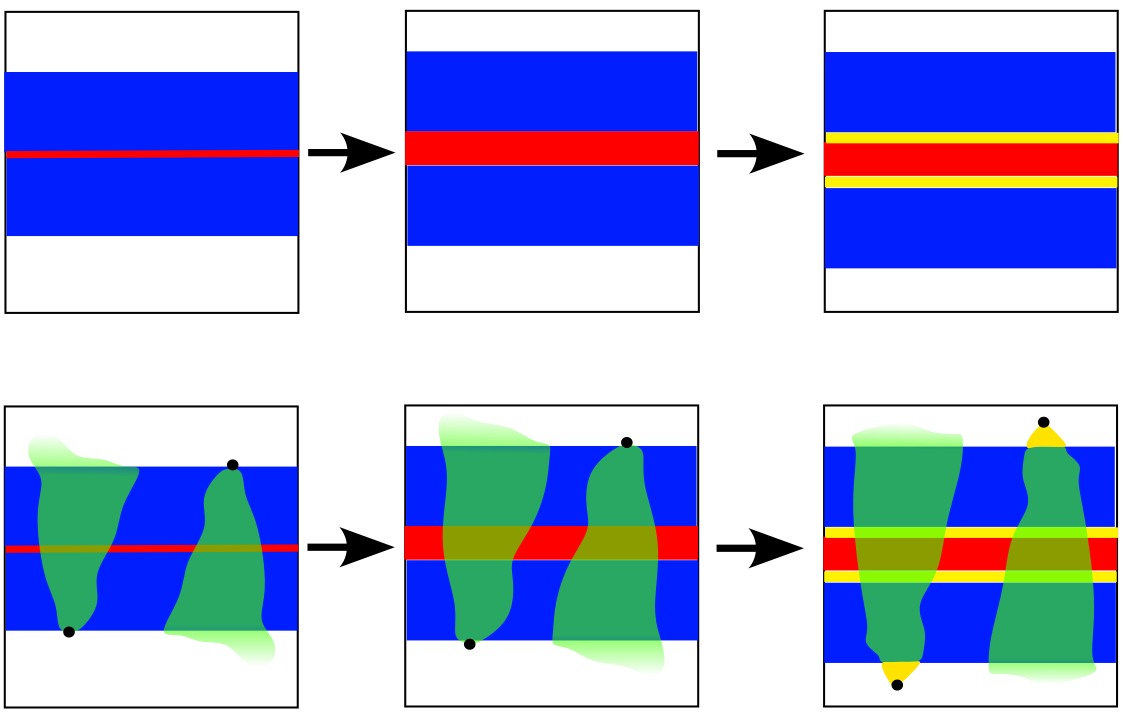}

\end{overpic}
\caption{\textit{The isotopy from $\psi^n_v$ (on the left) to $\phi^n_v$ (in the middle) to $\varphi^n_v$ (on the right). The image above describes the changes in the cross-section - the blue rectangles always denote $R'_0$ and $R'_1$, the red line and rectangles always denote $W$, while the yellow rectangles in the upper-right image denote $b_0$ and $b_1$. Below, from left to right, the green regions always denote the images of $R'_0$ and $R'_1$ under $\psi^n_v,\phi^n_v$ and $\varphi^n_v$ (respectively), while the black dots always denote the hitting point of the separatrices. The yellow regions denote the images of $b_0$ and $b_1$ under $\varphi^n_v$.}}
\label{interior}
\end{figure}

We now continue by moving flow lines and deforming $V$ continuously s.t. the set $\pi^{-1}_V(s)$ becomes the intersection of a nested sequence of (topological) rectangles - in other words, we deform the flow s.t. $\varphi^n_v:R'_0\cup R'_1\to R'$ is homeomorphic to a rectangle map (see the illustration in Fig.\ref{straightt}). As a consequence, it is easy to see that after this deformation, $\pi^{-1}_V(s)$ becomes homeomorphic to a convex set - therefore contractible. We now continue by continuously deforming the flow (if necessary) by simultaneously collapsing every component in the suspension of $I$ by the flow to a singleton. In other words, we deform the flow by isotoping $\varphi^n_v:R'_0\cup R'_1\to R'$ to a map $H:R'_0\cup R'_1\to R'$, which smoothly stretches $R'_0$ over $R'_1$ and vice versa (see the illustration in Fig.\ref{straightt}). Let $F_H$ denote the vector field generating $H$ as its first return map.

\begin{figure}[h]
\centering
\begin{overpic}[width=0.4\textwidth]{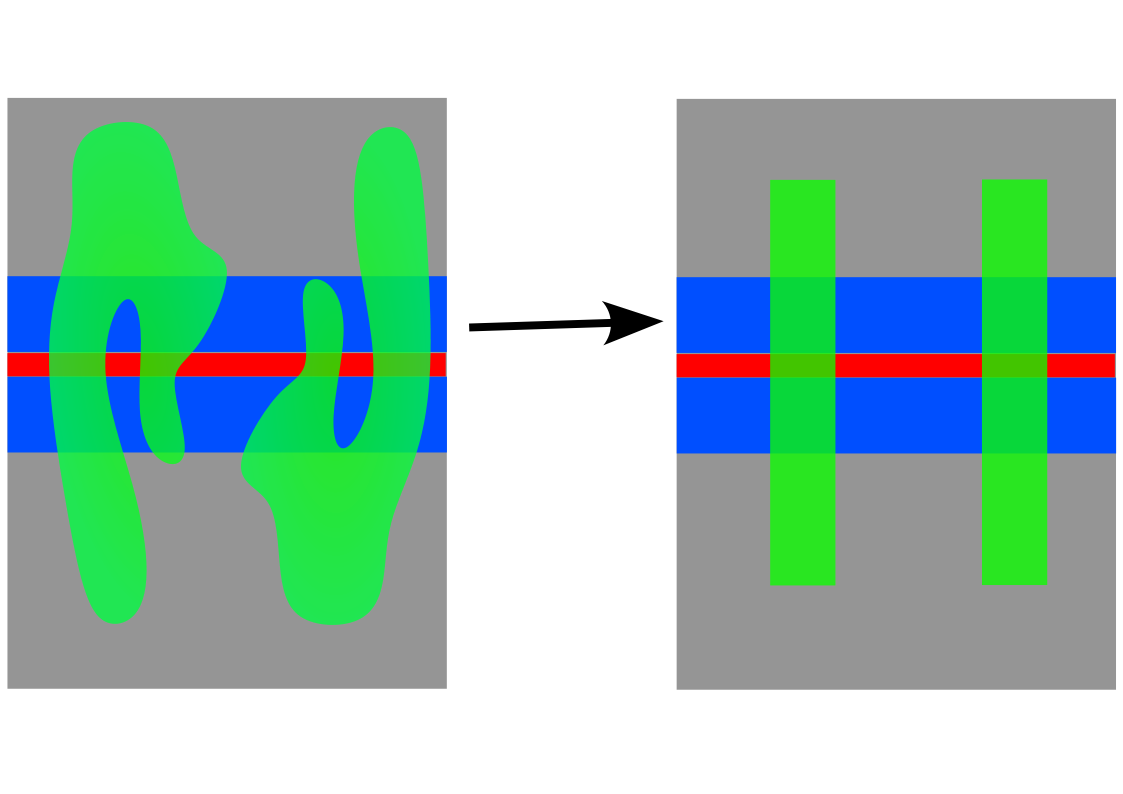}

\end{overpic}
\caption{\textit{The image describes the corresponding straightening $\phi^n_v$ on the left to $H$ on the right, i.e., how we deform the first-return map as we deform the cross-sections. The red rectangle denotes the blown-up $W$, while the blue rectangles always correspond to the regions collapsed to the renormalization interval of the Lorenz map (i.e., $R'_1$ and $R'_0$). The green images always correspond to the $n$--th iterate of each blue rectangle.}}
\label{straightt}
\end{figure}

The resulting map $H$ is easily seen to be a Smale Horseshoe map (see Fig.\ref{straightt}), hence it is hyperbolic on its invariant set. Moreover, the same arguments as before (i.e., Th.1 and Th.2 from \cite{Han}) again imply that as we isotope $H$ back to $\psi^n_v$, the first-return map for the Lorenz system, all the periodic orbits of $H$ persist without changing their minimal period or the way their corresponding solution curve intersects $R$. Put simply, we have proven that if $K$ is a knot type realized as a periodic orbit in the suspension of $H$ with respect to the vector field $F_H$, then $K$ is also realized as a periodic orbit for the suspension of $\psi^n_v$ w.r.t. the vector field $L_v$ - i.e., $K$ is generated as a periodic orbit for the Lorenz attractor at parameter value $v$.

We now recall that by the hyperbolicity of $H$ the Birman-Williams Theorem guarantees the existence of a Template $\tau$ encoding every knot type generated by suspending $H$ w.r.t. $F_H$ (save possibly for two extra knots). Therefore, by the above it follows every knot type $K$ encoded by $\tau$ (save possibly for two) is realized as a periodic orbit on the Lorenz attractor corresponding to the parameters $v$. Moreover, since by Cor.3.1.14 in \cite{KNOTBOOK} $\tau$ encodes infinitely many different knot types, the Lorenz attractor corresponding to $v$ also includes infinitely many knots. All in all, we conclude the dynamical complexity of the Lorenz attractor at parameter values $v$ is complex at least as dictated by $\tau$.

\begin{figure}[h]
\centering
\begin{overpic}[width=0.25\textwidth]{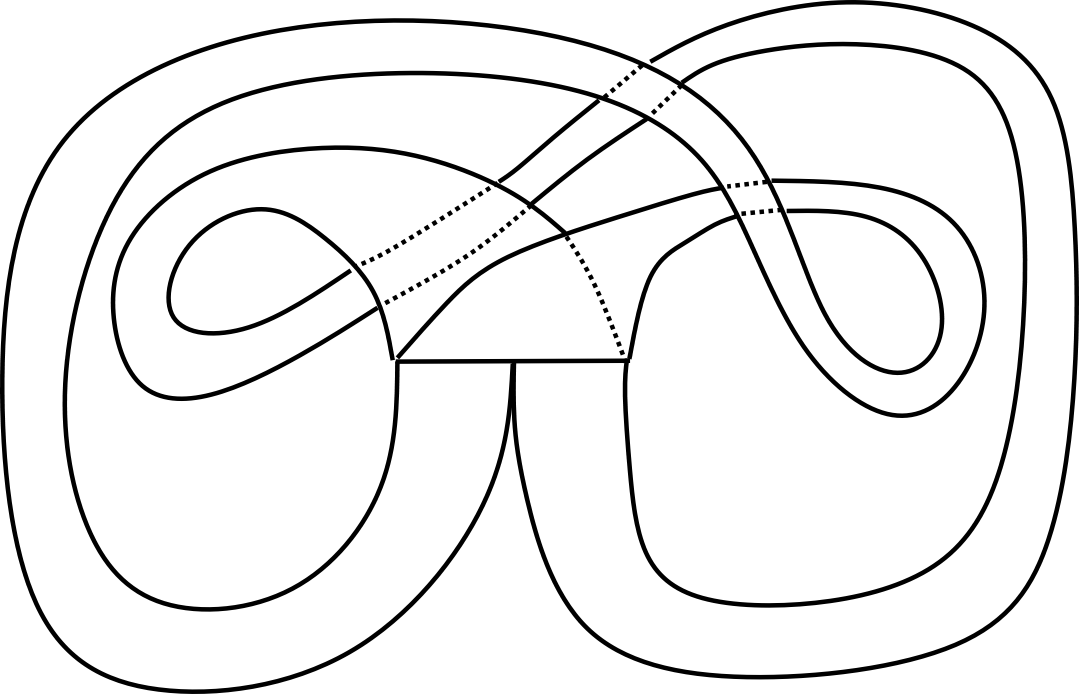}

\end{overpic}
\caption{\textit{The template $\tau$. Here we assume $n=2$, for simplicity.}}
\label{templat}
\end{figure}

Having associated a template with the dynamics of the Lorenz system at $v$, to complete the proof of Th.\ref{templateth} we need to show $\tau$ is a Lorenz $L(k,k)$ template, for some even $k$. To this, we first note the Template $\tau$ is unique - in other words, that it does not depend on our choice of deformation. This follows immediately by seeing that as we isotope $\psi^n_v$ to $H$ we do not change the way the flow suspends initial conditions in $I$ - we only change the size of the components in $I$, by collapsing the suspension of every component into a flow line. This proves that no matter how we perform the deformation above, we always get the same map $H:R'_0\cup R'_1\to R$, up to some isotopy. Consequentially, we get the same template $\tau$, or, in other words, $\tau$ is independent of the deformation.
 
In order to prove $\tau$ is a Lorenz Template, we first prove every knot type encoded by $\tau$ is also encoded by the Lorenz $L(0,0)$ template. To see why this is the case, note that given a parameter $p\in T$, as we deform the Lorenz system from the parameter $v$ to $p$ the periodic orbits encoded by $\tau$ persist, without bifurcating. This follows immediately by Th.\ref{reduct} and Prop.\ref{prop:sharkovsky}, which together imply the periodic orbits on the Lorenz persist as $v\to p$, i.e., as the dynamics on the attractor become increasingly $C^1$--close to those of trefoil parameters. Again, from the same arguments, they persist without changing their knot type. Now, recall that by Th.\ref{tali} the periodic dynamics at trefoil parameters are those of the Lorenz $L(0,0)$ Template - or, put precisely, every knot type on the template $L(0,0)$ (save possibly for two extra knots) is realized as a periodic orbit for trefoil parameters. Therefore, as the periodic orbits encoded by $\tau$ all persist as both $v\to p$ and $\beta\to2$, we conclude all the periodic orbits encoded by $\tau$ are also encoded by the $L(0,0)$ Template, as illustrated in Fig.\ref{templat}.

This teaches us that since the two strips emanating from the branch line for $L(0,0)$ are not braided, the same is true for $\tau$ (see the illustrations in Fig.\ref{templat}). In other words, the semiflow on $\tau$ which is given by the Birman-Williams Theorem suspends each strip a finite number of times around the wing centers of the butterfly attractor, until it returns to cover the branch line (see Fig.\ref{templat}). Moreover, these two strips remain unbraided as they move in space. Finally, by the symmetry properties of the Lorenz system, we know the semiflow winds both strips around the wing centers the same number of times. This implies that as we "straighten" the template $\tau$ isotopically as illustrated in Fig.\ref{straight}, every winding around the wing center is converted into two half-twists, one positive and one negative. Moreover, by the symmetry properties of the Lorenz system these twists are symmetric - i.e., a positive half twist one one strip corresponds to a negative half twist for the second, and vice versa. This implies, by definition, that $\tau$ is a Lorenz $L(k,k)$ Template, where $k\in 2\mathbb{Z}$, for some $k\ne0$.

Having proven the existence of the Template $\tau$ and that it is a Lorenz $L(k,k)$ template, to conclude the proof, it remains to show $L(k,k)$ is not universal - i.e., that it does not encode all knot types. That, however, is immediate - to see why, recall that by Prop.3.2.21 in \cite{KNOTBOOK} the Lorenz $L(0,0)$ template is not universal. As we already proved every knot in $L(k,k)$ is also encoded by $L(0,0)$ it follows $L(k,k)$ cannot be universal as well. The proof of Th.\ref{templateth} is now complete.
\end{proof}

\begin{figure}[h]
\centering
\begin{overpic}[width=0.2\textwidth]{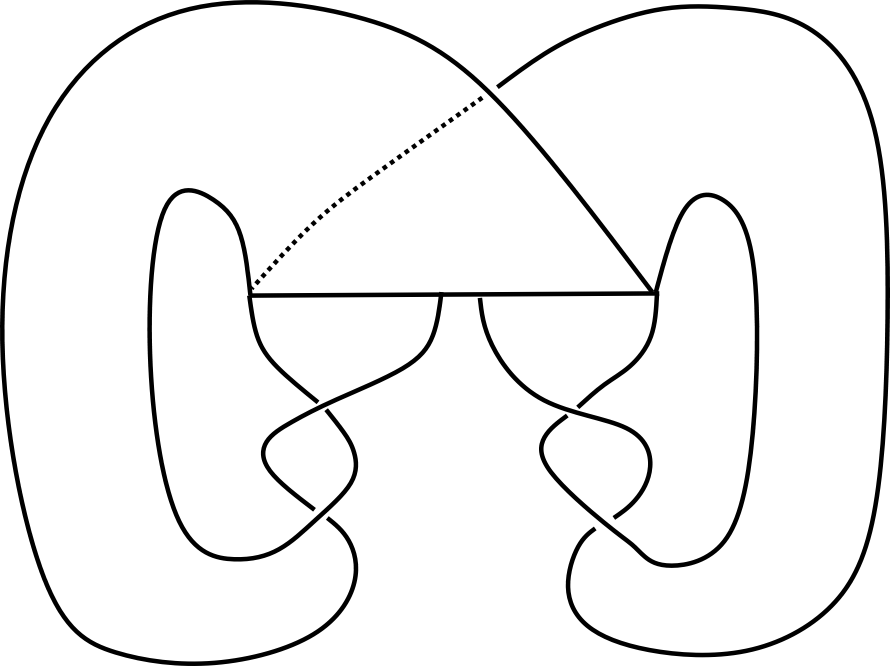}

\end{overpic}
\caption{\textit{"Straightening" the template $\tau$ in Fig.\ref{templat} by stretching the strips beginning at the branch line to get $L(k,k)$, for some $k\in\mathbb{Z}$. Here we assume $n=2$, for simplicity.}}
\label{straight}
\end{figure}
To continue, recall a \textbf{Lorenz knot} is a knot which is encoded by the $L(0,0)$ template. Therefore, recalling every Template encodes an infinite collection of knot types, as an immediate consequence of Th.\ref{templateth} and Th.2.2 in \cite{Deh} we conclude:
\begin{corollary}
\label{prime} Assume there exists a fully renormalizable parameter $v\in P$. Then, the Lorenz attractor corresponding to $v$ generates infinitely many Lorenz knot, of infinitely many different types. Moreover, all these knots are prime.
\end{corollary}

Before we conclude this section, we remark the Lorenz $L(0,0)$ Template was investigated in \cite{Hol}, where it was proven it encodes infinitely many Torus knots of different resonances. Moreover, the class of Lorenz knots is well-studied - for more details on the properties of Lorenz knots, see \cite{Deh}. 

\section{Discussion}
\label{discussion}

While answering many questions, our results also raise several others, which we survey in this section. To begin, we first recall that throughout this paper, our analysis of the Lorenz attractor and the $\beta$--transformations was mostly topological. In detail, most of our results in the previous section can be described as "pulling back" the topological dynamics of the $\beta$--transformations and showing they also exist in some form on the Lorenz attractor. As previously remarked, in addition to the topological theory there also exists a rich measurable theory for the dynamics of $\beta$--transformations (and of Lorenz maps, in general). This raises the following question: can we "pull back" these measurable results to the Lorenz attractor, and use them to study its statistical properties?

At present, we do not have a satisfying answer to this question - in fact, the best we have so far is Cor.\ref{measurableattractor}, which proves the dynamics on the Lorenz attractor are essentially those of a "deformed" measurable dynamical system. The main reason for that is because when we deform the interval map $F_\beta$ back to the first-return map for the flow, singletons in the invariant set of $F_\beta$ in $[0,1]\setminus\{\frac{1}{2}\}$ possibly expand to continua - which could obstruct the possibility of a dense orbit for the flow on the Lorenz attractor. That being said, as proven in \cite{TUC}, at parameter values  $(\sigma,\rho,\beta)=(10,28,\frac{8}{3})$ as the flow is essentially the same as the Geometric Model this cannot occur, i.e., the dynamics on the attractor at this parameter value are "well mixed". Therefore, inspired by both \cite{TUC} and by \cite{WY}, we conjecture the following:
\begin{conjecture}
\label{hyperbolicityconj}    Given any parameter $v\in P$, the first-return map of the attractor $\psi_v:R_0\cup R_1\to R$ is conjugate on its invariant set to the map $h_v:R_0\cup R_1\to \mathbb{R}^2$ (where $h_v$ is as in Fig.\ref{S}).
\end{conjecture}
\begin{figure}[h]
\centering
\begin{overpic}[width=0.4\textwidth]{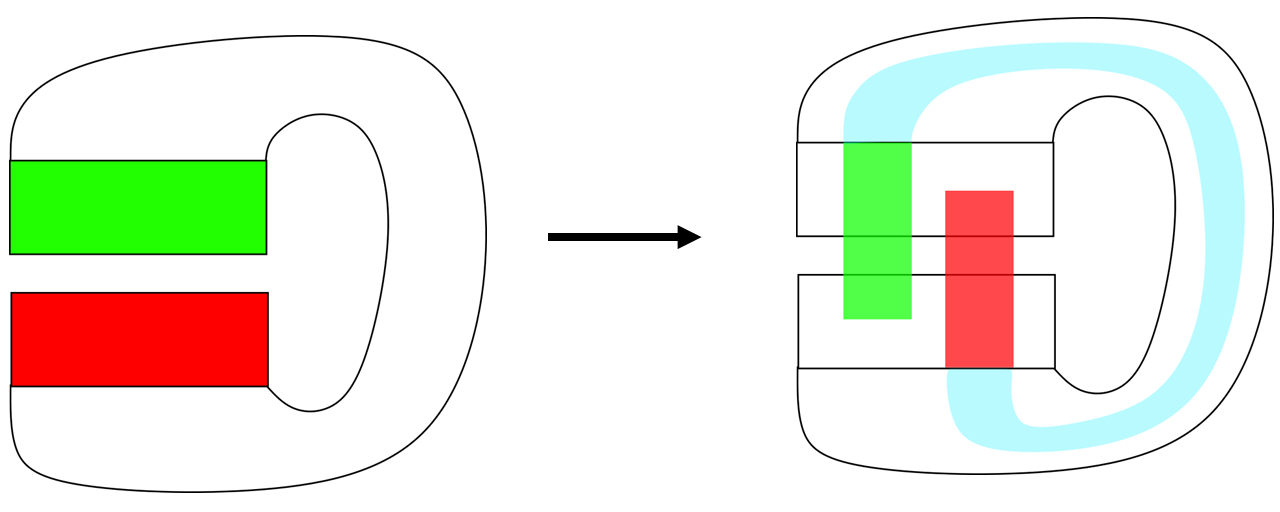}

\end{overpic}
\caption{\textit{The set $S$ on the left, and its image w.r.t. the map $h_v:S\to\mathbb{R}^2$ on the right. The red rectangle denotes $R_0$ while the green denotes $R_1$.}}
\label{S}
\end{figure}

Recalling the proof of Th.\ref{reduct}. we know the map $h_v$ is collapsed to a Lorenz-like map $f_r$ by collapsing the stable foliation to singletons. Assuming Conjecture \ref{hyperbolicityconj} can be proven, it would follow the dynamics of the Lorenz attractor for parameter values in $P$ are "essentially hyperbolic". In other words, if Conjecture \ref{hyperbolicityconj} is true, then for all $v\in P$ the dynamics of the corresponding Lorenz attractor $A_v$ have the same qualitative and statistical properties of hyperbolic dynamical systems - even in the absence of any splitting condition on the tangent bundle of $A_v$. As such, a proof of Conjecture \ref{hyperbolicityconj} would amount to a proof of the Chaotic Hypothesis for the specific case of the Lorenz attractor (see \cite{gal}). That being said, due to the probable non-hyperbolic nature of the Lorenz attractor, we believe any proof of Conjecture \ref{hyperbolicityconj} will have to include a certain computational component.

\begin{figure}[h]
\centering
\begin{overpic}[width=0.3\textwidth]{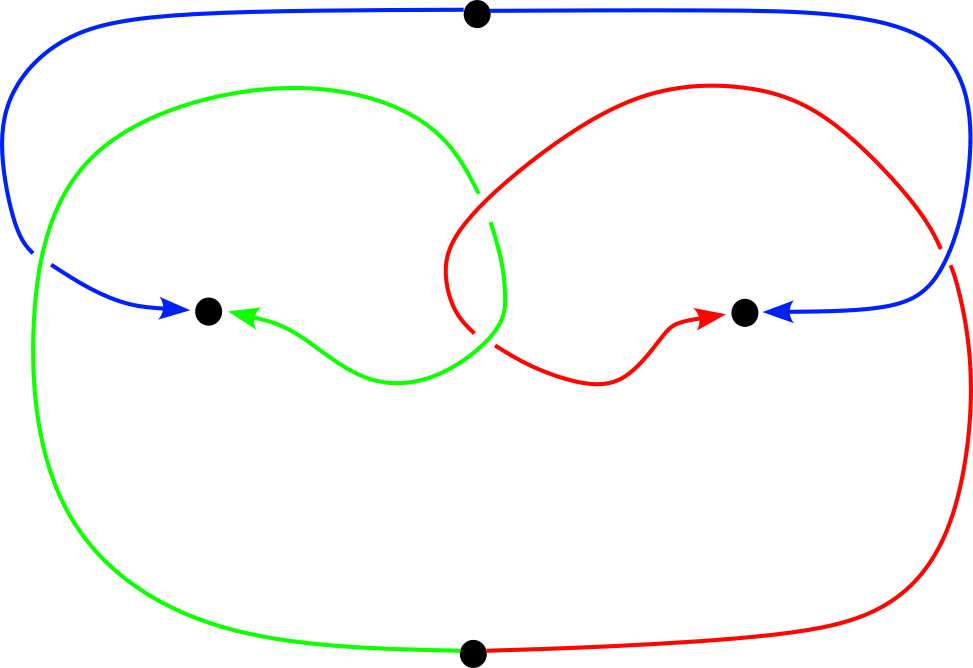}
\put(430,710){ $\infty$}
\put(450,-60){ $0$}
\put(180,410){$p_0$}
\put(720,410){$p_1$}
\end{overpic}
\caption{\textit{The configuration of the heteroclinic trefoil knot observed numerically in \cite{KOC}.}}
\label{fig8}
\end{figure}
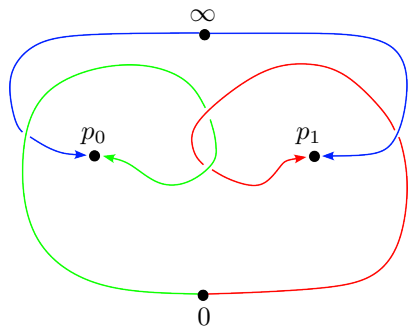

Another interesting question that arises from our results is how much they can be generalized. To clarify, recall our results were proven for parameters in $P$ that are sufficiently close to $T$ - where $T$ denotes the collection of parameters in $P$ in which the flow generates a heteroclinic trefoil knot. However, as observed numerically, there are also parameters where the Lorenz system generates heteroclinic knots whose topology is \textbf{not} that of a trefoil knot. This leads us to ask the following: let $p$ be a parameter where the corresponding Lorenz system generates a heteroclinic knot which is \textbf{not} a trefoil knot. Then, can we find a one (or more) parameter family of piecewise continuous maps describing the dynamics of its perturbations analogously to Th.\ref{reduct}?

\begin{figure}[h]
\centering
\begin{overpic}[width=0.3\textwidth]{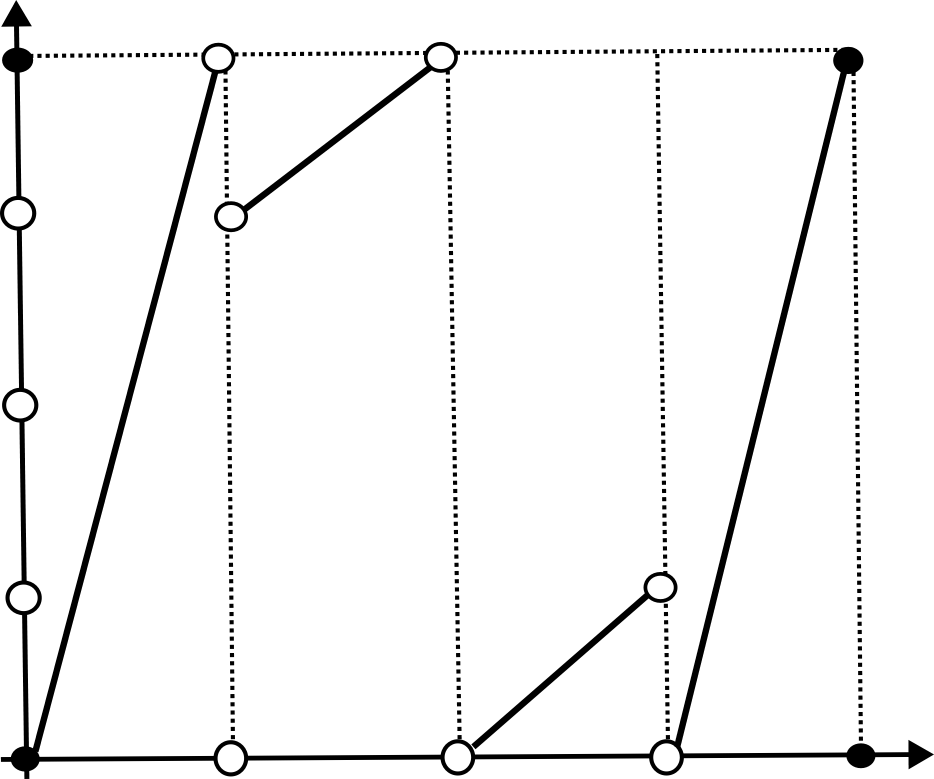}
\put(870,-50){ $1$}
\put(0,-50){ $0$}
\put(450,-50){ $c_2$}
\put(220,-50){$c_1$}
\put(720,-50){$c_3$}
\end{overpic}
\caption{\textit{A one-dimensional map derived from the geometric model for the figure $8$ knot, specifically, from the first-return map of the geometric model (compare with Fig.11 in \cite{YH}).}}
\label{alternative}
\end{figure}
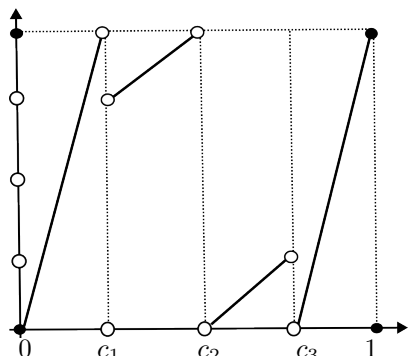

To motivate why this question is interesting, recall that as observed numerically in \cite{KOC}, there are parameters in $P$ where the Lorenz system generates a heteroclinic Figure $8$ Knot, sketched as in Fig.\ref{fig8}. Moreover, recall that, as proven in \cite{YH}, the geometric model for the Lorenz attractor at such scenarios would correspond to a shift on more than $2$ symbols - hence the $\beta$--transformations are no longer a realistic (simplified) model for the perturbed dynamics of such heteroclinic scenarios. In fact, based on the first-return map of the geometric model considered in \cite{YH} (see Fig.11 in \cite{YH}), it appears a realistic model for the perturbations of such parameters would be some "perturbation" of the piecewise-continuous map in Fig.\ref{alternative}. This leads us to conjecture the following:

\begin{conjecture}
    \label{onedim1} Let $\mathcal{R}$ denote the collection of all maps $\mathcal{F}:[0,1]\to[0,1]$ that serve as factor maps for the first-return maps $\psi_v:R\setminus W\to R$, of the Lorenz attractor (not necessarily only for parameters $v\in P$). Then, every map $\mathcal{F}:[0,1]\to[0,1]$ in $\mathcal{R}$ satisfies the following conditions.
    \begin{itemize}
        \item $\mathcal{F}$ is piecewise discontinuous. Specifically, there is a partition $0=c_0<c_1<\ldots<c_n<c_{n+1}=1$ of the interval $[0,1]$ s.t. $\mathcal{F}$ is an orientation-preserving affine map on $(c_k,c_{k+1})$, for $k=0,1,\ldots,n$, and has a jump discontinuity at each point $c_i$, for $i=1,\ldots,n$.
        \item Assume $\psi_v:R\setminus W\to R$ can be semi-conjugated to $\mathcal{F}$, and let $\mathcal{I}$ denote the invariant set of $\mathcal{F}$ in $[0,1]\setminus\{c_1,...,c_n\}$. Then, there exists $I_v$, an invariant set for the first-return map $\psi_v$, and a continuous surjection $\pi:I_v\to \mathcal{I}$ s.t. $\pi\circ\psi_v=\mathcal{F}\circ\pi$. 
        \item Moreover, if $x\in \mathcal{I}$ is periodic of minimal period $n$ for $\mathcal{F}$, then $\pi^{-1}(x)$ includes a periodic orbit for $\psi_v$ of minimal period $n$.
    \end{itemize}
\end{conjecture}

We believe any proof of Conjecture \ref{onedim1} would have to begin by first classifying which heteroclinic knots can be generated by the Lorenz system. As such, much like Conjecture \ref{hyperbolicityconj}, it is likely any proof of Conjecture \ref{onedim1} will have to include a certain computational component.

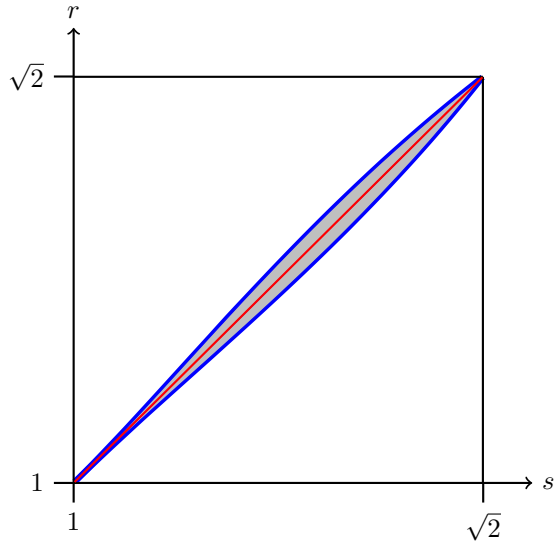
\begin{figure}[h]
\centering
\begin{tikzpicture}[scale=13]
	\draw[domain=0.001:0.414,smooth,variable=\x,blue,line width=2.2pt] plot ({\x},{1/(2*(\x+1))+1/2*((8*(\x+1)^2-9*(\x+1)+2)/((\x+1)^2*(2-(\x+1))))^(1/2)-1});
	\draw[domain=0.001:0.414,smooth,variable=\x,blue,line width=2.2pt] plot ({\x},{(2*(\x+1)^2+(\x+1)-2)/(2*(\x+1)^2)+1/2*((4*(\x+1)^4-4*(\x+1)^3+(\x+1)^2-4*(\x+1)+4)/((\x+1)^4))^(1/2)-1});
	\fill [lightgray, draw=blue, domain=0.001:0.413, variable=\x]
	(0.001, 0.001)
	-- plot ({\x}, {1/(2*(\x+1))+1/2*((8*(\x+1)^2-9*(\x+1)+2)/((\x+1)^2*(2-(\x+1))))^(1/2)-1})
	-- (0.413, 0.413)
	-- cycle;
	\fill [lightgray, draw=blue, domain=0.001:0.413, variable=\x]
	(0.001, 0.001)
	-- plot ({\x}, {(2*(\x+1)^2+(\x+1)-2)/(2*(\x+1)^2)+1/2*((4*(\x+1)^4-4*(\x+1)^3+(\x+1)^2-4*(\x+1)+4)/((\x+1)^4))^(1/2)-1})
	-- (0.413, 0.413)
	-- cycle;

	\draw[->, thick] (0,0) -- (0.464213,0) node[right]{$s$};
	\draw[->, thick] (0,0) -- (0,0.464213) node[above]{$r$};
	\draw[thick] (0,0.414513) -- (-0.02,0.414513) node[left]{$\sqrt{2}$};
	\draw[thick] (0.414513,0) -- (0.414513,-0.02) node[below]{$\sqrt{2}$};

	\draw[thick] (0,0) -- (-0.02,0) node[left]{1}; 
	\draw[thick] (0,0) -- (0,-0.02) node[below]{1}; 
	\draw[thick,domain=0:0.414213,variable=\x] plot ({\x},{0.414213});
	\draw[thick,domain=0:0.414213,variable=\x] plot ({0.414213},{\x});
	
	\draw[domain=0.001:0.414,variable=\x,red,thick] plot ({\x},{\x});
	\end{tikzpicture}	
\caption{\textit{Subset $(1,\sqrt{2}]\times(1,\sqrt{2}]$ of parameter space $(s,r)\in(1,2]\times(1,2]$. The lower and upper blue curves show the graph of $r=L(s)$ and $r=U(s)$, respectively. For parameters $(s,r)$ bounded by these curves the corresponding map $H_{s,r}$ is topologically conjugate to a $\beta$--transformation with a primary $2(1)$-cycle.}}
\label{fig:parameter_space_rs}
\end{figure}

Moving on, at this point we remark that many of our results easily generalize to $C^1$ perturbations of the heteroclinic trefoil scenario - even when they lie outside of the attractors corresponding to parameters in $P$. To illustrate, recall that while the Lorenz attractor is robust under general $C^1$ perturbations, its symmetry properties are not. Or, in other words, given a parameter $v\in P$ and its corresponding vector field $L_v$, we can $C^1$--approximate $L_v$ by non-symmetric vector fields. By Remark \ref{gen2} we know that in such cases, the one-dimensional factor map for the attractor given by Th.\ref{reduct} can still be defined, and takes the form:

 \begin{equation*}
f_{s,r}(x)=\begin{cases}
			rx, & \text{$x\in[0,\frac{1}{2})$}\\
            1-\frac{s}{2}+s(x-\frac{1}{2}), & \text{$x\in(\frac{1}{2},1]$},
		 \end{cases}
\end{equation*}
where $s,r\in(1,2]$ (see Fig.\ref{nosym}). As we discussed at the end of Sect.\ref{reductionsect} (see Remark \ref{rem:one-to-one}), the map $f_{s,r}$ can be reduced to an expanding Lorenz map $H_{s,r}\colon[0,1]\to[0,1]$ given by
\begin{equation}
\label{eq:non_symmetric}
H_{s,r}(x)=
\begin{cases}
			rx+\frac{(2-s)(r-1)}{r+s-2}, & \text{for}\ x\in[0,\frac{s-1}{r+s-2})\\
            sx-\frac{s(s-1)}{r+s-2}, & \text{for}\ x\in(\frac{s-1}{r+s-2},1]
		 \end{cases}
=
\begin{cases}
			rx+1-rc, & \text{for}\ x\in[0,c)\\
            s(x-c), & \text{for}\ x\in(c,1]
		 \end{cases},
\end{equation}
where $c=\frac{s-1}{r+s-2}$ denotes the critical point. Maps of the above form were studied by Cui and Ding in \cite{DiCui} (see also \cite{DC}). In particular, they proved such maps are topologically conjugate to $\beta$--transformations (see Main Theorem in \cite{DiCui}). Clearly, for $r=s$ we have  $H_{s,s}=F_s$, i.e., the map of the form \eqref{eq:non_symmetric} is a symmetric $\beta$--transformation. We will now show that for any $s\in(1,\sqrt{2}]$ and $r$ sufficiently close to $s$, the map $H_{s,r}$ is conjugate to a $\beta$--transformation with a primary $2(1)$-cycle. To be more precise, let us define maps
$$
L(s)=\frac{1}{2s}+\frac{1}{2}\sqrt{\frac{8s^2-9s+2}{s^2(2-s)}}\quad\text{and}\quad U(s)=\frac{2s^2+s-2}{2s^2}+\frac{1}{2}\sqrt{\frac{4s^4-4s^3+s^2-4s+4}{s^4}},
$$
for $s\in(1,\sqrt{2}]$ (see Fig.\ref{fig:parameter_space_rs}). Then we have the following:

\begin{proposition}
    \label{prop:non_symmetric}
    Let $H:=H_{s,r}$ be a map of form \eqref{eq:non_symmetric} defined by parameters
    $1<s\leq\sqrt{2}$ and $L(s)\leq r\leq U(s)$. Then the following conditions hold.
    \begin{enumerate}
        \item The map $H$ has a primary $2(1)$-cycle, so in particular it has a renormalization $G=(H^2,H^2)$.
        \item The map $H$ is topologically conjugate to the $\beta$-transformation $F_{\beta,\alpha}(x)=\beta x+\alpha\Mod{1}$ given by
        $$
        \beta=\sqrt{rs}\quad\text{and}\quad\alpha=\frac{1}{\beta+1}-\frac{\beta(\beta-1)(c-(1-rc))}{(\beta+1)(s(1-c)-(1-rc))}.
        $$
    \end{enumerate}
\end{proposition}
\begin{proof}
    We will prove that the points
    $$
    z_0:=\frac{s(c+rc-1)}{rs-1}\quad\text{and}\quad z_1:=\frac{rc(s+1)-1}{rs-1}
    $$
    forms a primary $2(1)$-cycle in three steps.
    \begin{itemize}
        \item First, we show that $0<z_0<c<z_1<1$. To this end, observe that for every $1<s\leq\sqrt{2}$ the inequalities
        $
        2-\frac{1}{s}<L(s)
        $
        and
        $
        U(s)<\frac{1}{2-s}
        $
        hold, as they reduce to $16(s-1)^3>0$ and $s^2(s-2)^2(s-1)>0$, respectively. In particular, we have
        $$
        1\leq2-\frac{1}{s}<L(s)\leq r \leq U(s)<\frac{1}{2-s}\leq1+\frac{\sqrt{2}}{2}.
       $$
       Since $c=\frac{s-1}{r+s-2}$, $1<s\leq\sqrt{2}$ and $1<r\leq2$, the inequalities $2-\frac{1}{s}<r<\frac{1}{2-s}$ imply $0<z_0<c<z_1<1$.

       \item Next, we show that $O=\{z_0,z_1\}$ forms a $2(1)$-cycle. Since $z_0<c<z_1$, simple calculations give
       $$
       H(z_0)=rz_0+1-rc=\frac{rs(c+rc-1)}{rs-1}+1-rc=\frac{rc(s+1)-1}{rs-1}=z_1
       $$
       and
       $$
       H(z_1)=s(z_1-c)=s\left(\frac{rc(s+1)-1}{rs-1}-c \right)=\frac{s(c+rc-1)}{rs-1}=z_0.
       $$

       \item It remains to prove that the $2(1)$-cycle $O$ is primary, that is, the inequalities $z_0\leq H(0)$ and $z_1\geq H(1)$ hold. Fix $1<s\leq\sqrt{2}$. Observe that the condition
       $$
       z_0=\frac{s(c+rc-1)}{rs-1}\leq1-rc=H(0),
       $$
       reduces to a quadratic inequality (w.r.t. $r$)
       $$
       (-s^2+2s)r^2+(s-2)r-2s+2\geq0,
       $$
       which is satisfied for $r\geq L(s)$. Similarly,
       $$
       z_1=\frac{rc(s+1)-1}{rs-1}\geq s(1-c)=H(1)
       $$
       leads to
       $$
       s^2r^2+(-2s^2-s+2)r+2s-2\leq0,
       $$
       which is true for $1<r\leq U(s)$.
    \end{itemize}
    We showed that the orbit $O=\{z_0,z_1\}$ is a primary $2(1)$-cycle of the map $H$. The rest of the proof follows from \cite[Theorem~3.5(1)]{ChO} and \cite[Corollary 3]{DiCui}.
    \end{proof}
We now give an example how these maps can be used to study general $C^1$ perturbations of the Lorenz attractor. To this end, recall $T$ was defined in earlier sections as the closure of the collection of all parameters in $P$ where the Lorenz system generates heteroclinic trefoil knot. Now, let $O$ denote the maximal $C^1$ neighborhood of the vector fields corresponding to parameters in $T$ s.t. every vector field $V\in O$ satisfies the following:
\begin{itemize}
    \item $V$ defines a cross-section $R$ as before which varies smoothly as we vary vector fields in $O$.
    \item The stable manifold of the origin always partitions $R$ into two sub-rectangles, $R_0\cup R_1$.
    \item There exists a continuous first return map $\psi_V:R_0\cup R_1\to R$, which can be reduced to a $\beta$--transformation $F_{\beta,\alpha}$, for some $(\beta,\alpha)$ in the parameter space $\Delta$ (see the illustration of the said parameter space in Fig \ref{fig:parameter_space}). 
\end{itemize}

The notion of a renormalizable parameter $v\in P$ in Def.\ref{renormvector} easily extends to vector fields $V\in O$. This implies the following result:
\begin{proposition}
    \label{condition1} Let $V\in O$ be a renormalizable vector field (in the sense of the generalized Def.\ref{renormvector}) and $F_{\beta,\alpha}$ denote its corresponding $\beta$--transformation with the critical point $c=\frac{1-\alpha}{\beta}$. Assume the map $F_{\beta,\alpha}$ has a renormalization $G=(F^l_{\beta,\alpha},F^r_{\beta,\alpha})$, defined on the renormalization interval $[u,v]$, that is conjugate to the doubling map $F_2(x)=2x\Mod{1}$. Suppose that in addition both $[u,c)\cap(\cup_{j=1}^{l-1}F^{-j}_{\beta,\alpha}(c))=\emptyset$ and $(c,v]\cap(\cup_{j=1}^{r-1}F^{-j}_{\beta,\alpha}(c))=\emptyset$. Then, the following holds:
 
    \begin{itemize}
        \item We can associate a Lorenz template $L(n,p)$ with the flow as in Th.\ref{templateth}, where $n,p\in \mathbb{Z}$ are both even.
        \item Every knot encoded by the template $L(n,p)$ is also be encoded by the $L(0,0)$ template.

    \end{itemize}
\end{proposition}
\begin{proof}
    The second assertion is immediate, and its proof is just a variation on the arguments used to prove Th.\ref{templateth} - where $l,r$ replace $n$, and $L(n,p)$ replaces $L(n,n)$. Similarly, the first assertion also follows from very similar arguments to those used to prove Th.\ref{templateth}, where the even nature of $n$ and $p$ follows from all the twists on the bands of the template being double - as forced by the motion of the flow (in other words, the endpoints of the intervals are not permuted).
\end{proof}
\begin{remark}
    Note that in the assertions above, we could replace the conditions $[u,c)\cap(\cup_{j=1}^{l-1} F^{-j}_{\beta,\alpha}(c))=\emptyset$ and $(c,v]\cap(\cup_{j=1}^{r-1}F^{-j}_{\beta,\alpha}(c))=\emptyset$ with $(u,c)\cap(\cup_{j=1}^{l-1} F^{-j}_{\beta,\alpha}(c))=\emptyset$ and $(c,v)\cap(\cup_{j=1}^{r-1}F^{-j}_{\beta,\alpha}(c))=\emptyset$. That being said, the second assumptions is non-generic, while the first one is.
\end{remark}

The assumption above on the pre-images of the critical point $c$ is essential. To elaborate, let $F:[0,1]\to[0,1]$ be a renormalizable expanding Lorenz map with the critical point $c$. Suppose that $F$ has a renormalization $G=(F^l,F^r)$ defined on the renormalization interval $[u,v]$ s.t. either $[u,c)\cap F^{-j}(c)\ne\emptyset$, $(c,v]\cap F^{-i}(c)\ne\emptyset$ (for some $1\leq j<l$, $1\leq i< r$), or both. From now on, we refer to this condition as the \textbf{cut and paste condition}, as it can be visualized as if, say, the interval $[u,c)$ (or alternatively, $(c,v]$) is cut in two after $j$ (or $i$) iterations and then glued together again by $F^l$ (or alternatively, $F^r$). In Appendix \ref{ex:map_nonsymmetric1}, we will present an example showing how such a scenario could occur, and discuss this condition in greater detail. We now consider a more general version of this condition. Assume there exists a sub-interval $(a,b)$ and some $0\leq i<j$ s.t. the following is satisfied:
    \begin{itemize}
        \item $c\in F^i(a,b)$.
        \item $F^j$ is continuous on $(a,b)$.
    \end{itemize}
We refer to this condition as the \textbf{general cut and paste condition}. Motivated by the above, we now ask the following question - are there vector fields $V\in O$ whose first-return map can be reduced to $F$ satisfying the cut and paste condition on some interval $(a,b)\subseteq[0,1]$? We believe a scenario where the action of a smooth flow "cuts and pastes" some sub-rectangle $R'\subseteq R$ can occur only when there exists a pair of homoclinic trajectories to the origin. To see why, note that in the absence of such a homoclinic pair, the Existence and Uniqueness Theorem tells us that after $R'$ is torn in two by the flow at $W$, the two halves must be kept apart by the flow. This discussion motivates us to conjecture the following:

\begin{conjecture} \label{condition2}

Let $\mathcal{L}$ denote the collection of Lorenz maps that form factor maps for vector fields $V\in O$, and let $\mathcal{CP}$ denote the collection of Lorenz maps that \textbf{do not} satisfy the general cut and paste condition. Then, the following holds:
\begin{itemize}
    \item There exists a topology s.t. the set $\mathcal{CP}$ is generic in $\mathcal{L}$ - i.e., $\mathcal{CP}$ is both open and dense w.r.t. some "meaningful" topology on $\mathcal{L}$ (see the discussion below).
    \item Th.\ref{reduct} defines a surjective map $\mathcal{F}:O\to\mathcal{L}$ matching each vector field $V$ with its corresponding Lorenz map such that $f$. 
    \item The map $\mathcal{F}$ is continuous w.r.t. the $C^1$ topology on $O$ and the topology  on $\mathcal{L}$ is the one from above.
    \item $\mathcal{F}^{-1}(\mathcal{L}\setminus\mathcal{CP})$ is precisely the collection of vector fields in $O$ generating a pair of homoclinic trajectories to the origin.
\end{itemize}
\end{conjecture}

At present, we do not know how to prove Conjecture \ref{condition2}. It is easy to see that for every $V\in O$ there is a corresponding map $f\in\mathcal{L}$, and that by definition this map is onto. That being said, it is far from obvious we can find a meaningful topology on $\mathcal{L}$ - i.e., some variation on the $C^1$ topology, suited for piecewise discontinuous maps that would make this correspondence continuous. 

\section*{Declarations:}
\subsection*{Acknowledgments:}
The authors would like to thank Tali Pinsky, Jorge Olivares-Vinales, Zin Arai, Valerii Sopin, Asad Ullah, Yanghong Yu and Noy Soffer-Aranov for their suggestions and helpful discussions.
\subsection*{Ethics and author contribution statement}
 Both authors made an equal contribution to this study, and both comply with the COPE regulations.
\subsection*{Competing interests statement}
On behalf of all authors, the corresponding author states that there are no conflicts of interest to disclose. Moreover, 
the authors have no financial or proprietary interests in any material discussed in this article.
\subsection*{Data availability statement}
No data was generated or used in this study.
\section{Appendix - more on the cut and paste condition}
\label{ex:map_nonsymmetric1}
In \cite{OPR}, Oprocha, Potorski, and Raith showed that there exists an expanding Lorenz map that is locally eventually onto (in short: l.e.o.), but not strongly locally eventually onto (see Def.1.1 and 1.2 in \cite{OPR}). It is easy to see that such a map must satisfy the general cut and paste condition. We also observe that this condition holds for each map $F_{\varepsilon_i}$ defined by parameters \eqref{eq:parameter_points_eps} with $i\geq2$. In fact, all these maps have a property called \textbf{matching} (see \cite{Bruin1,Bruin2}). Recall that a Lorenz map $F$ with a critical point $c$ has matching, if $F^n(c_-)=F^n(c_+)$ for some $n\in\mathbb{N}$.

\begin{proposition}\label{prop:cutandpaste}
    Let $F\colon[0,1]\to[0,1]$ be an expanding Lorenz map with a critical point $c$. The following conditions hold.
    \begin{enumerate}
        \item If $F$ is l.e.o. but not strongly l.e.o., then it satisfies the general cut and paste condition.
        \item $F$ satisfies the general cut and paste condition if and only if $F$ has matching.
    \end{enumerate}
\end{proposition}
\begin{proof}
    The first assertion follows immediately from the definitions of l.e.o. and strongly l.e.o. maps. To prove the second one, assume that $F$ satisfies the general cut and paste condition on an interval $(a,b)$. Then, since the interval $(a,b)$ is cut into two sub-intervals after the $i$th iteration and then glued again by $F^j$ (for some $0\leq i<j$), we must have $F^n(c_-)=F^n(c_+)$, where $n=j-i$.

    Now, assume that $F$ has matching and let $n\in\mathbb{N}$ be the minimal number satisfying $F^n(c_-)=F^n(c_+)$. Observe that we can choose a neighborhood $(a,b)$ of the critical point $c$ s.t. $(a,b)\cap\bigcup_{k=1}^{n-1}F^{-k}(c)=\emptyset$. Then the map $F^n$ is continuous on $(a,b)$, so $F$ satisfies the general cut and paste condition.
\end{proof}

Consider the map mentioned above from \cite[Example~5.1]{OPR} (see also \cite[Example~3.1]{ChO}), that is, a $\beta$-transformation $F(x):=F_{\beta,\alpha}(x)=\beta x+\alpha\Mod{1}$ defined by
$$
\beta\approx 1.2207440846\quad\text{and}\quad\alpha=1-\frac{1}{\beta}\approx 0.1808274865,
$$
(see Fig.~\ref{fig:examp_nonsymetr} and Fig.~\ref{fig:parameter_space}). Note that $c=\frac{1}{\beta^2}\approx0.67104$. As shown in \cite{OPR}, the map $F$ is l.e.o. but not strongly l.e.o. Moreover, we have $F^{12}(c_-)=F^{12}(c_+)$, so $F$ has matching. By Prop.\ref{prop:cutandpaste}, the map $F$ satisfies the general cut and paste condition.
  \begin{figure}[h]
		\centering
		\begin{tikzpicture}[scale=5]
		\draw[thick] (0,0) -- (1,0);
		\draw[thick] (0,0) -- (0,1);
		\draw (0,1) node[left]{1};
		\draw (1,0) node[below]{1};
		\draw (-0.02,0) node[below]{0}; 
		\draw[thick,domain=0:1,variable=\y] plot ({1},{\y});
		\draw[thick,domain=0:1,variable=\x] plot ({\x},{1});
	
		\draw[domain=0:1,dashed,variable=\y,orange,very thin] plot ({0.18082},{\y});
		\draw[domain=0:1,dashed,variable=\y,orange,very thin] plot ({0.40157},{\y});
		\draw[domain=0:1,dashed,variable=\y,orange,very thin] plot ({0.671044},{\y});

		\draw[domain=0:1,dashed,variable=\x,orange,very thin] plot ({\x},{0.18082});
		\draw[domain=0:1,dashed,variable=\x,orange,very thin] plot ({\x},{0.40157});
		\draw[domain=0:1,dashed,variable=\x,orange,very thin] plot ({\x},{0.671044});

		\draw[domain=0:1,smooth,variable=\x,red,very thick] plot ({\x},{\x});
	
		\draw[domain=0:0.67104,smooth,variable=\x,blue,very thick] plot ({\x},{1.22074*\x+0.18082});
	
		\draw[domain=0.67104:1,smooth,variable=\x,blue,very thick] plot ({\x},{1.22074*\x+0.18082-1});
	
		\filldraw [black] (0.18082,0) circle (0.2pt) node[anchor=north] {$F(0)$};
		\filldraw [black] (0,0.18082) circle (0.2pt) node[anchor=east] {$F(0)$};
		\filldraw [black] (0.671044,0) circle (0.2pt) node[anchor=north] {$F^3(0)$};
		\filldraw [black] (0,0.671044) circle (0.2pt) node[anchor=east] {$F^3(0)$};
		\filldraw [red] (0.67104,0) circle (0.2pt);
		\filldraw [black] (0.40157,0) circle (0.2pt) node[anchor=north] {$F^2(0)$};
		\filldraw [black] (0,0.40157) circle (0.2pt) node[anchor=east] {$F^2(0)$};
		\end{tikzpicture}	
		\caption{\textit{Graph of the $\beta$-transformation $F$ from Appendix ~\ref{ex:map_nonsymmetric1}.}}
		\label{fig:examp_nonsymetr}
	\end{figure}
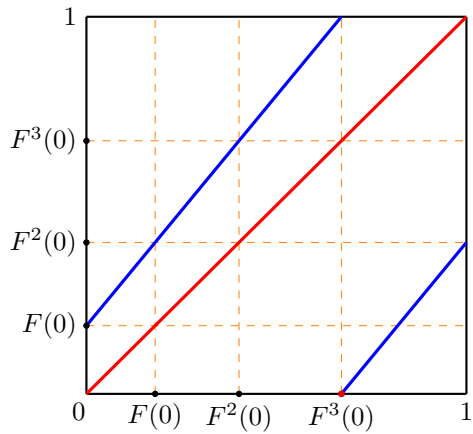

We also note that the map $F$ shows a slight discrepancy between the analytic and symbolic definitions of renormalization. To observe it, denote $[u,v]:=[F^2(0),1]$. Then
$$
	(u,c)\xrightarrow{\ F\ }(c,v)\xrightarrow{\ F\ }(0,u)\xrightarrow{\ F\ }(F(0),c)\xrightarrow{\ F\ }(u,v)
	\quad\text{and}\quad
	(c,v)\xrightarrow{\ F\ }(0,u)\xrightarrow{\ F\ }(F(0),c)\xrightarrow{\ F\ }(u,v).
	$$
In particular, the map $G\colon[u,v]\to[u,v]$ defined as
$$
G(x)=
	\begin{cases}
F^4(u_{+}), &\text{for}\ x=u\\
	F^4(x), &\text{for}\ x\in(u,c)\\
	F^3(x), &\text{for}\ x\in (c,v)\\
	F^3(v_{-}), &\text{for}\ x=v
\end{cases},
$$
is conjugate to the doubling map. However, since $F(c)=0$ implies $F^3(v_-)\neq F^3(v)=0$ (similarly, $F(c)=1$ would imply $F^4(u)=c$), the map $G$ is not a renormalization in the sense of Def.\ref{defn:renormalization}. On the other had, the symbolic definition of renormalization only looks at the one-sided limits at the critical point $c$: since the kneading invariant of $F$ has form $k_F=((1000)^\infty,(010)^\infty)$, it is renormalizable by words of lengths $4$ and $3$ (see details in \cite{ChO}). 

The occurrence of the general cut and paste condition in the case of a renormalizable (at least symbolically) map raises the following question: is there a scenario in which the interior of $[u,c)$ or $(c,v]$ covers the critical point and is torn into two nondegenerate intervals, which are then glued together? In other words, is there a renormalizable map $F$ (in the sense of Def.\ref{defn:renormalization}) with renormalization $G\colon[u,v]\to[u,v]$ for which the cut and paste condition is satisfied on the subinterval $(u,c)$ or $(u,c)$? If so, the discrepancy between the analytic and symbolic definitions of renormalization may be significant, since the renormalizable form of the kneading invariant $k_F$ could be destroyed while the Def.\ref{defn:renormalization} remains intact.

\section{Appendix - Homoclinic bifurcations as compressed Period Doubling Bifurcations}
\label{compressing}

As proven in Th.\ref{expl}, the periodic dynamics of the Lorenz flow given by Th.\ref{reduct} and Cor.\ref{reduct2} can only be destroyed by a homoclinic bifurcation. In detail, the said Theorem proves that given such a periodic orbit $\tau$ for the Lorenz system corresponding to some $v\in P$, as we vary $v$ in $P$ the orbit $\tau$ can only be destroyed by colliding with the stable manifold of the origin $0$ - thus becoming a homoclinic or a pair of homoclinic trajectories for $0$. In this section we are going to show that heuristically, such homoclinic bifurcations can be thought of as "compressed period-doubling" cascades. As will be made clear, this is an intuitive heuristic and not a formal proof - therefore, we will do so in the form of discussion, as outlined below. \\

To begin, recall the map $h_v:ABCD\to\mathbb{R}^2$ introduced in the proof of Th.\ref{reduct}. As shown in the proof of Th.\ref{reduct}, given any parameter $v\in P$ the following holds:
\begin{itemize}
    \item The rectangle $ABCD$ can be decomposed to $ABCD=R_0\cup R_2\cup R_1$, where $R_2$ is an inner rectangle separating $R_0$ and $R_1$ (see Fig.\ref{deformation3} and Fig.\ref{isot1}).
    \item The first-return map for the flow $\psi_v:R_0\cup R_1\to R$ can be homotoped to $h_v$ on $R_0$ and $R_1$ in a way s.t. as we deform $h_v$ back to $\psi_v$, all the periodic orbits for $h_v$ persist - without changing their minimal period.
    \item Let $f_v$ denote the Lorenz-like map given by Th.\ref{reduct} (i.e., the map $f_r$ in the notation of Th.\ref{reduct}). Then, we can homotope $h_v$ to $f_v$ s.t. the periodic orbits of $h_v$ are deformed continuously to the periodic orbits of $f_v$ - without changing their minimal period. Moreover, this deformation defines a bijection from the periodic orbits of $h_v$ and those of $f_v$ (see the illustration in Fig.\ref{model}).
\end{itemize}

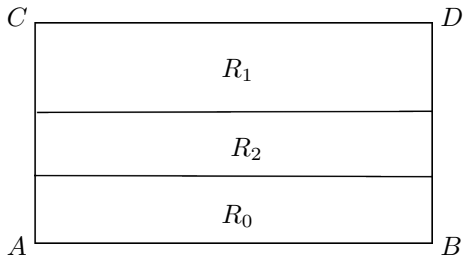
\begin{figure}[h]
\centering
\begin{overpic}[width=0.3\textwidth]{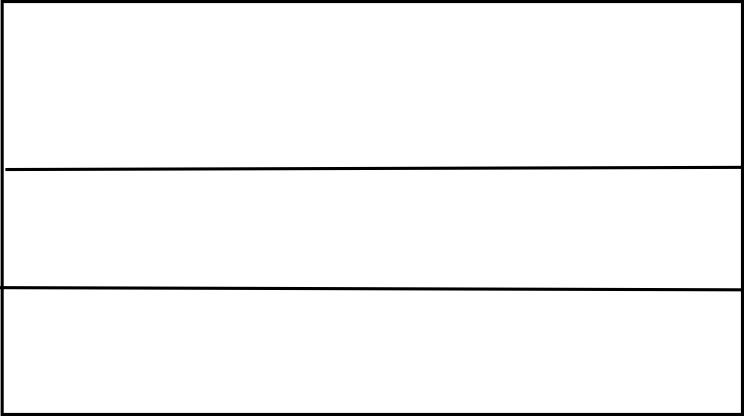}
\put(470,50){$R_0$}
\put(470,420){$R_1$}
\put(470,220){ $R_2$}
\put(-70,-30){$A$}
\put(-70,550){$C$}
\put(1020,-30){$B$}
\put(1020,550){$D$}
\end{overpic}
\caption{\textit{Partitioning $R$ to $R_0,R_1$ and $R_2$.}}
\label{isot1}
\end{figure}

Now, consider a smooth, orientation-preserving isotopy $g_t:ABCD\to\mathbb{R}^2$ as in Fig.\ref{isot2}, which bends the rectangle $R_2$ as described in Fig.\ref{isot2}. Moreover, we choose this isotopy s.t. $g_\frac{1}{2}=h_v$ and $g_1=h_p$, where $p\in T$ is some parameter. Moreover, we choose $g_0:ABCD\to\mathbb{R}^2$ as described in the uppermost left part of Fig.\ref{isot2}. By Th.\ref{reduct} and Th.\ref{tali} (and the definition of the isotopy above) we know $g_1$ into a horseshoe on $3$ symbols. It is easy to see that as we vary $v\to p$ in the parameter space we can define such an isotopy, induced by the change of the first-return maps $\psi_v$ as we vary $v$ towards $p$. Therefore, we now adopt the assumption that as we vary $v$ towards $p$, we also vary the parameter $t$ from $\frac{1}{2}$ towards $1$. 

In other words, we assume that for every $t\geq\frac{1}{2}$ the map $g_t$ is $h_{v'}$, for some $v'\in P$, and that as we vary $\psi_v$ smoothly to $\psi_p$ through the parameter space, we induce a smooth variation of $g_t$ along the isotopy from $t=\frac{1}{2}$ to $1$. At this point we remark it is precisely such assumptions which is why our arguments are more of a heuristic, and not a formal proof. Namely, it is far from immediate that we can induce this smooth transition, if only because the deformation from $\psi_{v'}$ to $g_t=h_{v'}$, $t\geq\frac{1}{2}$, given by Th.\ref{reduct} is possibly only continuous (among other reasons).

\begin{figure}[h]
\centering
\begin{overpic}[width=0.4\textwidth]{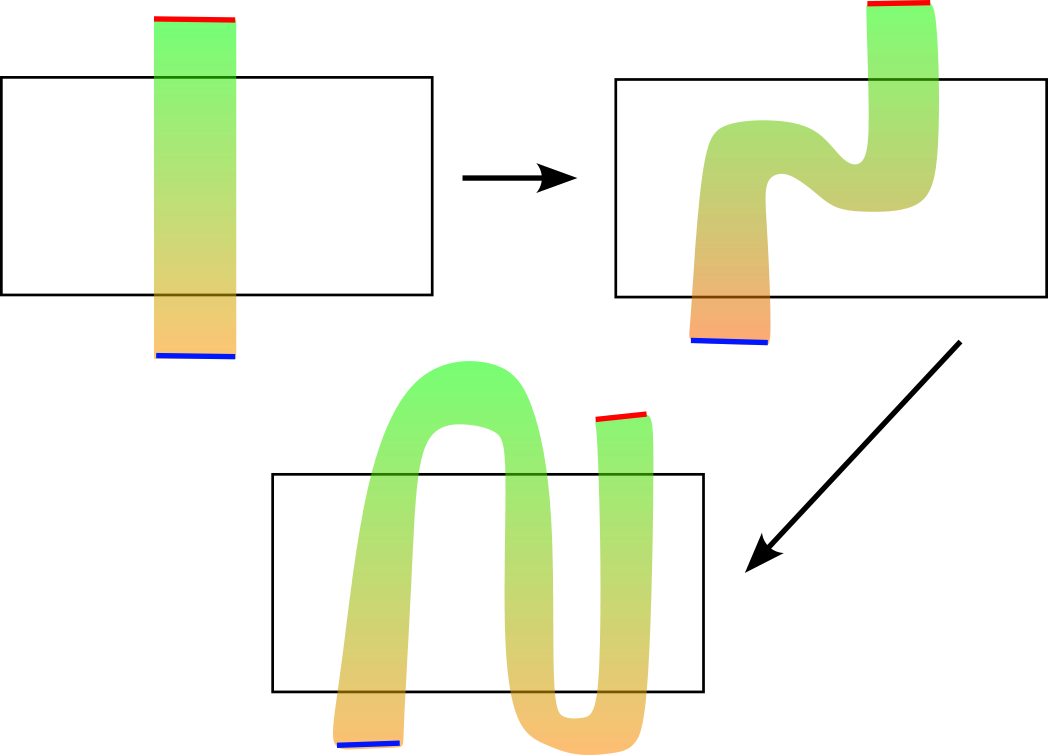}

\end{overpic}
\caption{\textit{The isotopy $g_t:ABCD\to\mathbb{R}^2$. The red and blue arcs denote the images of the sides $CD$ and $AB$, respectively. $g_0$ is illustrated in the image on the uppermost left.}}
\label{isot2}
\end{figure}

Back to our heuristic, as proven in \cite{KY3}, as we vary $t$ from $0$ to $1$ the periodic orbits are added via period-doubling cascades (see section $4$ in \cite{KY3}). To elaborate further, we first need to recall some terminology. To this end, consider periodic point $x$ of minimal period $k$ for $g_t$ (for some $t\in(0,1)$), and let $\lambda_1$ and $\lambda_2$ denote the multiplier of the differential $D_{g_t^k}(x)$. We recall $x$ is called:
\begin{itemize}
    \item \textbf{Hyperbolic} if $\lambda_1\in(1,\infty)$, $\lambda_2\in(0,1)$.
    \item \textbf{Elliptic} if $0<|\lambda_1|,|\lambda_2|<1$ or $|\lambda_1|,|\lambda_2|>1$.
    \item \textbf{Möbius} if $\lambda_1\in(-\infty,-1)$, $\lambda_2\in(-1,0)$.
\end{itemize}

Note that since our maps are orientation-preserving, either $\lambda_1,\lambda_2$ are jointly positive, or they are jointly negative. As proven in \cite{KY3}, generically, the periodic dynamics of the isotopy $g_t:ABCD\to\mathbb{R}^2$ all arise by saddle-node bifurcations, where a hyperbolic and elliptic orbits are created. Following that, the elliptic orbit undergoes a period-doubling cascade where the Möbius orbits are created as a result of the period-doubling bifurcations (see the illustration in Fig.\ref{cascade}). Moreover, once a saddle periodic orbit is created (i.e., Möbius or hyperbolic), whether at a saddle-node or period-doubling bifurcation, it persists as we vary $t$ up to $g_1$ (see the illustration in Fig.\ref{cascade}). In particular, these bifurcations all occur in the parameter range $t\in(0,1)$.

By the proof of Th.\ref{reduct} it is easy to see the periodic orbits of the first-return map $\psi_v$, $v\in P$, that survive as we continuously deform it to $h_v=g_{\frac{1}{2}}$ are precisely the periodic orbits for $h_v$ which visit only $R_0$ and $R_1$. We now recall $h_v$ stretches both $R_0,R_1$ while keeping the respective arcs $AB,CD$ in place (see the illustrations in Fig.\ref{deformation3} and Fig.\ref{isot2}). It is easy to see we can assume this stretching to be linear, which implies all the periodic orbits for $h_v$ which never visit the sub-rectangle $R_2$ are hyperbolic saddles. Moreover, it is also easy to see this property remains true as we let $t$ vary from $\frac{1}{2}$ to $1$ - or, in other words, the periodic orbits for $g_t,t\geq\frac{1}{2}$ that never visit $R_2$ are always hyperbolic. Combined with the results of \cite{KY3} surveyed above, this teaches us that the progression from order into chaos for the isotopy $g_t:ABCD\to\mathbb{R}^2$ can be described generically as follows:

\begin{itemize}
    \item Every periodic orbit is generated via either a period-doubling or a saddle node bifurcation which occurs in the set $V_t=\cup_{n\geq0}g_t^{-n}(R_2)$, for some $t\in(0,1)$.
    \item After their creation in $V_t$ (for some $t$), the hyperbolic periodic orbits migrate into $R_0\cup R_1$ as $t\to 1$, while the Möbius orbits and the attracting orbits all remain strictly inside $V_t$, for all $t$. 
\end{itemize}

\begin{figure}[h]
\centering
\begin{overpic}[width=0.4\textwidth]{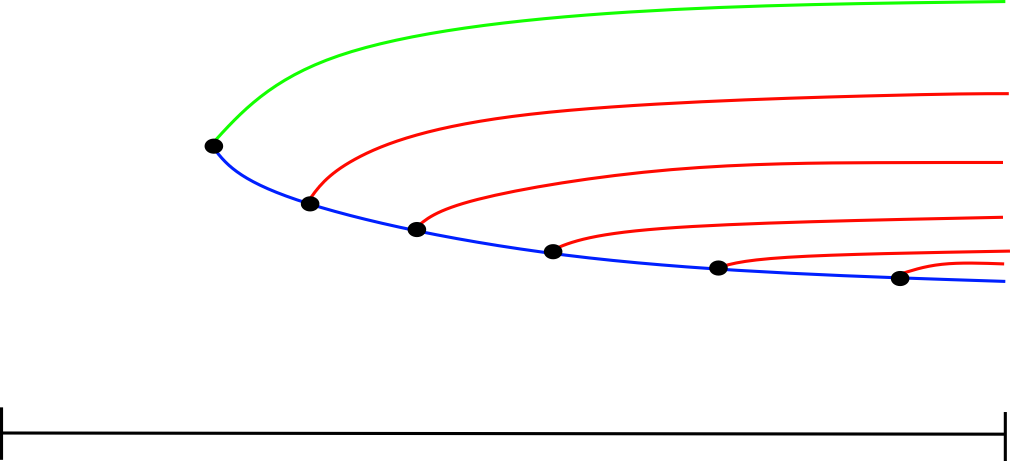}
\put(-40,10){$0$}
\put(1010,10){$1$}

\end{overpic}
\caption{\textit{The (partial) bifurcation diagram of a (generic) period doubling cascade occurring along the isotopy $g_t:ABCD\to\mathbb{R}^2$, $t\in[0,1]$. The green arc denotes the hyperbolic orbit, the red arcs the Möbius orbits, and the blue arcs the elliptic orbit which undergoes a period-doubling cascade. }}
\label{cascade}
\end{figure}

Having completed this analysis, we now return to the Lorenz system, and to the Lorenz-like maps. To this end, given a periodic orbit $\alpha$ for $g_t$, $t\geq\frac{1}{2}$, $t_\alpha$, the \textbf{entrance parameter}, would be the first parameter s.t. for all $g_t$, $t>t_\alpha$, $\alpha\subseteq R_1\cup R_2$. We now recall that as we pass back from the smooth isotopy $g_t:R_0\cup R_1\to\mathbb{R}^2$, $t\geq\frac{1}{2}$ back to the first-return maps of the flow, the rectangle $R_2$ closes to $W$ - where $W$ is the main line of intersection between $W^s(0)$ and the cross-section $R$.  By the definition of $g_t,t\geq\frac{1}{2}$, one could interpret the first-return map $\psi_d$ corresponding to $g_{t_\alpha}$, $d\in P$, as where the periodic orbit $\alpha$ first appears in the Lorenz attractor. And since by the correspondence with $g_{t_\alpha}$ it occurs via collision with $R_2$, it follows that at $d$ the corresponding Lorenz system has to undergo a homoclinic bifurcation. This shows one can think of such homoclinic bifurcations as "compressed" period-doubling bifurcations for $g_t$, $t\geq\frac{1}{2}$. Therefore, from this discussion we derive the following heuristic statement - \textbf{the homoclinic bifurcations for the Lorenz system are essentially compressed period-doubling cascades.}

\printbibliography

\end{document}